\documentclass[a4paper]{amsart}
        \usepackage{latexsym}
        \usepackage{amssymb}
        \usepackage{amsmath}
        \usepackage{amsfonts}
        \usepackage{amsthm}
        \usepackage[hypertexnames=false]{hyperref}
        \ifpdf
          \usepackage[all,knot,poly,2cell]{xy}
        \else
          \usepackage[all,knot,poly,2cell,dvips]{xy}
        \fi
             \UseAllTwocells
        \usepackage{xspace}

        \usepackage{eucal}

        \theoremstyle{plain}
        \newtheorem{theorem}{Theorem}[section]
        \newtheorem{corollary}[theorem]{Corollary}
        \newtheorem{lemma}[theorem]{Lemma}
        \newtheorem{question}[theorem]{Question}
        \newtheorem{proposition}[theorem]{Proposition}
        \newtheorem{maintheorem}{Theorem}

        \theoremstyle{definition}
        \newtheorem{definition}[theorem]{Definition}
        \newtheorem{example}[theorem]{Example}
        
        \newtheorem*{example*}{Example}

        \theoremstyle{remark}
        \newtheorem{remark}[theorem]{Remark}
        \newtheorem*{remark*}{Remark}  
        \newcommand{\suchthat}{\,:\,}
        \newcommand{\itemref}[1]{\eqref{#1}}

        \newcommand{\mathscript}{\mathcal}

        \newcommand{\Z}{\mathbb{Z}}

        \newcommand{\Q}{\mathbb{Q}}

        \newcommand{\Orb}{\mathcal{O}}   
       \DeclareMathOperator{\spec}{Spec} 
        \newcommand{\red}[1]{{#1}_{\mathrm{red}}} 
           \newcommand{\et}{\mathrm{\acute{e}t}}

           \newcommand{\lisset}{\mathrm{lis\text{\nobreakdash-}\acute{e}t}}

        \newcommand{\MOD}{\mathsf{Mod}}    

        \DeclareMathOperator{\coker}{coker}
         \DeclareMathOperator{\Hom}{Hom}
        \DeclareMathOperator{\Ext}{Ext}
        \DeclareMathOperator{\Tor}{Tor}
        \DeclareMathOperator{\supp}{supp}
        \newcommand{\COHO}[1]{\mathcal{H}^{{#1}}}
        \newcommand{\trunc}[1]{\tau^{{#1}}}
        \newcommand{\RDERF}{\mathsf{R}}
        \newcommand{\LDERF}{\mathsf{L}}
        \newcommand{\STor}{\mathcal{T}or}
        \newcommand{\DCAT}{\mathsf{D}}
        \newcommand{\RHom}{\RDERF\!\Hom}
        \newcommand{\SHom}{\mathcal{H}om}
        \newcommand{\SRHom}{\RDERF\SHom}
        \DeclareMathOperator{\supph}{supph} 
        \newcommand{\QCOH}{\mathsf{QCoh}}
        \newcommand{\COH}{\mathsf{Coh}}

        \newcommand{\KGRP}{\mathsf{K}_0}
         
        \renewcommand{\bar}[1]{\overline{{#1}}}

        \newcommand{\ID}[1]{\mathrm{Id}_{#1}}

        \newcommand{\tensor}{\otimes}
        \newcommand{\ltensor}{\overset{\LDERF}{\otimes}}
        \newcommand{\homotopic}{\simeq}

        \newcommand{\opp}{\circ}
        
                \newcommand{\AB}{\mathsf{Ab}}

\newcommand{\fndefn}[1]{\emph{{#1}}}

\renewcommand{\subset}{\subseteq}
\numberwithin{equation}{section}

\newcommand{\TRICAT}{\mathbf{TCat}}

\newcommand{\qcsubscript}{\mathrm{qc}} 

\newcommand{\DQCOH}[1][]{\DCAT_{\qcsubscript{#1}}} 

\newcommand{\QCHR}[1][]{\mathcal{Q}_{{#1}}} 

\newcommand{\QCPSH}[1]{(#1_{\qcsubscript})_*} 
\newcommand{\MODPSH}[1]{(#1_{\lisset})_*} 
\newcommand{\QCPBK}[1]{#1^*_{\qcsubscript}}

\newcommand{\GL}{\mathrm{GL}} 
\newcommand{\Gm}{\mathbb{G}_m} 
\newcommand{\Ga}{\mathbb{G}_a} 

\newcommand{\REST}[1][]{\mathrm{res}_{{#1}}} 
\newcommand{\shv}[1]{\mathcal{#1}} 
\newcommand{\cplx}[1]{\shv{#1}} 
\newcommand{\VB}[1]{\mathbf{VB}({#1})} 

\newcommand{\spref}[1]{\href{http://stacks.math.columbia.edu/tag/#1}{#1}}

\makeatletter
\newcommand{\labitem}[2]{%
\def\@itemlabel{(\textbf{#1})}
\item
\def\@currentlabel{\textbf{#1}}\label{#2}}
\makeatother

\newcommand{\REP}[2][]{\mathbf{Rep}^{{#1}}/{#2}}
\usepackage{mathtools}
\newcommand{\crisp}{crisp}%
\newcommand{\Crisp}{Crisp}%
\newcommand{\crispness}{crispness}%
\DeclareMathOperator{\Br}{Br}
\newcommand{\cms}{\mathrm{cms}}

\newcommand{\cartsubscript}{\mathrm{cart}} 
\newcommand{\DCART}[1][]{\DCAT_{\cartsubscript{#1}}}

\newcommand{\conssubscript}{\mathrm{c}} 
\newcommand{\DCONS}[1][]{\DCAT_{\conssubscript{#1}}}

\DeclareMathOperator{\kar}{char}

\newcommand{\DCATinfty}{\mathcal{D}}
\newcommand{\DQCOHinfty}[1][]{\DCATinfty_{\qcsubscript{#1}}} 
\newcommand{\MODinfty}{\mathcal{M}\mathsf{od}}
\newcommand{\QCOHinfty}{\mathcal{QC}\mathsf{oh}}
\newcommand{\sX}{\mathfrak{X}}

\theoremstyle{theorem}
\newtheorem*{corollary*}{Corollary}
\theoremstyle{definition}
\newtheorem*{definition*}{Definition}

\title{Perfect complexes on algebraic stacks}
\date{May 24, 2017}
\author[J. Hall]{Jack Hall}
\address{Department of Mathematics\\University of Arizona\\Tucson, AZ 85721-0089\\USA}
\email{jackhall@math.arizona.edu}
\author[D. Rydh]{David Rydh}
\address{KTH Royal Institute of Technology\\Department of Mathematics\\SE\nobreakdash-100\ 44\ Stockholm\\Sweden}
\email{dary@math.kth.se}
\thanks{This collaboration was supported by the G\"oran Gustafsson foundation.
The first author was supported by the Australian Research Council DE150101799.
The second author is supported by the Swedish Research Council
2011-5599.}


\begin{document}
\begin{abstract}
  We develop a theory of unbounded derived categories of quasi-coherent sheaves
on algebraic stacks. In particular, we show that these categories are compactly
generated by perfect complexes for stacks that either have finite stabilizers
or are local quotient stacks. We also extend To\"en and Antieau--Gepner's
results on derived Azumaya algebras and compact generation of sheaves on linear
categories from derived schemes to derived Deligne--Mumford stacks. These are
all consequences of our main theorem: compact generation of a presheaf of
triangulated categories on an algebraic stack is local for the quasi-finite
flat topology.

\end{abstract}
\maketitle
\section*{Introduction}
Our first main result is the following.
\begin{maintheorem}\label{thm:main_qf} 
  Let $X$ be a quasi-compact algebraic stack with quasi-finite and
  separated diagonal. Then the unbounded derived category $\DQCOH(X)$, of
  $\Orb_X$-modules with quasi-coherent cohomology, is
  compactly generated by a single perfect complex. Moreover, for every quasi-compact open subset $U \subseteq X$, there exists a compact perfect complex with support exactly $X\setminus U$. 
\end{maintheorem}
This generalizes the results of A.~Bondal and M.~Van den Bergh
\cite[Thm.~3.1.1]{MR1996800} for schemes and B. To\"en
\cite[Cor.~5.2]{MR2957304} for Deligne--Mumford
stacks admitting coarse moduli spaces (i.e.,~$X$ has finite inertia).
While we can show that every compact object of $\DQCOH(X)$ is a perfect complex (Lemma \ref{L:compact_conc}), a subtlety in
Theorem \ref{thm:main_qf} is that the converse does not always hold in
positive characteristic. Indeed, 
if $X$ is not tame, then the structure sheaf is perfect but not compact. 

Theorem \ref{thm:main_qf}, as well as the theory developed to establish it, have been used to classify 
the thick tensor ideals in $\DQCOH(X)^c$ \cite{hall_balmer_cms} (generalizing work of \cite{MR2570954,MR2927050}), to resolve the telescope conjecture for 
algebraic stacks \cite{telescope-stacks} (extending \cite{MR3161100}), and for results on dg-enhancements \cite{canonaco_stellari_dg_grothendieck,bergh_lunts_schnurer_geometricity}.

Extending Theorem \ref{thm:main_qf} to certain stacks with infinite stabilizer groups is our second main result. We briefly recall a notion from \cite[\S2]{rydh-2009}: an algebraic stack $X$ is of \fndefn{s-global type} if \'etale-locally it is the quotient of a quasi-affine scheme by $\GL_N$ for some $N$. 
\begin{maintheorem}\label{thm:main_globaltype}
  Let $X$ be an algebraic stack of s-global type. If $X$ is of equicharacteristic zero (i.e., it is a $\Q$-stack), then the unbounded
  derived category $\DQCOH(X)$ is compactly generated by a countable set of perfect complexes. Moreover, for every quasi-compact open subset $U \subseteq X$, there exists a compact perfect complex with support exactly $X\setminus U$. 
\end{maintheorem}

Stacks of s-global type are frequently encountered in practice. Sumihiro's Theorem~\cite{MR0337963} and its recent generalization by Brion~\cite{MR3329192} show that many quotient stacks are of s-global type (Proposition \ref{P:sglobal_sumi}). Thus, we have the following corollary (see Corollary \ref{C:compact-generated-stacks} for an amplification).

\begin{corollary*}\label{C:quotient-stacks}
 Let $X$ be a variety over a field $k$ of characteristic zero. Let $G$ be an
 affine algebraic $k$-group acting on $X$. If $X$ is either (a) normal or
 (b) semi-normal and quasi-projective, then $\DCAT( \QCOH_G(X) )=\DQCOH([X/G])$ is
 compactly generated.
 Moreover, for every $G$-invariant open subset $U \subseteq X$, there exists a
 perfect $G$-equivariant complex with support exactly $X\setminus U$.
\end{corollary*}

A key advantage of Theorem \ref{thm:main_globaltype} over previous results (e.g., 
\cite[Cor.~3.22]{MR2669705} and \cite[Prop.~2.2.4.13]{MR2717173}) is its applicability to a much wider class of stacks.
The main result of 
\cite{AHR_lunafield} implies that algebraic stacks of finite type over a field with affine diagonal and linearly reductive 
stabilizers at closed points (e.g., stacks with a good moduli space) are of s-global type~\cite[Thm.~2.25]{AHR_lunafield}.
In general, it is not possible to generate $\DQCOH(X)$ by a single
perfect complex (e.g., $X=B\Gm$). 

Theorems \ref{thm:main_qf} and \ref{thm:main_globaltype} are both
consequences of a general result that we now describe. Let $\beta$ be a cardinal. Let $X$ be an algebraic stack. We say that $X$ satisfies the \emph{$\beta$-Thomason condition} if: 
\begin{enumerate}
\item $\DQCOH(X)$ is compactly generated by a set of cardinality $\leq \beta$; and
\item for every quasi-compact open subset $U\subset X$,
there is a compact perfect complex supported on the complement.
\end{enumerate}
We say that $X$ satisfies the \fndefn{Thomason condition} if it satisfies the $\beta$-Thomason condition for some $\beta$.
Our main results are that a very large class of stacks satisfy the 
Thomason condition. In order to prove these results, however, we found it necessary to 
consider the following refinement of the Thomason condition. 

We say that $X$ is \emph{$\beta$-\crisp} if the $\beta$-Thomason condition is satisfied for every \'etale localization of $X$ (Definition~\ref{def:beta-perfect}). If $X$ is $\beta$-\crisp{} and $\beta$ is
finite, then $X$ is compactly generated by a single perfect
complex. Hence Theorems \ref{thm:main_qf} and
\ref{thm:main_globaltype} are implied by
\begin{maintheorem}\label{thm:main_list}
  Let $p\colon X'\to X$ be a morphism of quasi-compact and quasi-separated algebraic stacks that is representable, separated, quasi-finite, locally of finite presentation, and faithfully flat. If
  $X'$ is $\beta$-\crisp, then $X$ is $\beta$-\crisp.  
\end{maintheorem}
Theorem \ref{thm:main_list} is proved using the technique of quasi-finite 
flat d\'evissage for algebraic stacks---due to the second author
\cite{MR2774654}---together with some descent results for compact generation. In
sections~\ref{sec:presheaves-tri-cats}--\ref{sec:descent-compact-gen} these
descent results are stated in great generality---for presheaves of triangulated
categories---without requiring monoidal or linear structures. This allows us to establish 
compact generation in other contexts (see below, 
Theorem~\ref{thm:general_qffdesc_cpt_gen}, and \S\ref{sec:applications}). Along the way, we will review and develop foundational material for unbounded derived categories on algebraic stacks. 

We also wish to point out that for schemes the fppf and quasi-finite flat topologies coincide, but for 
algebraic stacks they differ. Moreover, compact generation is not fppf local for algebraic 
stacks. Indeed, $\DQCOH(B_k\Ga)$ when $k$ is of characteristic $p>0$ has no compact objects besides $0$; thus, it is not compactly generated---even though it is so fppf-locally \cite[Prop.~3.1]{hallj_neeman_dary_no_compacts}. In particular Theorem \ref{thm:main_list} (as well as its generalizations to other 
contexts in this article), which is about quasi-finite local compact generation, can be 
viewed as the correct generalization of fppf local compact generation results to algebraic 
stacks.

\subsection*{Azumaya algebras and the cohomological Brauer group}
Our work is strongly influenced by To\"en's excellent paper~\cite{MR2957304} on
derived Azumaya algebras and generators of twisted derived categories.
In~\cite{MR2957304}, To\"en shows that compact generation of certain
linear categories on derived schemes is an fppf-local question. The
salient example is the \emph{derived category of twisted sheaves}
$\DCAT(\QCOH^\alpha(X))$, where the twisting is given by a Brauer class $\alpha$ of
$H^2(X,\mathbb{G}_m)$. A compact generator of $\DCAT(\QCOH^\alpha(X))$ gives rise to
a \emph{derived Azumaya algebra}---the endomorphism algebra of the
generator~\cite[Prop.~4.6]{MR2957304}. More classically, a twisted vector bundle
that is generating gives rise to an Azumaya algebra $\mathcal{A}$.

The Brauer group $\Br(X)$ classifies Azumaya algebras $\mathcal{A}$ up to
Morita equivalence, that is, up to equivalence of the category of modules
$\MOD(\mathcal{A})$. Moreover, $\MOD(\mathcal{A})\homotopic \QCOH^\alpha(X)$
for a unique element $\alpha$ in the cohomological Brauer group
$\Br'(X):=H^2(X,\mathbb{G}_m)_{\mathrm{tors}}$.
Existence of twisted vector bundles thus answers the question whether
$\Br(X)\to \Br'(X)$ is surjective. 

Constructing twisted vector bundles, or equivalently Azumaya algebras, is difficult. Indeed, the question is not local as
vector bundles rarely extend over open immersions.
When $X$ is affine, or the separated union of two affines, Gabber proved in his thesis that $\Br(X)=\Br'(X)$~\cite{MR611868},
or equivalently, that $\DCAT(\QCOH^\alpha(X))$ is compactly generated by a twisted vector bundle~\cite[Thm.~2.2.3.3]{MR2717173}.
The 
state of the art is also due to Gabber: twisted vector bundles exist if $X$ is quasi-projective \cite{deJong_result-of-Gabber}.

Compact objects of the
derived category are typically easier to construct as we may extend them over open immersions using
Thomason's localization theorem (Corollary \ref{cor:thomason_lift}).
With this technique, M.\ Lieblich proved that $\DCAT(\QCOH^\alpha(X))$ is
compactly
generated when $X$ is any quasi-compact and quasi-separated
scheme~\cite[Cor.~2.2.4.14]{MR2717173}. Lieblich has also
studied twisted vector bundles in great detail and obtained a number of
arithmetic applications.

In 
\S\ref{sec:applications}, we prove that compact generation is quasi-finite flat local for twisted derived categories. In
particular, we prove that on a quasi-compact algebraic stack with quasi-finite
and separated diagonal every twisted derived category has a compact generator
(Example~\ref{ex:Br=Br'}). We thus establish a derived analogue of
$\Br(X)=\Br'(X)$ for such stacks, extending the results of To\"en \cite{MR2957304} and Antieau--Gepner \cite{MR3190610}.
\subsection*{Sheaves of linear categories on derived Deligne--Mumford stacks}
Although we work with non-derived schemes and stacks, our methods are strong
enough to deduce similar
results for derived (and spectral) Deligne--Mumford
stacks (Example~\ref{ex:derived}). Indeed, if $X$ is a derived
Deligne--Mumford stack, then the small \'etale topos of $X$ is equivalent to
the small \'etale topos of the non-derived $0$-truncation $\pi_0 X$. Thus,
(local) compact generation of a presheaf of triangulated categories on $X$ can
be studied on $\pi_0 X$. 

Sometimes results for stacks can be deduced from schemes using a similar
approach: if $\pi\colon X\to X_\cms$ is a coarse moduli space, then a presheaf
$\mathcal{T}$ of triangulated categories on $X$ induces a presheaf
$\pi_*\mathcal{T}$ of triangulated categories on $X_\cms$. If
$\pi_*\mathcal{T}$ is locally compactly generated, then it is enough to show
that compact generation is local on $X_\cms$
to deduce compact generation of $\mathcal{T}(X)$. This is how To\"en extends
his result to Deligne--Mumford stacks admitting a coarse moduli
scheme~\cite[Cor.~5.2]{MR2957304}.

\subsection*{Perfect and compact objects}
As we already have mentioned, some care has to be taken since perfect objects
are not necessarily compact. The perfect objects are the \emph{locally compact}
objects or, equivalently, the dualizable objects.
If $X$ is a quasi-compact and quasi-separated
algebraic stack, then the following conditions are equivalent (Remark~\ref{rem:fin_coh_dim}):
\begin{itemize}
\item every perfect object of $\DQCOH(X)$ is compact;
\item the structure sheaf $\Orb_X$ is compact;
\item there exists an integer $d_0$ such that for all quasi-coherent
  sheaves $M$ on $X$, the cohomology groups $H^{d}(X,M)$ vanish for
  all $d>d_0$; and
\item the derived global section functor $\RDERF\Gamma\colon \DQCOH(X)\to
  \DCAT(\AB)$ commutes with small coproducts.
\end{itemize}
We say that a stack is \emph{concentrated} when it satisfies the
conditions above.

In \cite[Thm.~B]{hallj_dary_alg_groups_classifying} we give a complete list of the group schemes $G/k$
such that $BG$ is concentrated: every linear group and certain non-affine
groups in characteristic zero but only the linearly reductive groups in
positive characteristic. Note that
\cite[Thm.~A]{hallj_dary_alg_groups_classifying} also gives
many examples of classifying stacks that are not concentrated, yet compactly
generated.

Drinfeld and Gaitsgory have proved that noetherian algebraic stacks with affine
stabilizer groups in characteristic zero are
concentrated~\cite[Thm.~1.4.2]{MR3037900}. This is generalized in
\cite[Thm.~C]{hallj_dary_alg_groups_classifying} to positive characteristic.
In particular, a stack with finite stabilizers is concentrated if and only if
it is tame.

\subsection*{Perfect stacks}
Ben--Zvi, Francis and Nadler introduced the notion of a \emph{perfect}
(derived) stack in~\cite{MR2669705}. In our context, an algebraic stack $X$
is perfect if and only if it has affine diagonal, it is concentrated and
its derived category $\DQCOH(X)$ is compactly generated~\cite[Prop.~3.9]{MR2669705}. A
direct consequence of our main theorems
and~\cite[Thm.~C]{hallj_dary_alg_groups_classifying} is that the
following classes of algebraic stacks are perfect:
\begin{enumerate}
\item quasi-compact tame Deligne--Mumford stacks with affine diagonal; and
\item $\Q$-stacks of s-global type with affine diagonal.
\end{enumerate}
%
The affine diagonal assumption is needed only because it is required in
the definition of a perfect stack. It is useful though: if $X$ is perfect, then
$\DCAT(\QCOH(X))=\DQCOH(X)$ by~\cite{hallj_neeman_dary_no_compacts}.

In the terminology of Lurie~\cite[Def.~8.14]{dag11}, an algebraic stack is
perfect if it has quasi-affine diagonal, is concentrated and $\DQCOH(X)$ is
compactly generated. Thus in Lurie's terminology, we have shown that
\begin{enumerate}
\item quasi-compact tame Deligne--Mumford stacks with quasi-compact and
separated diagonals; and
\item $\Q$-stacks of s-global type
\end{enumerate}
are perfect.
\subsection*{Coherence}
Compact generation is extremely useful and we will illustrate this with a
simple application---also the origin of this paper. Let $A$ be a commutative
ring and $\MOD(A)$ the category of $A$-modules. A functor
$F\colon \MOD(A)\to \AB$ is \emph{coherent} if there exists a
homomorphism of $A$-modules $M_1 \to M_2$ together with isomorphisms
\[
F(N) \cong \coker\bigl(\Hom_A(M_2,N) \to \Hom_A(M_1,N)\bigr)
\]
natural in $N$.
This definition is due to Auslander \cite{MR0212070} who initiated the
study of coherent functors. Hartshorne studied in detail
\cite{MR1656482} coherent functors when $A$ is noetherian and $M_1$
and $M_2$ are coherent and obtained some very nice applications to
classical algebraic geometry. For background material on coherent
functors, we refer the reader to Hartshorne's article. Recently, the
first author has used coherent functors to prove Cohomology and Base
Change for algebraic stacks \cite{hallj_coho_bc} and to give a
new criterion for algebraicity of a stack
\cite{MR3589351}. 

Using the compact generation results of
this article, we can give a straightforward proof of the following
Theorem (combine Theorem \ref{thm:main_qf} with Corollary
\ref{cor:coherent_functor}). 
\begin{maintheorem}\label{mainthm:coherent_functor}
  Let $A$ be a noetherian ring and let $\pi \colon X \to \spec A$ be a
  proper morphism of algebraic stacks with finite diagonal. If $\cplx{F} \in \DQCOH(X)$
  and $\cplx{G} \in \DCAT_{\COH}^b(X)$, then the functor
  \[
  \Hom_{\Orb_X}(\cplx{F},\cplx{G} \tensor_{\Orb_X}^{\LDERF} \LDERF
  \QCPBK{\pi}(-)) \colon \MOD(A) \to \MOD(A)
  \]
  is coherent.
\end{maintheorem}
Theorem \ref{mainthm:coherent_functor} generalizes a result of the
first author for algebraic spaces \cite[Thm.~E]{hallj_coho_bc}, which
was proved using a completely different argument. The first author has
also proved a non-noetherian and infinite stabilizer variant of
Theorem \ref{mainthm:coherent_functor} at the expense of assuming that
$\cplx{G}$ has flat cohomology sheaves over $S$ \cite[Thm.~C]{hallj_coho_bc}.

\subsection*{Related results}
The first proof that $\DQCOH(X)$ is compactly generated when $X$ is a
quasi-compact separated scheme appears to be due to
Neeman~\cite[Prop.~2.5]{MR1308405} although he attributes the ideas to
Thomason~\cite{MR1106918}. Bondal--Van den Bergh~\cite[Thm.~3.1.1]{MR1996800}
adapted the proof to deal with quasi-separated schemes and noted that there is
a single compact generator. Lipman and Neeman further refined the result by
giving an effective bound on the existence of maps from the compact
generator~\cite[Thm.~4.2]{MR2346195}. As noted by Ben-Zvi, Francis and Nadler,
the proof of Bondal and Van den Bergh readily extends to derived
schemes~\cite[Prop.~3.19]{MR2669705}.

In~\cite[Thm.~6.1]{dag11} and~\cite[Thm.~1.5.10]{dag12} Lurie proves that
compact generation is \'etale local on $\mathbb{E}_\infty$-algebras and on
spectral algebraic spaces for quasi-coherent stacks (sheaves of linear
$\infty$-categories). Lurie uses \emph{scallop decompositions}, which are a special type 
of \'etale neighborhoods (or Nisnevich squares). Unfortunately, scallop decompositions do not 
exist for algebraic stacks. This is what necessitates our stronger inductive 
assumption---$\beta$-\crispness---for our local-global principle, Theorem~\ref{thm:main_list}. 
Indeed, to apply Thomason's localization theorem it is necessary to establish the existence of compact objects with prescribed support. 
On affine schemes (which appear in the scallop decompositions) this is
done using Koszul
complexes, cf.\ B\"okstedt and Neeman~\cite[Prop.~6.1]{MR1214458}. This is the basis
for our induction and also used in all previous proofs, e.g.,
To\"en~\cite[Lem.~4.10]{MR2957304} and~\cite[Prop.~6.9]{MR3190610}.

Drinfeld and Gaitsgory~\cite[Thm.~8.1.1]{MR3037900} prove that on an algebraic
stack of finite type over a field of characteristic zero with affine
stabilizers, the derived category of $D$-modules is compactly generated. They
remark that compact generation of $\DQCOH(X)$ is much
subtler and open in general~\cite[0.3.3]{MR3037900}.

Antieau~\cite{MR3161100} has considered local-global results for the
telescope conjecture. Some of these are generalized in~\cite{telescope-stacks}.

Krishna \cite{MR2570954} has considered the K-theory and G-theory for
tame Deligne--Mumford stacks with the resolution property admitting
projective coarse moduli schemes (i.e., projective stacks).

\subsection*{Future extensions}
In~\cite[Thm.~4.1]{MR2346195}, Lipman and Neeman prove that pseudo-coherent
complexes can be approximated arbitrarily well by perfect complexes on a
quasi-compact
and quasi-separated scheme. Local approximability by perfect complexes
is essentially the definition of pseudo-coherence so this is a local-global
result in the style of Theorem~\ref{thm:main_list}. This result has been
extended to algebraic spaces in~\cite[\spref{08HH}]{stacks-project} and
we expect that it can be extended to stacks with quasi-finite diagonal
using the methods of this paper amplified with $t$-structures. Similarly,
we expect that there is an effective bound on the compact generator
in Theorem~\ref{thm:main_qf} as in~\cite[Thm.~4.2]{MR2346195}.

\subsection*{Contents of this paper}
In \S\S 1--2 we recall and develop some generalities on unbounded
derived categories of quasi-coherent sheaves on stacks and
concentrated morphisms---working in the unbounded derived category is
absolutely essential for this range of mathematics. Unfortunately
some important foundational notions, such as concentrated morphisms,
had not been considered in the literature before.

In \S 3, we recall the concept of compact objects and
Thomason's localization theorem for triangulated categories.

In \S 4, we address fundamental results on perfect and compact objects in the derived 
categories of quasi-coherent sheaves on algebraic stacks. Using this we 
establish a general projection formula for stacks, tor-independent base
change, and finite flat duality. We also prove Theorem \ref{mainthm:coherent_functor} assuming Theorem \ref{thm:main_qf}.

In \S\S 5--6 we introduce presheaves of triangulated categories and
 Mayer--Vietoris triangles. We also prove our main result on descent
 of compact generation (Theorem \ref{thm:general_qffdesc_cpt_gen}).

In \S 7 we introduce the $\beta$-resolution property, which gives a
convenient method to keep track of the number of vector bundles needed for
generating the derived category of a stack with the resolution
property. 

In \S 8, we introduce compact generation with supports and
$\beta$-\crispness{} and relate these to Koszul complexes.

In \S 9, we prove the main theorems.

\subsection*{Acknowledgments}
It is our pleasure to acknowledge useful discussions with B.\ Bhatt and
A.\ Neeman. We would also like to thank D.\ Bergh, M.\ Hoyois, M.\ Lieblich and O.\ Schn\"urer for a
number of useful 
comments. We also wish to thank the referee for their careful reading and several suggested improvements to the article.

\subsection*{Notations and assumptions}
For an abelian category $\mathscript{A}$, denote by $\DCAT(\mathscript{A})$
its unbounded derived category. For a
complex $M \in \DCAT(\mathscript{A})$, denote its $i$th
cohomology group by $\COHO{i}(M)$. 
For a sheaf of rings $A$ on a topos 
$E$, denote by $\MOD(A)$ (resp.\ $\QCOH(A)$) the category of
$A$-modules (resp.\ the category of quasi-coherent $A$-modules). If the
sheaf of rings $A$ on the topos $E$ is implicit, it
will be convenient to denote $\MOD(A)$ as $\MOD(E)$ and
$\DCAT(\MOD(A))$ as $\DCAT(E)$. 

For algebraic stacks, we adopt the conventions of the \emph{Stacks
  Project} \cite{stacks-project}. This means that algebraic stacks are
stacks over the big fppf site of some scheme, admitting a smooth,
representable, surjective morphism from a scheme (note that there are no
separation hypotheses here). A morphism of algebraic stacks is
\fndefn{quasi-separated} if its diagonal and double diagonal are
represented by quasi-compact morphisms of algebraic spaces.

For a scheme $X$ denote its underlying topological space by
$|X|$. For a scheme $X$ and a point $x\in |X|$, denote by $\kappa(x)$
the residue field at $x$.  

Let $f \colon X \to Y$ be a $1$-morphism of algebraic stacks. Then for any
other $1$-morphism of algebraic stacks $g \colon Z \to Y$, we denote by
$f_Z \colon X_Z \to Z$ the pullback of $f$ by $g$. 
\section{Quasi-coherent sheaves on algebraic stacks}\label{sec:qc_shv}  
In this section we review derived categories of quasi-coherent sheaves
on algebraic stacks. For generalities on unbounded derived categories
on ringed topoi we refer the reader to \cite[\S18.6]{MR2182076}. In
\cite[\S18.6]{MR2182076}, a morphism of ringed topoi is assumed to
have a left exact inverse image---we will not make this assumption,
but instead indicate explicitly when it does and does not hold.

Let $X$ be an algebraic stack. Let $\MOD(X)$
(resp.~$\QCOH(X)$) 
denote the abelian category of $\Orb_X$-modules (resp.~quasi-coherent
$\Orb_X$-modules) on the lisse-\'etale topos of $X$
\cite[12.1]{MR1771927}. Let $\DCAT(X)$ (resp.~$\DQCOH(X)$)  denote
the unbounded derived category of $\MOD(X)$ (resp.~the full
subcategory of $\DCAT(X)$ with cohomology in $\QCOH(X)$). Superscripts
such as $+$, $-$, $\geq n$, and $b$ decorating $\DCAT(X)$ and
$\DQCOH(X)$ are to be interpreted as usual.

If $X$ is a Deligne--Mumford stack (e.g., a scheme or an algebraic
space), then there is an associated small \'etale topos which we denote
as $X_{\et}$. There is a natural morphism of ringed topoi
$\REST[X]\colon X_{\lisset} \to X_{\et}$. Let $\MOD(X_{\et})$
(resp.~$\QCOH(X_{\et})$) denote the abelian category of
$\Orb_{X_{\et}}$-modules (resp.~quasi-coherent
$\Orb_{X_{\et}}$-modules). The restriction of $(\REST[X])_*\colon
\MOD(X) \to \MOD(X_{\et})$ to $\QCOH(X)$ is fully faithful with
essential image $\QCOH(X_{\et})$~\cite[Prop.~13.2.3]{MR1771927}.
Let $\DQCOH(X_{\et})$ denote the
triangulated category $\DCAT_{\QCOH(X_{\et})}(\MOD(X_{\et}))$. Then
the natural functor $\RDERF(\REST[X])_*\colon \DQCOH(X) \to
\DQCOH(X_{\et})$ is an equivalence of categories~\cite[Prop.~12.10.1]{MR1771927}. If $X$ is a scheme,
then the corresponding statement for the Zariski topos also holds~\cite[Lem.13.1.5]{MR1771927}.

\subsection{Hypercoverings and simplicial sites}
We now recall the relationship between the unbounded derived
categories of quasi-coherent sheaves on an algebraic stack and those
on a smooth hypercovering (i.e., cohomological descent). Our approach follows
\cite{MR2312554} and  \cite{MR2434692}. Let $X$ be an algebraic stack
and let $p_\bullet \colon U_\bullet \to X$ be a smooth hypercovering
by algebraic spaces. Typically, we will take $U_{\bullet}$ to be the $0$-coskeleton 
associated to a smooth covering $p_0 \colon U_0 \to X$, where $U_0$ is an algebraic 
space. In plainer language, $U_n$ is the $(n+1)$th fiber product of $U$ over $X$ and the 
simplicial structure (i.e., face and degeneracy maps) come from the various projections 
and diagonals between the $U_n$ as $n$ varies. 

The simplicial algebraic space $U_\bullet$ gives rise to two semi-simplicial topoi: 
$U_{\bullet,\lisset}^+$ and $U_{\bullet,\et}^+$. The semi-simplicial topos $U_{\bullet,\et}^+$ is 
formed as follows: for each integer $n\geq 0$ there is the \'etale topos $U_{n,\et}$ and for 
each injective map $\delta\colon [n]=\{0,\dots,n\} \to [m]=\{0,\dots,m\}$ there is a 
morphism of topoi $\delta \colon U_{m,\et} \to U_{n,\et}$. A sheaf $F_\bullet$ on 
$U_{\bullet,\et}^+$ is a sheaf $F_n$ on each $U_{n,\et}$ together with transition maps 
$\delta^{-1}F_n \to F_m$ for each injective map $\delta\colon [n] \to [m]$ that are 
compatible with composition; the sheaf $F_\bullet$ is \emph{cartesian} if the transition 
maps are always isomorphisms.

The topos $U_{\bullet,\et}^+$ is naturally ringed by the flat sheaf
$\Orb_{U_{\bullet,\et}}^+$. Here flat means that the transition maps
$\delta^{-1}\Orb_{U_n,\et}\to\Orb_{U_m,\et}$ are flat. Let
$\MOD(U_{\bullet,\et}^+)$ denote the associated category of modules and
$\MOD_{\cartsubscript}(U_{\bullet,\et}^+)$ the subcategory of cartesian
sheaves. Here cartesian means that the transition maps
$\delta^*F_n=\delta^{-1}F_n\otimes_{\delta^{-1}\Orb_{U_{n,\et}}} \Orb_{U_{m,\et}}\to F_m$
are isomorphisms.

An $\Orb_{U_{\bullet,\et}}^+$-module is \emph{quasi-coherent} if it is cartesian and its restriction to each $U_{n,\et}$ is quasi-coherent. Let $\QCOH(U_{\bullet,\et}^+)$ denote the category of quasi-coherent sheaves. Let $\DCAT(U_{\bullet,\et}^+)$ be the unbounded derived category of 
$\MOD(U_{\bullet,\et}^+)$, let $\DCART(U_{\bullet,\et}^+)$ denote the subcategory whose 
objects have cartesian cohomology sheaves and let $\DQCOH(U_{\bullet,\et}^+)$ denote the subcategory with quasi-coherent cohomology sheaves. The semi-simplicial topos 
$U_{\bullet,\lisset}^+$ and its various module categories are defined similarly.

Thus, there are
natural morphisms of ringed topoi:
\begin{equation}
X_{\lisset} \xleftarrow{p_{\bullet,\lisset}^+} U_{\bullet,\lisset}^+
\xrightarrow{\REST[{U_\bullet}]} U_{\bullet,\et}^+.\label{eq:augmentation}
\end{equation}
One way phrasing smooth descent of quasi-coherent sheaves is that these morphisms of topoi induce equivalences of abelian categories:
\[
\QCOH(X) \xleftarrow{(p_{\bullet,\lisset}^+)_*}
\QCOH(U_{\bullet,\lisset}^+) \xrightarrow{
  (\REST[U_{\bullet}])_*} \QCOH(U_{\bullet,\et}^+)
\]
In \cite[Ex.~2.2.5]{MR2434692}, it is shown that this can be improved to \emph{unbounded cohomological descent}, that is, these morphism of topoi induce equivalences of triangulated categories:
\begin{equation}
\DQCOH(X) \xleftarrow{\RDERF (p_{\bullet,\lisset}^+)_*}
\DQCOH(U_{\bullet,\lisset}^+) \xrightarrow{\RDERF
  (\REST[U_{\bullet}])_*} \DQCOH(U_{\bullet,\et}^+).\label{eq:cohomological_descent}
\end{equation}
The morphisms $p_{\bullet,\lisset}^+$ and
$\REST[U_{\bullet}]$ have left exact inverse image functors.

\subsection{Operations in unbounded categories of modules}
We now record for future reference some useful
formulae. If $\cplx{M}$ and $\cplx{N} \in \DCAT(X)$,
then there is
\begin{align*}
\cplx{M} \tensor^{\LDERF}_{\Orb_X} \cplx{N} &\in \DCAT(X) \quad\text{(the derived tensor product)} \\
\SRHom_{\Orb_X}(\cplx{M},\cplx{N}) &\in \DCAT(X) \quad\text{(the derived sheaf
Hom functor)} \\
\RHom_{\Orb_X}(\cplx{M},\cplx{N}) &\in \DCAT(\AB) \quad\text{(the derived global Hom functor)}
\end{align*}
If in addition $\cplx{P}\in\DCAT(X)$, then we have a
functorial isomorphism:
\begin{equation}
  \Hom_{\Orb_X}(\cplx{M} \tensor^{\LDERF}_{\Orb_X} \cplx{N}, \cplx{P})
  \cong
  \Hom_{\Orb}(\cplx{M},\SRHom_{\Orb_X}(\cplx{N},\cplx{P})),\label{eq:ghomadj}  
\end{equation}
as well as a functorial quasi-isomorphism: 
\begin{equation}
  \SRHom_{\Orb_X}(\cplx{M} \tensor^{\LDERF}_{\Orb_X} \cplx{N},
  \cplx{P}) \homotopic
  \SRHom_{\Orb_X}(\cplx{M},\SRHom_{\Orb_X}(\cplx{N},\cplx{P}))\label{eq:lhomadj}. 
\end{equation}
Letting $\RDERF \Gamma(X,-)  = \RHom_{\Orb_X}(\Orb_X,-)$, there
is also a natural quasi-isomorphism:
\begin{equation}
  \RHom_{\Orb_X}(\cplx{M},\cplx{N}) \homotopic \RDERF \Gamma
  \SRHom_{\Orb_X}(\cplx{M},\cplx{N}). \label{eq:lghom} 
\end{equation}
If $\cplx{M}$ and $\cplx{N}$ belong to $\DQCOH(X)$,
then $\cplx{M}\tensor^{\LDERF}_{\Orb_X}\cplx{N} \in \DQCOH(X)$. These results are all
consequences of \cite[\S6]{MR2312554} and
\cite[\S\S2.1--2.2, Ex.~2.2.4]{MR2434692}. Since the category $\DQCOH(X)$ is well generated
\cite[Thm.~B.1]{hallj_neeman_dary_no_compacts} and the functor
$-\tensor^{\LDERF}_{\Orb_X}\cplx{M} \colon \DQCOH(X) \to \DQCOH(X)$
preserves small coproducts, it admits a right adjoint
\cite[Thm.~8.4.4]{MR1812507}
\[
\SRHom_{\Orb_X}^{\qcsubscript}(\cplx{M},-) \colon \DQCOH(X) \to \DQCOH(X).
\]
In fact, if
\[
\QCHR[X] \colon \DCAT(X) \to \DQCOH(X)
\]
is the right adjoint to the inclusion $\DQCOH(X) \subseteq \DCAT(X)$, which exists for the 
same reasons as above, then 
\[
\SRHom_{\Orb_X}^{\qcsubscript}(\cplx{M},-) \simeq \QCHR[X](\SRHom_{\Orb_X}(\cplx{M},-)).
\]
Note that while the formation of $\SRHom_{\Orb_X}(\cplx{M},-)$ is smooth local on $X$, 
this is not true in general for $\SRHom_{\Orb_X}^{\qcsubscript}(\cplx{M},-)$. It is true, 
however, if $\cplx{M}$ is perfect 
(Lemma~\ref{L:basic_props_perfect}\itemref{L:basic_props_perfect:dual}).

\subsection{Direct and inverse image}
For a morphism of algebraic stacks $f \colon X \to Y$, the induced
morphism of ringed topoi $f_{\lisset} \colon X_{\lisset} \to
Y_{\lisset}$ does not necessarily have a left exact inverse image functor \cite[5.3.12]{MR1963494}. Thus the construction of
the correct derived functors of $f^* \colon \QCOH(Y) \to \QCOH(X)$ is
somewhat subtle. There are currently two approaches to constructing
these functors. The first, due to M.~Olsson \cite{MR2312554} and
Y.~Laszlo and M.~Olsson \cite{MR2434692}, uses cohomological
descent. The other approach appears in the Stacks Project
\cite[Tag \spref{07BD}]{stacks-project}. In this article, 
we will employ the approach of Olsson and Laszlo--Olsson, which we
now briefly recall.

Let $f\colon X \to Y$ be a morphism of algebraic stacks. Let $q \colon
V \to Y$ be a smooth surjection from an algebraic space. Let $U \to
X\times_Y V$ be another smooth surjection from an algebraic space. Let
$\tilde{f}\colon U \to V$ be the resulting morphism of algebraic
spaces and let $p\colon U \to X$ be the resulting smooth covering. By
\eqref{eq:augmentation}, there is an induced $2$-commutative diagram
of ringed topoi:
\begin{equation}
  \xymatrix@C4pc{X \ar[d]_f & \ar[l]_{p_{\bullet,\lisset}^+}
    \ar[d]_{\tilde{f}_{\bullet,\lisset}^+} U_{\bullet,\lisset}^+
    \ar[r]^{\REST[{U_\bullet}]} & U_{\bullet,\et}^+
    \ar[d]^{\tilde{f}_{\bullet,\et}^+}\\
    Y & \ar[l]^{q_{\bullet,\lisset}^+} V_{\bullet,\lisset}^+
    \ar[r]_{\REST[{V_\bullet}]} & V_{\bullet,\et}^+. }
  \label{eq:simplicial_construction_of_derived_functors}
\end{equation}
The $2$-commutativity of the diagram above induces natural
transformations:
\begin{align}
  \label{eq:relationship_derived_pushforwards:mod_lisset}
  \RDERF \MODPSH{f} &\Rightarrow \RDERF (q_{\bullet,\lisset}^+)_*
  \RDERF (\tilde{f}_{\bullet,\lisset}^+)_* \LDERF
  (p_{\bullet,\lisset}^+)^* \quad \mbox{and}\\
  \RDERF (\tilde{f}_{\bullet,\et}^+)_* &\Rightarrow \RDERF(\REST[V_{\bullet}])_*\RDERF
  (\tilde{f}_{\bullet,\lisset}^+)_* \LDERF(\REST[U_{\bullet}])^*,
  \label{eq:relationship_derived_pushforwards:et_lisset}
\end{align}
which are natural isomorphisms for those complexes with quasi-coherent
cohomology that are sent to complexes with quasi-coherent cohomology
by $\RDERF(\tilde{f}_{\bullet,\lisset}^+)_*$ or $\RDERF(\tilde{f}_{\bullet,\et}^+)_*$. 

\begin{remark}
  Note, however, that if $f \colon X \to Y$ is not representable, then
  $\RDERF \MODPSH{f}$ does not, in general, send $\DQCOH(X)$ to
  $\DQCOH(Y)$---even if $f$ is proper and \'etale and $X$ and $Y$ are
  smooth Deligne--Mumford stacks
  \cite[\spref{07DC}]{stacks-project}. The problem is that
  quasi-compact and quasi-separated morphisms of algebraic stacks can
  have unbounded cohomological dimension, which is in contrast to the
  case of schemes and algebraic spaces
  \cite[\spref{073G}]{stacks-project}. In the next section we will
  clarify this with the concept of a \fndefn{concentrated} morphism.
\end{remark}

Another crucial observation here is that the morphism of topoi
$\tilde{f}^+_{\bullet,\et}$ has a left exact inverse image
functor. The general theory now gives rise to an unbounded derived
functor $\LDERF(\tilde{f}^+_{\bullet,\et})^* \colon
\DCAT(V_{\bullet,\et}^+) \to \DCAT(U_{\bullet,\et}^+)$, which is left
adjoint to $\RDERF (\tilde{f}^+_{\bullet,\et})_* \colon
\DCAT(U_{\bullet,\et}^+) \to \DCAT(V_{\bullet,\et}^+)$. The functor
$\LDERF(\tilde{f}^+_{\bullet,\et})^*$ is easily verified to preserve
small coproducts and complexes with quasi-coherent
cohomology. Using the equivalences of
\eqref{eq:cohomological_descent}, we may now define a functor $\LDERF
\QCPBK{f} \colon\DQCOH(Y) \to \DQCOH(X)$ such that $\COHO{0}(\LDERF
\QCPBK{f} \shv{M}[0]) \cong f^*\shv{M}$, whenever $\shv{M} \in
\QCOH(Y)$. If $f \colon X \to Y$ is flat, then for all integers $q$
and all $\cplx{M} \in \DQCOH(X)$ there is a natural isomorphism:
\begin{equation}
  f^*\COHO{q}(\cplx{M}) \cong \COHO{q}(\LDERF\QCPBK{f}\cplx{M}) \label{eq:fl_pbk}
\end{equation}
Since the category $\DQCOH(Y)$ is well generated
\cite[Thm.~B.1]{hallj_neeman_dary_no_compacts} and the functor $\LDERF
\QCPBK{f}$ preserves small coproducts, it admits a right adjoint
\cite[Thm.~8.4.4]{MR1812507} 
\[
\RDERF \QCPSH{f} \colon \DQCOH(X) \to \DQCOH(Y). 
\]
The functor above is closely related to the functors we have already
seen. Indeed, since $\LDERF \QCPBK{f}\Orb_Y[0] \homotopic \Orb_X[0]$,
it follows that if $\cplx{M} \in \DQCOH(X)$, then
\begin{equation}
  \RDERF \Gamma(Y,\RDERF\QCPSH{f}\cplx{M}) \homotopic \RDERF \Gamma(X,\cplx{M}).\label{eq:global_sections_qcpsh}
\end{equation}
We now describe $\RDERF \QCPSH{f}$ locally. Let 
\[
\QCHR[V_{\bullet,\et}^+] \colon \DCAT(V_{\bullet,\et}^+) \to \DQCOH(V_{\bullet,\et}^+)
\]
be a right adjoint to the natural inclusion functor
$\DQCOH(V_{\bullet,\et}^+) \to \DCAT(V_{\bullet,\et}^+)$, which exists
by \cite[Thm.~8.4.4]{MR1812507}. A straightforward calculation,
utilizing the equivalences \eqref{eq:cohomological_descent}, induces
a natural isomorphism of functors:
\begin{equation}
\RDERF \QCPSH{f} \homotopic \RDERF (q_{\bullet,\lisset}^+)_*\LDERF
(\REST[V_{\bullet}])^*\QCHR[V_{\bullet,\et}^+]\RDERF
(\tilde{f}_{\bullet,\et}^+)_* \RDERF(\REST[U_{\bullet}])_*\LDERF
(p_{\bullet,\lisset}^+)^*.\label{eq:two_pushforwards}
\end{equation}
The following Lemma clarifies the situation
somewhat.
\begin{lemma}\label{lem:qcqs_conditions}
  If $f\colon X \to Y$ is a morphism of algebraic stacks that is
  quasi-compact and quasi-separated, then the following hold. 
  \begin{enumerate}
  \item \label{item:qcqs_conditions:bdd_psh_cplx} The restriction of
    $\RDERF \MODPSH{f}$ to $\DQCOH^+(X)$ factors through
    $\DQCOH^+(Y)$.
  \item \label{item:qcqs_conditions:global_trivial_duality} The
    restrictions of $\RDERF \MODPSH{f}$ and $\RDERF \QCPSH{f}$ to
    $\DQCOH^+(X)$ are isomorphic.
  \item \label{item:qcqs_conditions:bdd_coprod} For each integer $d$,
    the restriction of the functor $\RDERF\QCPSH{f}$ to
    $\DQCOH^{[d,\infty)}(X)$ preserves direct limits (in particular,
    small coproducts).
  \item \label{item:qcqs_conditions:bdd_fl_bc} Consider a
    $2$-cartesian diagram of algebraic stacks:
    \[
    \xymatrix{X' \ar[r]^{g'} \ar[d]_{f'}& X \ar[d]^f \\ Y' \ar[r]^{g}
      & Y.} 
    \]
    If $g$ is flat, then the base change transformation
    \[
    \LDERF\QCPBK{g}\RDERF \QCPSH{f} \Rightarrow \RDERF\QCPSH{f'}
    \LDERF \QCPBK{(g')}
    \]
    is an isomorphism upon restriction to $\DQCOH^+(X)$.
  \end{enumerate}
\end{lemma}
\begin{proof}
  Claim \itemref{item:qcqs_conditions:bdd_psh_cplx} is
  \cite[Lem.~6.20]{MR2312554}. Claim
  \itemref{item:qcqs_conditions:global_trivial_duality} follows from
  cohomological descent \eqref{eq:cohomological_descent},
  claim \itemref{item:qcqs_conditions:bdd_psh_cplx} and equations
  \eqref{eq:relationship_derived_pushforwards:mod_lisset},
  \eqref{eq:relationship_derived_pushforwards:et_lisset}, and
  \eqref{eq:two_pushforwards}. 

  For \itemref{item:qcqs_conditions:bdd_coprod}, by
  \itemref{item:qcqs_conditions:global_trivial_duality}, we may replace
  $\RDERF\QCPSH{f}$ by $\RDERF\MODPSH{f}$. The hypercohomology
  spectral sequence:
  \begin{equation}
    \RDERF^{r}\MODPSH{f}\COHO{s}(\cplx{M}) \Rightarrow
    \RDERF^{r+s}\MODPSH{f}\cplx{M}\label{eq:hyper_coho} 
  \end{equation}
  now applies and it is thus sufficient to prove that the higher
  pushforwards 
  \[
  \RDERF^r\MODPSH{f} \colon \QCOH(X) \to \QCOH(Y)
  \]
  preserve direct limits for every integer $r\geq 0$. This is local on
  $Y$ for the smooth topology, so we may assume that $Y$ is an affine
  scheme. Thus it suffices to prove that the cohomology functors
  $H^r(X_{\lisset},-) \colon \QCOH(X) \to \AB$ preserve direct limits
  for every integer $r\geq 0$. Since $X$ is quasi-compact and
  quasi-separated, this is well-known (e.g.,
  \cite[\spref{0739}]{stacks-project}). 

  The base change transformation of
  \itemref{item:qcqs_conditions:bdd_fl_bc} exists by functoriality of
  the adjoints. Applying \itemref{item:qcqs_conditions:global_trivial_duality}
  we may replace $\QCPSH{f}$ and $\QCPSH{f'}$ by $\MODPSH{f}$ and
  $\MODPSH{f'}$ respectively. The statement is now local on $Y$ and
  $Y'$ for the smooth topology, so we may assume that both $Y$ and
  $Y'$ are affine schemes. Small modifications to the argument of
  \cite[\spref{073K}]{stacks-project} complete the proof.
\end{proof}
In \S\ref{sec:conc_stacks} we describe a class of morphisms for which the conclusions of
Lemma \ref{lem:qcqs_conditions} remain valid
in the unbounded derived category.

\subsection{Comparison with definitions in derived algebraic geometry}
We will now compare the derived category of quasi-coherent
sheaves $\DQCOH(X)$ that we are using in
this paper with derived categories of
quasi-coherent sheaves as defined in derived algebraic geometry.
This is not used in the
remainder of the paper (but see Example~\ref{ex:derived}).

First recall that if $\mathcal{A}$ is an abelian category, then the derived
category $\DCAT(\mathcal{A})$ is the homotopy category of a natural stable
$\infty$-category
$\DCATinfty(\mathcal{A})$~\cite[Def.\ 1.3.5.8]{lurie_highalg}.  If in addition
$\mathcal{C}\subset\mathcal{A}$ is a weak Serre subcategory, then we can
consider the full $\infty$-subcategory $\DCATinfty_{\mathcal{C}}(\mathcal{A})$
of objects with cohomology in $\mathcal{C}$ and this has homotopy
category $\DCAT_{\mathcal{C}}(\mathcal{A})$. We thus define
$\DQCOHinfty(X):=\DCATinfty_{\QCOH(X)}(\MOD(X))$ which has homotopy category
$\DQCOH(X)$. If $X$ is Deligne--Mumford, we also define
$\DQCOHinfty(X_{\et}):=\DCATinfty_{\QCOH(X_{\et})}(\MOD(X_{\et}))$ which is
equivalent to $\DQCOHinfty(X)$~\cite[Prop.~12.10.1]{MR1771927}.

\begin{proposition}[{cf.\ \cite[Rem.~1.2.3]{MR3037900}}]\label{P:dqcohinfty}
Let $X$ be an algebraic stack. Then
\[
\DQCOHinfty(X)=\varprojlim_{\spec A\to X} \DCATinfty(\MOD(A))
\]
where the limit is taken in the $\infty$-category of $\infty$-categories
and is over either (i) the category of all affine schemes over $X$,
(ii) the full subcategory of those smooth over $X$, or (iii) the subcategory
of those smooth over $X$ with only smooth $X$-morphisms $\spec A\to \spec B$.
\end{proposition}

We denote the limit of the right-hand side going over all morphisms by
$\QCOHinfty(X)$. This is a stable $\infty$-category which is left- and
right-complete, and has a t-structure with heart the abelian category $\QCOH(X)$~\cite[Cor.~9.1.3.2,
  Rem.~9.1.3.3]{lurie_sag}.
The $\infty$-category $\QCOHinfty(X)$ also makes
sense for any contravariant
functor $X$ from affine schemes to groupoids and can be adapted to variants in
derived algebraic geometry. This is how derived categories of
quasi-coherent sheaves usually are defined in derived algebraic geometry,
cf.\ \cite[\S3.1]{MR2669705}, \cite[\S1.2]{MR3037900}, 
\cite[\S6.2.2]{lurie_sag} and \cite[I.3, \S1.1]{gaitsgory-rozenblyum_dag-book}.

The category $\DQCOH(X)$ is
left-complete~\cite[Thm.~B.1]{hallj_neeman_dary_no_compacts}. We do not use
this fact in the proof so we obtain an independent proof of the
left-completeness of both $\DQCOH(X)$ and $\DQCOHinfty(X)$.




\begin{proof}[Proof of Proposition~\ref{P:dqcohinfty}]
  The limits restricted to smooth morphisms or to smooth morphisms with
  smooth maps between them, are also equivalent to
  $\QCOHinfty(X)$~\cite[I.3, \S1.4.2]{gaitsgory-rozenblyum_dag-book}.
  Moreover, if $U\to X$ is a smooth presentation and $U^+_\bullet$ is the
  corresponding semi-simplicial algebraic space, then restricting
  the limit to the diagram $U^+_\bullet$ gives the same
  limit (use that $X$ is a stack and that $\Delta^+\subset \Delta$ is right
  cofinal, cf.\ \cite[Prop. 6.2.3.1 and pf.\ of Prop.~9.1.3.1]{lurie_sag}
  or~\cite[I.3, Cor.~1.3.11]{gaitsgory-rozenblyum_dag-book}).
  We may also instead take the limit over the site $U^+_{\bullet,\et}$
  since $U^+_\bullet\colon \Delta^+\to U^+_{\bullet,\et}$ is right cofinal.

  We have a sequence of maps between $\infty$-categories
  \begin{align*}
  \DQCOHinfty(X)\xrightarrow{\;\alpha\;}
  \DQCOHinfty(U^+_{\bullet,\et}) &\xrightarrow{\;\beta\;}
  \varprojlim_{V\in U^+_{\bullet,\et}} \DQCOHinfty(U^+_{\bullet,\et}/V)\\
  &\xrightarrow{\;\gamma\;}
  \varprojlim_{V\in U^+_{\bullet,\et}} \DQCOHinfty(V)\xleftarrow{\;\delta\;}
  \QCOHinfty(X).
  \end{align*}
  That $\alpha$ is an equivalence follows by unbounded cohomological
  descent~\eqref{eq:cohomological_descent}. That $\beta$ is an equivalence
  follows from Lemma~\ref{L:dcatinfty-topos} applied to the semi-simplicial
  \'etale site $U^+_{\bullet,\et}$ since having quasi-coherent cohomology can
  be verified locally.
  The map $\gamma$ comes from the morphism of topoi $\epsilon\colon
  U^+_{\bullet,\et}/V\to V_\et$. Since
  $\epsilon_*$ (restriction) is exact and $\epsilon^*$ is exact and fully
  faithful with
  essential image the cartesian modules, we obtain equivalences between
  cartesian modules and between derived categories of modules with cartesian
  cohomology sheaves, cf.\ \cite[Prop.~12.10.1]{MR1771927}. This shows that
  $\gamma$ is an equivalence. We saw that $\delta$ was an equivalence above.
\end{proof}

\begin{remark}
  The maps $\alpha$, $\beta$ and $\gamma$ are also equivalences if we replace
  quasi-coherent cohomology with cartesian cohomology.
\end{remark}

\begin{remark}
  If in addition $X$ has affine diagonal, or is noetherian and affine-pointed,
  then the natural functor $\DCAT^+(\QCOH(X))\to\DQCOH^+(X)$ is
  an equivalence~\cite[App.~C]{hallj_neeman_dary_no_compacts}.
  It follows that $\DCATinfty^+(\QCOH(X))\to \QCOHinfty^+(X)$ is an
  equivalence, which can also be proven directly,
  cf.\ \cite[Rem.~1.2.10]{MR3037900} and~\cite[I.3, Prop.~2.4.3]{gaitsgory-rozenblyum_dag-book}.
  Note that $\DCAT(\QCOH(X))$ and $\DCATinfty(\QCOH(X))$ are not always
  left-complete, e.g., when $X=B\Ga$ in positive
  characteristic~\cite{MR2875857}. They are left-complete, and hence coincide
  with $\DQCOH(X)$ and $\DQCOHinfty(X)$, respectively, when
  $\DQCOH(X)$ is compactly
  generated~\cite[Thm.~1.2]{hallj_neeman_dary_no_compacts}.
\end{remark}

\newcommand{\Einfty}{\mathbb{E}_\infty}
\newcommand{\CAT}{\mathsf{Cat}}
\newcommand{\SET}{\mathsf{Set}}
\newcommand{\dg}{\mathrm{dg}}
\newcommand{\dgCAT}{\CAT_\dg(\Z)}
\newcommand{\CH}{\mathrm{CH}}

\begin{lemma}\label{L:dcatinfty-topos}
Let $(\mathscript{T},\Orb)$ be a ringed topos and let
$(\mathscript{T}/U,\Orb|_U)$ denote the localized topos for
any $U\in \mathscript{T}$. Then the assignment $U\mapsto
\DCATinfty(\MOD(\Orb|_U))$ is a sheaf, that is, there is a limit preserving
functor
\begin{align*}
\mathscript{T}^\opp &\longrightarrow \CAT_\infty\\
U &\longmapsto \DCATinfty(\MOD(\Orb|_U)).
\end{align*}
\end{lemma}
\begin{proof}
We identify a ringed topos with a $\Einfty$-ringed $\infty$-topos by taking
nerves. Such an $\infty$-topos is a $1$-topos, that is, discrete, and $\Orb$
is discrete.
We let $\MODinfty_{\Orb}$ denote the $\infty$-category of $\Orb$-module
spectra. Its heart is the category of usual modules
$\MOD(\Orb)$~\cite[Def.~2.1.0.1, Rem.~2.1.2.1]{lurie_sag}. By the universal
property of derived categories, there is a functor
$\DCATinfty^+(\MOD(\Orb|_U))\to \MODinfty_{\Orb|_U}$.
This functor extends to
a fully faithful functor
$\DCATinfty(\MOD(\Orb|_U))\to \MODinfty_{\Orb|_U}$ whose essential
image is the full subcategory of hypercomplete
objects~\cite[Cor.~2.1.2.3]{lurie_sag}.

By~\cite[Rem.~2.1.0.5]{lurie_sag}, there is a limit-preserving functor $U
\mapsto \MODinfty_{\Orb|_U}$. Since being hypercomplete is a local
property~\cite[Rem.~6.5.2.22]{MR2522659}, we obtain a limit-preserving functor
$U \mapsto \DCATinfty(\MOD(\Orb|_U))$.
\end{proof}

\begin{remark}[Deligne--Mumford stacks]
  If $X$ is a Deligne--Mumford stack, we can associate a spectral
  Deligne--Mumford stack $\sX$ to $X$.  To $\sX$, one associates a stable
  $\infty$-category of quasi-coherent sheaves $\QCOHinfty(\sX)$, a subcategory
  of $\MODinfty(\sX)=\MODinfty(\Orb_{X_\et})$. It is
  equivalent to $\QCOHinfty(X)$ as defined
  above~\cite[Prop.~6.2.4.1]{lurie_sag}, hence to
  $\DQCOHinfty(X)$. This can also be seen directly as
  follows~\cite[Cor.~2.2.6.2]{lurie_sag}.
  The $\infty$-category $\DCATinfty(\MOD(\Orb_{X_\et}))$ can be identified with
  the full subcategory of hypercomplete objects of
  $\MODinfty(\sX)$~\cite[Cor.~2.1.2.3]{lurie_sag} and $\QCOHinfty(\sX)$ can be
  identified with the full subcategory of hypercomplete objects with
  quasi-coherent homotopy sheaves~\cite[Prop.~2.2.6.1]{lurie_sag}. That is,
  $\DQCOHinfty(X_\et)=\QCOHinfty(\sX)$.
\end{remark}

\section{Concentrated morphisms of algebraic stacks}\label{sec:conc_stacks}
A morphism of schemes $f \colon X \to Y$ is concentrated if it is
quasi-compact and quasi-separated \cite[\S3.9]{MR2490557}. Concentrated morphisms of schemes
are natural to consider when working with
unbounded derived categories of quasi-coherent sheaves. Indeed, if $f$
is concentrated, then
\begin{enumerate}
\item $\RDERF\QCPSH{f}$ coincides with the restriction of $\RDERF
(f_{\mathrm{Zar}})_*$ to $\DQCOH(X)$,
\item $\RDERF\QCPSH{f}$ preserves small coproducts, and
\item $\RDERF\QCPSH{f}$ is compatible with flat base change on $Y$.
\end{enumerate}
Here, as before, $\RDERF\QCPSH{f}$ denotes the right adjoint to the
unbounded derived functor $\LDERF \QCPBK{f} \colon \DQCOH(Y) \to \DQCOH(X)$.
In this
section we isolate a class of morphisms of algebraic stacks, which we
will also call concentrated, that enjoy the same properties.

\begin{definition}
Let $n\geq 0$ be an integer. A quasi-compact and quasi-separated morphism
of algebraic stacks $f \colon X \to Y$ has \fndefn{cohomological
  dimension $\leq n$} if for all $i>n$ and all $\shv{M}\in
\QCOH(X)$ we have that $\RDERF^i \MODPSH{f}\shv{M} = 0$ (by Lemma
\ref{lem:qcqs_conditions}\itemref{item:qcqs_conditions:global_trivial_duality},
this is equivalent to $\RDERF^i \QCPSH{f}\shv{M} = 0$).  
\end{definition}

The next result
is inspired by \cite[Prop.~3.9]{2008arXiv0804.2242A}, where similar
results are proven in the context of cohomologically affine morphisms. Note, however, that cohomologically
affine morphisms are not quite the same as morphisms of cohomological
dimension $\leq 0$ \cite[Rem.~3.5]{2008arXiv0804.2242A}. 
\begin{lemma}\label{lem:fcd}
  Let $f \colon X \to Y$ be a $1$-morphism of algebraic stacks that is
  quasi-compact and quasi-separated. Let $n\geq 0$
  be an integer.
  \begin{enumerate}
  \item\label{lem:fcd:item:2mor} Let $\alpha \colon f
    \Rightarrow f'$ be a $2$-morphism. If $f$ has cohomological
    dimension $\leq n$, then so has $f'$.
  \item\label{lem:fcd:item:ff_desc} Let 
    $g \colon Z \to Y$ be a $1$-morphism of algebraic stacks that is
    faithfully flat. If  $f_Z \colon X\times_Y Z
    \to Z$ has cohomological dimension $\leq n$, then so has $f$.
  \item\label{lem:fcd:item:affine} If $f$ is affine, then
    it has cohomological dimension $\leq 0$. 
  \item\label{lem:fcd:item:comp} Let $h \colon W \to X$ be a $1$-morphism
    of algebraic stacks that is quasi-compact and quasi-separated and
    let $m\geq 0$ be an integer. If $f$ (resp.~$h$) has
    cohomological dimension $\leq n$ (resp.~$\leq m$), then the
    composition $f\circ h \colon W \to Y$ has cohomological
    dimension $\leq m+n$. 
  \item\label{lem:fcd:item:quaff_bc} Let $g \colon Z \to Y$ be a
    $1$-morphism of algebraic stacks that is \emph{quasi-affine}. If $f$ has
    cohomological dimension $\leq n$, then so has the $1$-morphism 
    $f_Z \colon X\times_Y Z \to Z$.
  \item\label{lem:fcd:item:quaff_delta} Let $g \colon Z \to Y$ be a
    $1$-morphism of algebraic stacks. If $f$
    has cohomological dimension $\leq n$ and $Y$ has
    quasi-affine diagonal, then the $1$-morphism $f_Z \colon X\times_Y Z
    \to Z$ has cohomological dimension $\leq n$.
  \end{enumerate}
\end{lemma}
\begin{proof}
  The claim \itemref{lem:fcd:item:2mor} is trivial. To address the claim
  \itemref{lem:fcd:item:ff_desc} we note that higher pushforwards
  commute with flat base change (Lemma \ref{lem:qcqs_conditions}\itemref{item:qcqs_conditions:bdd_fl_bc}). As faithfully flat
  morphisms are conservative, the morphism $f$ has cohomological dimension $\leq 
  n$. The claim \itemref{lem:fcd:item:affine} follows trivially from
  \itemref{lem:fcd:item:ff_desc}. The claim \itemref{lem:fcd:item:comp}
  follows from the Leray spectral sequence.

  We now address the claim \itemref{lem:fcd:item:quaff_bc}. Denote the
  pullback of $g$ by $f$ as $g_X \colon Z_X \to X$ and throughout we fix $M
  \in \QCOH(Z_X)$. We first assume that the morphism $g$ is a
  quasi-compact open immersion. In this situation the adjunction
  $\QCPBK{(g_X)}\MODPSH{(g_X)}M \to M$ is an isomorphism. For $i\geq 0$
  we deduce that there are isomorphisms in $\QCOH(Z)$:
  \[
  \RDERF^i
  \MODPSH{(f_Z)}\bigl(\QCPBK{(g_X)}\MODPSH{(g_X)}M\bigr) \to \RDERF^i \MODPSH{(f_Z)}M.
  \]
  Since higher pushforward commute with flat base change, 
  we deduce that for all $i\geq 0$ there are isomorphisms:
  \[
  g^*\RDERF^i\MODPSH{f}\bigl(\MODPSH{(g_X)}M\bigr) \to \RDERF^i \MODPSH{(f_Z)}M.
  \]
  Since $\MODPSH{(g_X)}M \in \QCOH(X)$, it follows that $f_Z \colon
  X_Z\to Z$ has cohomological dimension $\leq n$. Next
  assume that the morphism $g$ is affine. Then the morphism $g_X$ is also
  affine and so, by \itemref{lem:fcd:item:affine}, both morphisms have
  cohomological dimension $\leq 0$. By \itemref{lem:fcd:item:comp} we
  conclude that the composition $f\circ g_X \colon Z_X \to Y$ has
  cohomological 
  dimension $\leq n$. But we have a $2$-isomorphism
  $f\circ g_X \Rightarrow g\circ f_Z$ and so by
  \itemref{lem:fcd:item:2mor} the morphism $g\circ f_Z$
  has cohomological dimension $\leq n$. By the Leray spectral
  sequence, however, we see that there is an isomorphism for all
  $i\geq 0$:
  \[
  \MODPSH{g}\RDERF^i \MODPSH{(f_Z)}M \to \RDERF^i\MODPSH{(g\circ f_Z)}M.
  \]
  Since the morphism $g$ is affine, the functor $g_{*}$ is
  faithful; thus we conclude that the morphism $f_Z$ has
  cohomological dimension $\leq n$.  In general, a quasi-affine
  morphism $g \colon Z \to Y$ factors as $Z \xrightarrow{j} \bar{Z}
  \xrightarrow{\bar{g}} Y$, where the morphism $j$ is a
  quasi-compact open immersion and the morphism $\bar{g}$ is
  affine. Combining the above completes the proof of
  \itemref{lem:fcd:item:quaff_bc}. 

  To prove the claim \itemref{lem:fcd:item:quaff_delta} we observe that by \itemref{lem:fcd:item:ff_desc} the statement is
  smooth local on $Z$---thus we are free to assume that $Z$ is an affine
  scheme. Since the diagonal of the stack $Y$ is quasi-affine,
  the morphism $g \colon Z \to Y$ is quasi-affine. An application of 
  \itemref{lem:fcd:item:quaff_bc} now gives the claim. 
\end{proof}
We wish to point out that Lemma
\ref{lem:fcd}\itemref{lem:fcd:item:quaff_delta} is false if $Y$ does
not have affine stabilizers
\cite[Rem.~1.6]{hallj_dary_alg_groups_classifying}. 
\begin{definition}
A quasi-compact and quasi-separated morphism $f \colon X \to Y$ of algebraic
stacks has \fndefn{finite cohomological dimension} if
there exists an integer $n\geq 0$ such that the morphism $f$ has
cohomological dimension $\leq n$. 
\end{definition}

Morphisms of quasi-compact
and quasi-separated algebraic
spaces have finite cohomological dimension \cite[\spref{073G}]{stacks-project}. V.~Drinfeld and D.~Gaitsgory
\cite[Thm.~1.4.2, \S2]{MR3037900} have shown that a morphism of
quasi-compact and quasi-separated algebraic stacks $f \colon X \to Y$
has finite cohomological dimension if $Y$ is a $\Q$-stack and $f$ has affine
stabilizers and finitely presented inertia. This result is refined and
generalized in \cite{hallj_dary_alg_groups_classifying}: the condition on
inertia is not required and in positive characteristic $f$ has finite
cohomological dimension exactly when $f$ has linearly reductive stabilizers.

\begin{definition}
A morphism of algebraic stacks $f \colon X \to Y$ is \fndefn{concentrated}
if it is quasi-compact, quasi-separated, and for every quasi-compact and
quasi-separated algebraic stack $Z$ and every morphism $g \colon Z \to Y$, the
pulled back morphism $f_Z \colon X_Z \to Z$ has finite cohomological
dimension.
\end{definition}

By the result of Drinfeld and Gaitsgory, a quasi-compact and
quasi-separated morphism of algebraic stacks $f\colon X\to Y$ is concentrated
if $Y$ is a $\Q$-stack and $f$ has affine stabilizers.

The next result is immediate from Lemma \ref{lem:fcd}.
\begin{lemma}\label{lem:conc}
  Let $f \colon X \to Y$ be a $1$-morphism of algebraic stacks that is
  quasi-compact and quasi-separated.
  \begin{enumerate}
  \item\label{lem:conc:item:bc} If $f$ is concentrated, then
    it remains so after base change. 
  \item\label{lem:conc:item:flat_local} Let $g \colon Z \to Y$ be a
    $1$-morphism that is faithfully flat. If $f_Z \colon X\times_Y Z \to Z$
    is concentrated, then so is $f$.
  \item\label{lem:conc:item:rep} If $f$ is representable,
    then it is concentrated. 
  \item\label{lem:conc:item:propP} Let $h \colon Y \to W$ be a $1$-morphism
    that is concentrated. Then the composition $h\circ f \colon X \to W$ is
    concentrated if and only if $f$ is concentrated.
  \item\label{lem:conc:item:quaff_delta} Assume that $Y$ is
    quasi-compact with quasi-affine diagonal. Then $f$ is
    concentrated if and only if it has finite cohomological
    dimension.  
  \end{enumerate}
\end{lemma}

The main result of this section is the following Theorem
that refines Lemma \ref{lem:qcqs_conditions}. 
\begin{theorem}\label{thm:fcd_coprod}
  Let $f \colon X \to Y$ be a concentrated $1$-morphism of algebraic stacks
  \begin{enumerate}
  \item \label{thm:fcd_coprod:item:uni_bd} If $Y$ is quasi-compact and
    quasi-separated, then there is an integer $n$ such that the
    natural morphism
    \[
    \trunc{\geq j}\RDERF \QCPSH{f}\cplx{M} \to \trunc{\geq j}\RDERF
    \QCPSH{f}(\trunc{\geq j-n}\cplx{M})
    \]
    is a quasi-isomorphism for every integer $j$ and $\cplx{M} \in \DQCOH(X)$. 
  \item\label{thm:fcd_coprod:item:ra_coincide} The restriction of $\RDERF
    \MODPSH{f}$ to $\DQCOH(X)$ coincides with $\RDERF\QCPSH{f}$.
  \item \label{thm:fcd_coprod:item:sm_cp} The functor $\RDERF
    \QCPSH{f}$ preserves small coproducts. 
  \item\label{thm:fcd_coprod:item:fl_bc} If $g \colon Y' \to Y$ is a flat morphism of algebraic stacks, then the $2$-cartesian square:
    \[
    \xymatrix{X' \ar[r]^-{g'} \ar[d]_{f'} & X \ar[d]^f\\ Y'
      \ar[r]^{g} & Y}
    \]
    induces a natural quasi-isomorphism for every $\cplx{M}\in \DQCOH(X)$:
    \[
    \LDERF \QCPBK{g} \RDERF \QCPSH{f}\cplx{M} \homotopic \RDERF
    \QCPSH{f'} \LDERF \QCPBK{g'}\cplx{M}.
    \]
  \end{enumerate}
\end{theorem}
\begin{proof}
  For \itemref{thm:fcd_coprod:item:ra_coincide}, choose
  a diagram as in
  \eqref{eq:simplicial_construction_of_derived_functors}. By the
  natural transformations
  \eqref{eq:relationship_derived_pushforwards:mod_lisset},
  \eqref{eq:relationship_derived_pushforwards:et_lisset}, and
  \eqref{eq:two_pushforwards}, it is sufficient to prove that the
  restriction of $\RDERF(\tilde{f}^+_{\bullet,\et})_*$ to
  $\DQCOH(U_{\bullet,\et}^+)$ factors through
  $\DQCOH(V_{\bullet,\et}^+)$. This can be verified smooth locally on
  $Y$, so we may assume that $Y$ is quasi-compact and quasi-separated and $f$ has
  cohomological dimension $\leq n$ for some integer $n$. In
  particular, $\RDERF^i(\tilde{f}^+_{\bullet,\et})_*\shv{M} = 0$ for
  every $i>n$ and $\shv{M} \in \QCOH(U_{\bullet,\et}^+)$. By
  \cite[Lem.~2.1.10]{MR2434692}, for every $\cplx{M} \in
  \DQCOH(U_{\bullet,\et}^+)$ and integer $j$ the natural morphism:
  \begin{equation}
    \trunc{\geq j}\RDERF (\tilde{f}_{\bullet,\et}^+)_* \cplx{M} \to
    \trunc{\geq j}\RDERF
    (\tilde{f}_{\bullet,\et}^+)_* (\trunc{\geq j-n}\cplx{M})\label{eq:LO_fcd} 
  \end{equation}
  is an isomorphism. By
  Lemma
  \ref{lem:qcqs_conditions}\itemref{item:qcqs_conditions:global_trivial_duality}
  the result follows. Note that the equation (\ref{eq:LO_fcd}) now
  proves \itemref{thm:fcd_coprod:item:uni_bd}. Finally, the claims
  \itemref{thm:fcd_coprod:item:sm_cp} and
  \itemref{thm:fcd_coprod:item:fl_bc} follow from
  \itemref{thm:fcd_coprod:item:uni_bd} and the corresponding results
  for the bounded below category in Lemma \ref{lem:qcqs_conditions}.
\end{proof}
\begin{corollary}\label{C:aff_equiv_der}
  If $f\colon X \to Y$ is an affine morphism of algebraic stacks, then there is a natural 
  equivalence of triangulated categories
  \[
  \bar{f}^* \colon \DQCOH(Y_{\lisset},f_*\Orb_X) \to \DQCOH(X).
  \]
\end{corollary}
\begin{proof}
  Pick a diagram as in 
  \eqref{eq:simplicial_construction_of_derived_functors}. Since $f$ is
  an affine morphism, we may assume that $U_i = X\times_Y V_i$. The morphism of ringed 
  topoi $\tilde{f}_{\bullet,\et}^+ \colon
  U^+_{\bullet,\et} \to V^+_{\bullet,\et}$ factors as:
  \[
  (U^+_{\bullet,\et},\Orb_{U^+_{\bullet,\et}})
  \xrightarrow{g}
  (V^+_{\bullet,\et},(\tilde{f}_{\bullet,\et}^+)_*\Orb_{U^+_{\bullet,\et}})
  \xrightarrow{k} (V^+_{\bullet,\et}, \Orb_{V^+_{\bullet,\et}}).
  \]
  We claim that $g^{-1}(\tilde{f}_{\bullet,\et}^+)_*\Orb_{U^+_{\bullet,\et}} \to 
  \Orb_{U^+_{\bullet,\et}}$ is flat. It is sufficient to verify this upon restriction to each 
  $(U_i)_{\et}$ and then work \'etale-locally; thus, we may assume that $f$ is a 
  morphism of affine schemes $\spec B \to \spec A$. Hence, it suffices to prove that the 
  induced morphism of ringed sites  $\bar{f} \colon ((\spec B)_{\et},\Orb_{\spec 
    B_{\et}}) \to ((\spec A)_{\et}, (f_{\et})_*\Orb_{\spec B})$ is flat. This can be verified at geometric 
  points, so let $\mathfrak{p}$ be a prime 
  ideal of $B$. We must prove 
  that if $\mathfrak{q}=f(\mathfrak{p})$, then the induced ring homomorphism $B\otimes_A 
  A_{\mathfrak{q}}^{\mathrm{sh}} \to 
  B_{\mathfrak{p}}^{\mathrm{sh}}$ is flat, where $\mathrm{sh}$ denotes the strict 
  henselization at the 
  relevant prime ideal. Since 
  $\spec(B\tensor_A A_{\mathfrak{q}}^{\mathrm{sh}})$ has flat diagonal over $\spec B$ and 
  $\spec B_{\mathfrak{p}}^{\mathrm{sh}}$ is a flat $\spec B$-scheme, the assertion is clear and 
  the claim is proved.
  
  It follows that $g^* \colon \MOD 
  (V^+_{\bullet,\et},(\tilde{f}_{\bullet,\et}^+)_*\Orb_{U^+_{\bullet,\et}}) \to 
  \MOD(U^+_{\bullet,\et},\Orb_{U^+_{\bullet,\et}})$ is exact. Moreover since $f$ is affine, it has 
  cohomological dimension $\leq 0$ (Lemma \ref{lem:fcd}\itemref{lem:fcd:item:affine})
  and is concentrated. In particular, the restriction of $\RDERF g_*$ to $\DQCOH 
  (U^+_{\bullet,\et},\Orb_{U^+_{\bullet,\et}})$ factors through $\DQCOH 
  (V^+_{\bullet,\et},(\tilde{f}_{\bullet,\et}^+)_*\Orb_{U^+_{\bullet,\et}})$ and is exact. It remains 
  to prove that the restriction of the adjunctions $\ID{} \Rightarrow \RDERF g_* \LDERF g^*$ and 
  $\LDERF g^*\RDERF g_* \Rightarrow \ID{}$ to complexes with quasi-coherent cohomology 
  sheaves are isomorphisms. By the exactness, it is sufficient to prove that the restrictions 
  of the underived adjunctions $\ID{} \Rightarrow g_*g^*$ and $g^*g_* \Rightarrow \ID{}$ 
  to quasi-coherent modules are isomorphisms. Unwinding the definitions, this is just the 
  assertion that $\QCOH(X) \cong \QCOH(Y_{\lisset}, f_*\Orb_X)$. By smooth descent, we 
  may reduce to the situation where $X$ and $Y$ are affine schemes. The result now 
  follows from \cite[\S II.1.4]{EGA}.  
\end{proof}
\begin{corollary}\label{C:qaff_cons}
  Let $f\colon X \to Y$ be a quasi-affine morphism of algebraic stacks and let $M \in 
  \DQCOH(X)$. If $\RDERF \QCPSH{f}M \simeq 0$, then $M \simeq 0$, that is, $\RDERF 
  \QCPSH{f}$ is conservative.
\end{corollary}
\begin{proof}
  We may factor $f$ as $X \xrightarrow{j} X' \xrightarrow{f'} Y$, where $j$ is a 
  quasi-compact open immersion and $f'$ is an affine morphism. By Corollary 
  \ref{C:aff_equiv_der}, the result is true for $f'$. Hence, we are reduced to the situation 
  where $f$ is a quasi-compact open immersion. In this case, however, $\Delta_{X/Y} \colon 
  X \to X\times_Y X$ is an isomorphism. By Theorem 
  \ref{thm:fcd_coprod}\itemref{thm:fcd_coprod:item:fl_bc} $\LDERF 
  \QCPBK{f}\RDERF \QCPSH{f}M \simeq M$ and the result follows.
\end{proof}

\section{Triangulated categories}\label{sec:tri_cats}
In this section we will recall some results on triangulated categories
that may not be familiar to everyone. For excellent and comprehensive
treatments of these topics see~\cite{MR1191736} and
\cite[\S2]{MR1436741}.
In particular, we will recall \emph{thick} and \emph{localizing} triangulated
subcategories. This leads to the concept of \emph{compact} objects and Thomason's
localization theorem.

Throughout this section, let $\mathscript{S}$ be a triangulated category with shift operator
$\Sigma$.

A functor $F \colon \mathscript{S} \to \mathscript{S}'$ between triangulated
categories is \fndefn{triangulated} if $F$ sends triangles to
triangles and is compatible with shifts. We say that a full
subcategory $\mathscript{R}$ of $\mathscript{S}$ is \fndefn{triangulated} if
the category $\mathscript{R}$ is triangulated and the inclusion functor
$\mathscript{R} \to \mathscript{S}$ is triangulated. A subcategory
$\mathscript{R} \subset \mathscript{S}$ is \fndefn{thick} (also known as
\'epaisse or saturated) if it is full, triangulated, and every
$\mathscript{S}$-direct summand of every $r\in \mathscript{R}$ belongs to
$\mathscript{R}$.
\begin{example}
  For a triangulated functor $F \colon \mathscript{S} \to \mathscript{S}'$ we
  denote by $\ker F$ the full triangulated subcategory consisting of
  those $x\in \mathscript{S}$ such that $F(x) \homotopic 0$. The
  subcategory $\ker F \subset \mathscript{S}$ is thick.
\end{example}
\begin{example}\label{ex:thick_2_3}
  Given triangulated subcategories $\mathscript{R}_1 \subset
  \mathscript{R}_2 \subset \mathscript{S}$ such that $\mathscript{R}_1$ is a
  thick subcategory of $\mathscript{S}$, then
  $\mathscript{R}_1$ is a thick subcategory $\mathscript{R}_2$.
\end{example}
Kernels of triangulated functors produce
essentially all thick subcategories \cite[\S1.3]{MR1436741}. Indeed,
for every thick subcategory $\mathscript{R} \subset \mathscript{S}$ there is a
\fndefn{quotient} functor $Q \colon \mathscript{S} \to  
\mathscript{S}/\mathscript{R}$ such that $\mathscript{R} \cong \ker
Q$, $Q$ is essentially surjective, and satisfies a universal property \cite[Thm.~2.1.8]{MR1812507}.

For a class $R \subset \mathscript{S}$ the \fndefn{thick 
  closure} of $R$ is the smallest thick 
subcategory $\mathscript{R} \subset \mathscript{S}$ containing
$R$. A subcategory $\mathscript{R} \subset
\mathscript{S}$ is \fndefn{dense} if it is full, triangulated, and its
thick closure coincides with $\mathscript{S}$.

If the triangulated category $\mathscript{S}$ is essentially
small, then there is a notion of $\KGRP(\mathscript{S})$: it is
the free abelian group on the \emph{set} of isomorphism classes of
objects in $\mathscript{S}$ modulo the relation that given an
$\mathscript{S}$-triangle $s_1 \to s_2 \to s_3$, then $[s_2] = [s_1] +
[s_3]$. It is easy to see that for $s$, $t\in \mathscript{S}$ then
$[s\oplus t] = [s] + [t]$ and $[s] = -[\Sigma s]$. Also, for every
$\sigma \in \KGRP(\mathscript{S})$, there exists $s\in \mathscript{S}$
such that $\sigma = [s]$. Given a triangulated functor $F \colon
\mathscript{S} \to \mathscript{S}'$ between essentially small triangulated
categories, there is an induced group homomorphism $\KGRP(F)\colon
\KGRP(\mathscript{S}) \to \KGRP(\mathscript{S}')$. The following is a
nice result of A. Neeman \cite[Cor.\ 0.10]{MR1191736} (also see \cite[5.2.2]{MR1106918}
and \cite[Lem.\ 2.2]{MR1436741}).  
\begin{lemma}\label{lem:neeman_k_grps}
  Let $\mathscript{S}$ be an essentially small triangulated category and
  let $\mathscript{R}$ be a \emph{dense} subcategory. If $s\in \mathscript{S}$, then
  $s \in \mathscript{R}$ if and only if $s$ belongs to the image of
  $\KGRP(\mathscript{R})$ in $\KGRP(\mathscript{S})$. In particular, if $s\in
  \mathscript{S}$, then $s\oplus \Sigma s \in \mathscript{R}$. 
\end{lemma}
A pair of triangulated functors $\mathscript{R} \to
\mathscript{S} \xrightarrow{F} {\mathscript{T}}$ is \fndefn{left-exact}
(resp.\ \fndefn{almost exact}, resp.\ \fndefn{exact}) if 
$\mathscript{R}$ is a thick subcategory of $\mathscript{S}$ via the functor $\mathscript{R} \to \mathscript{S}$ and the
functor $F \colon \mathscript{S} \to {\mathscript{T}}$ factors through the
quotient $\bar{F} \colon \mathscript{S}/\mathscript{R} \to {\mathscript{T}}$ and
this functor is fully faithful (resp.\ fully faithful and dense, resp.\ an
equivalence). The following (well-known) Lemma will be useful.
\begin{lemma}\label{lem:identify_quots}
  Let $F \colon \mathscript{S} \to
  \mathscript{T}$ be a triangulated functor. If $F$ has a right
  adjoint $G \colon \mathscript{T} \to \mathscript{S}$ such that the adjunction
  $\epsilon\colon FG \to \ID{\mathscript{T}}$ is an isomorphism, then the sequence $\ker
  F \to \mathscript{S} \xrightarrow{F} \mathscript{T}$ is
  exact.
\end{lemma}
\begin{proof}
  We will show that $F$ satisfies the universal property of the
  quotient. Let $P \colon \mathscript{S} \to \mathscript{P}$ be a
  triangulated functor such that $\ker F\subset \ker P$.
  We must prove that there is a functor
  $P' \colon \mathscript{T} \to \mathscript{P}$ and an isomorphism
  $\alpha \colon P\simeq P'F$ unique up to unique isomorphism.
  For the uniqueness, let $P'_1$, $P'_2$ be two such functors with isomorphisms
  $\alpha_i \colon P\simeq P'_iF$. Then the isomorphism
  $\alpha_2\alpha_1^{-1}\colon P'_1F\simeq P'_2F$ induces a unique isomorphism
  $P'_1\simeq P'_1FG\simeq P'_2FG\simeq P'_2$ compatible with the $\alpha_i$.
  For the existence, set $P' = PG$ and let $\alpha=P_*\eta$, where
  $\eta\colon \ID{\mathscript{S}}\to GF$ is the unit of the adjunction. Now
  $F_*\eta$ is an isomorphism since $\epsilon$ is an isomorphism,
  so $\alpha=P_*\eta$ is an
  isomorphism since $\ker F\subset \ker P$.
\end{proof}
A triangulated category is said to be \fndefn{closed under small
  coproducts} if 
it admits small categorical coproducts and small coproducts of
triangles remain triangles. If the triangulated category $\mathscript{S}$
is closed under small coproducts, then we say that a
subcategory $\mathscript{R} \subset \mathscript{S}$ is \fndefn{localizing}
if it is a full triangulated subcategory, closed under small coproducts,
and the functor $\mathscript{R} \to \mathscript{S}$ preserves small 
coproducts. For a class 
$R \subset \mathscript{S}$, where $\mathscript{S}$ 
is closed under small coproducts, there is a smallest
subcategory $\mathscript{R}\subset \mathscript{S}$ that is localizing and
contains $R$. We refer to $\mathscript{R}$ as the \fndefn{localizing envelope}
of $R$.  
\begin{example}
  If a subcategory $\mathscript{R} \subset \mathscript{S}$ is localizing,
  then it is thick. Indeed, given $r\in \mathscript{R}$ and $r\homotopic
  r'\oplus 
  r''$ in $\mathscript{S}$, the Eilenberg swindle produces an
  $\mathscript{S}$-isomorphism: 
  \[
  r'' \oplus r \oplus r \oplus \cdots \homotopic r \oplus r
  \oplus r \oplus \cdots.
  \]
  Since $\mathscript{R}$ is localizing, $r'' \oplus r\oplus r
  \oplus \cdots \in \mathscript{R}$. The cone of the natural morphism
  $r\oplus r \oplus \cdots \to r'' \oplus r \oplus r \oplus \cdots$ is
  $r''$; thus $r'' \in \mathscript{R}$. 
\end{example}
A result of A. Neeman \cite[Prop.~1.9]{MR1191736} says that if a
subcategory $\mathscript{R} \subset \mathscript{S}$ is localizing, then the
quotient $Q\colon \mathscript{S} \to \mathscript{S}/\mathscript{R}$ preserves small coproducts. In
particular, since the quotient is essentially surjective, the
category $\mathscript{S}/\mathscript{R}$ is closed under small coproducts. 
\begin{example}
  If $F \colon \mathscript{S} \to \mathscript{S}'$ is a triangulated functor that
  preserves small coproducts, then the subcategory $\ker F$ is localizing. 
\end{example}
An object $s\in \mathscript{S}$ is \fndefn{compact} if the functor
$\Hom_{\mathscript{S}}(s,-)$ preserves small coproducts. Denote by
$\mathscript{S}^c$ the 
full subcategory of compact objects of $\mathscript{S}$. 
\begin{example}\label{ex:compact_ring}
  Let $A$ be a ring. 
  A complex of $A$-modules is compact in $\DCAT(A)$ if and only if it
  is quasi-isomorphic to a bounded complex of projective $A$-modules
  \cite[\spref{07LT}]{stacks-project}.
  That is, the compact objects of $\DCAT(A)$ are the \fndefn{perfect} complexes of $A$-modules. 
\end{example}
\begin{example}\label{ex:pres_cpt_adj}
  Let $F \colon \mathscript{S} \to \mathscript{S}'$ be a triangulated functor
  that admits a right adjoint $G \colon \mathscript{S}' \to \mathscript{S}$. If
  $G$ preserves small coproducts, then $F$ sends $\mathscript{S}^c$ to
  $\mathscript{S}'^c$ \cite[Thm.~5.1~``$\Rightarrow$'']{MR1308405}. 
\end{example}
\begin{example}\label{ex:pres_cpt_conc}
  If $f\colon X' \to X$ is a concentrated morphism of algebraic stacks, then $\LDERF 
  \QCPBK{f}$ sends $\DQCOH(X)^c$ to $\DQCOH(X')^c$. This follows by combining Example \ref{ex:pres_cpt_adj} with Theorem \ref{thm:fcd_coprod}\itemref{thm:fcd_coprod:item:sm_cp}.
\end{example}
A class $S \subset
\mathscript{S}$ is \fndefn{generating} if given $x\in \mathscript{S}$ such
that $\Hom_{\mathscript{S}}(\Sigma^ns,x) = 0$ for all $s\in S$ and $n\in \Z$,
then $x\homotopic 0$. The triangulated category $\mathscript{S}$ is
\fndefn{compactly generated} if it admits a \emph{set} of generators
consisting of compact objects.
\begin{example}\label{ex:compact_ring_gen}
  Let $A$ be a ring. Denote the unbounded derived category of
  $A$-modules by $\DCAT(A)$. Then the set $\{A\}$ compactly generates
  $\DCAT(A)$. Hence
  $\DCAT(A)$ is compactly generated.  
\end{example}
\begin{example}\label{ex:compact_cons_gen}
  This is a refinement of Example \ref{ex:pres_cpt_adj}. Let $F \colon \mathscript{S} \to \mathscript{S}'$ be a triangulated functor
  that admits a right adjoint $G \colon \mathscript{S}' \to \mathscript{S}$ that preserves 
  small coproducts. In addition, assume that 
  $G$ is conservative (i.e., $G(x)\homotopic 0$ implies $x\homotopic 0$). If $\mathscript{S}$ is compactly generated by a class $S$, then $\mathscript{S}'$ is compactly generated by the class $\{F(s) \suchthat s \in S\}$. 
\end{example}

We now recall Thomason's Localization Theorem, which was
proved in this generality by A. Neeman \cite[Thm. 2.1]{MR1191736,MR1308405} (also see \cite[5.1]{MR1106918}).
\begin{theorem}[Thomason's Localization]\label{thm:thomason}
  Consider an exact sequence of
  triangulated categories $\mathscript{R}
  \to \mathscript{S} \xrightarrow{F} \mathscript{T}$ that are
  closed under small 
  coproducts. If the triangulated category $\mathscript{S}$ is compactly
  generated and $\mathscript{R}$ is the localizing envelope of a subset $R
  \subset \mathscript{S}^c$, then there is an induced sequence 
  $\mathscript{R}^c \to \mathscript{S}^c \to \mathscript{T}^c$ which is almost
  exact. In particular, $\mathscript{R}^c
  = \mathscript{S}^c \cap \mathscript{R}$ and $\mathscript{R}^c$ is the thick
  closure of $R$. 
\end{theorem}
Combining Theorem \ref{thm:thomason} with the elementary Lemma
\ref{lem:neeman_k_grps} produces something very surprising, which was
observed by A. Neeman \cite[Cor. 0.9]{MR1191736}.
\begin{corollary}\label{cor:thomason_lift}
  In Theorem \ref{thm:thomason} assume that the category
  $\mathscript{T}^c$ is essentially small. Then for every $t\in
  \mathscript{T}^c$, there exists an $s\in \mathscript{S}^c$ and an isomorphism
  $t\oplus \Sigma t \homotopic F(s)$. 
\end{corollary}
Another useful Corollary is the following \cite[Thm.~2.1.2]{MR1308405}.
\begin{corollary}\label{cor:thomason_gens}
  In Theorem \ref{thm:thomason} suppose that $R$ is a generating set
  for $\mathscript{S}$, then $\mathscript{R} = \mathscript{S}$. 
\end{corollary}
\section{Perfect complexes, projection formulas, and finite duality}
We now use the results of the previous section to prove some useful results 
for derived categories of algebraic stacks. We begin with the notion of a perfect complex on
an algebraic stack.

\subsection{Perfect complexes}
We recall some notions from \cite[Exp.~II]{MR0354655} (also see \cite[Tag \spref{08FK}]{stacks-project}). If $A$ is a ring, then a complex of $A$-modules $P$ is \emph{strictly perfect} 
if it is a bounded complex of finitely generated and projective $A$-modules. More 
generally, if $A$ is a sheaf of rings on a site $E$, then a complex $P=(P^k)$
of 
$A$-modules is \emph{strictly perfect} if it is a bounded complex and
each term $P^k$ is 
a direct summand of a finite free $A$-module. A complex $P \in \DCAT(A)$ is \emph{perfect} if it is locally strictly perfect.

Let $X$ be an algebraic stack. Then a complex $\cplx{P}$ on $X$ is 
\fndefn{perfect} if it is a perfect object of $\DCAT(X)$. Note that all perfect complexes on $X$ belong to $\DQCOH(X)$. The following lemma provides a useful criterion for perfection.
\begin{lemma}\label{L:perf_wd}
  Let $X$ be an algebraic stack and let $\cplx{P} \in \DQCOH(X)$. The following conditions 
  are equivalent.
  \begin{enumerate}
  \item $\cplx{P}$ is perfect; and
  \item for every $x\in |X|$,
    there exists a flat morphism $f\colon \spec A \to X$ with image
    containing $x$ such that
    $\RDERF\Gamma(\spec A,\LDERF \QCPBK{f}\cplx{P})$ is a strictly perfect
    complex of $A$-modules.
  \end{enumerate}
  In particular, every perfect complex on an affine scheme is strictly perfect.
\end{lemma}
\begin{proof}
  It is sufficient to prove that if $A \to B$ is a faithfully flat ring homomorphism and $M \in 
  \DCAT(A)$,
  then $M$ is a perfect complex of $A$-modules if and only if $M\tensor_A B$ is a perfect complex of 
  $B$-modules. This is \cite[Tag \spref{068T}]{stacks-project}.
\end{proof}
\begin{example}
  Let $X$ be an algebraic stack.
  Then every flat $\Orb_X$-module of finite presentation 
  defines a perfect complex in $\DQCOH(X)$. 
  In particular, if $f \colon X \to Y$ is a morphism of algebraic stacks
  that is finite, flat, and of finite presentation, then
  $f_*\Orb_X$ is perfect in $\DQCOH(Y)$. 
\end{example}
We recall the following well-known definition (e.g., 
\cite[Defn.~4.7.1]{brandenburg_thesis} or \cite[Defn.~3.3]{MR2669705}). Let $X$ be an 
algebraic stack. An object $\cplx{P} \in \DQCOH(X)$ is \emph{dualizable} if there exists 
$\cplx{P}^* \in \DQCOH(X)$ together with morphisms $e\colon \cplx{P}^* 
\tensor^{\LDERF}_{\Orb_X} \cplx{P} \to \Orb_X$, $c\colon \Orb_X \to \cplx{P} 
\tensor^{\LDERF}_{\Orb_X} \cplx{P}^*$ such that the two induced maps:
\[
  \xymatrixrowsep{0.5pc}%
  \xymatrix{
    \Orb_X \tensor^{\LDERF}_{\Orb_X} \cplx{P} \ar[r]^-{c \tensor \ID{\cplx{P}}} & \cplx{P} \tensor^{\LDERF}_{\Orb_X} \cplx{P}^* \tensor^{\LDERF}_{\Orb_X} \cplx{P} \ar[r]^-{\ID{\cplx{P}} \tensor e} & \cplx{P}\\
        \cplx{P}^*\tensor^{\LDERF}_{\Orb_X} \Orb_X \ar[r]^-{\ID{\cplx{P}^*} \tensor c} & \cplx{P}^* \tensor^{\LDERF}_{\Orb_X} \cplx{P} \tensor^{\LDERF}_{\Orb_X} \cplx{P}^* \ar[r]^-{e\tensor \ID{\cplx{P}^*}} & \cplx{P}^*
    }
\]
are isomorphisms. It is standard that $\cplx{P}$ dualizable implies that $\cplx{P}^* 
\simeq \SRHom_{\Orb_X}^{\qcsubscript}(\cplx{P},\Orb_X)$ and the $e$ and $c$ maps are 
simply those arising from the adjunction between $-\tensor^{\LDERF}_{\Orb_X} \cplx{P}$ and 
$\SRHom_{\Orb_X}^{\qcsubscript}(\cplx{P},-)$.

The following Lemma is straightforward but crucial.
\begin{lemma}\label{L:basic_props_perfect}
  Let $X$ be an algebraic stack and let $\cplx{P} \in \DQCOH(X)$ be a perfect
  complex. 
  \begin{enumerate}
  \item\label{L:basic_props_perfect:dbl} The double duality morphism $\cplx{P} \to 
    \SRHom_{\Orb_X}(\SRHom_{\Orb_X}(\cplx{P},\Orb_X),\Orb_X)$ is a quasi-isomorphism.
  \item\label{L:basic_props_perfect:dual} The restriction of the 
    functor $\SRHom_{\Orb_X}(\cplx{P},-)\colon
    \DCAT(X) \to \DCAT(X)$ to $\DQCOH(X)$ factors through $\DQCOH(X)$
    and preserves small coproducts in $\DQCOH(X)$. Moreover, if
    $\cplx{M}\in\DQCOH(X)$, then there is a natural 
    quasi-isomorphism:
    \[
    \SRHom_{\Orb_X}(\cplx{P},\Orb_X) \tensor^{\LDERF}_{\Orb_X} \cplx{M} \simeq \SRHom_{\Orb_X}(\cplx{P},\cplx{M}).
    \]
    In particular, $\SRHom^{\qcsubscript}_{\Orb_X}(\cplx{P},-) \simeq \SRHom_{\Orb_X}(\cplx{P},-) \simeq \cplx{P}^*\otimes -$ on $\DQCOH(X)$.
  \item\label{L:basic_props_perfect:dp} If $\cplx{P} \in \DQCOH(X)$ is dualizable, then it is 
    perfect.
  \end{enumerate}
  In particular for every algebraic stack $X$, the notions of perfect and dualizable objects 
  in $\DQCOH(X)$ coincide. 
\end{lemma}
\begin{proof}
  To prove \itemref{L:basic_props_perfect:dbl}, we note that the existence of the double 
  duality morphism follows from $\SRHom$-$\tensor$ adjunction. That it is a 
  quasi-isomorphism can be verified smooth-locally on $X$, so we may assume that 
  $X=\spec A$ is an affine scheme. To prove \itemref{L:basic_props_perfect:dual}, we 
  may argue similarly to reduce to the affine setting. Now the collection
  of all $\cplx{P}$ that satisfy the conclusions of \itemref{L:basic_props_perfect:dbl} and \itemref{L:basic_props_perfect:dual} is closed
  under finite coproducts, direct summands, shifts, and the taking of
  cones, that is, it is a thick subcategory of $\DQCOH(X) \simeq \DCAT(A)$. Since $\Orb_X$ 
  satisfies the conclusions of \itemref{L:basic_props_perfect:dbl} and 
\itemref{L:basic_props_perfect:dual} and $X$ is affine, 
\itemref{L:basic_props_perfect:dbl} and \itemref{L:basic_props_perfect:dual} follow.

For \itemref{L:basic_props_perfect:dp}: since dualizable implies locally dualizable, it follows that we may assume that $X$ is affine. In this case, the result is classical. One may also argue as follows: $\cplx{Q}$ being dualizable implies that there is an isomorphism of functors $\SRHom^{\mathrm{qc}}_{\Orb_X}(\cplx{Q},-) \simeq \cplx{Q}^* \tensor^{\LDERF}_{\Orb_X} - $ from $\DQCOH(X)$ to $\DQCOH(X)$. In particular, the functor $\SRHom^{\mathrm{qc}}_{\Orb_X}(\cplx{Q},-) \colon \DQCOH(X) \to \DQCOH(X)$ preserves small coproducts. Taking global sections, we see that $\Hom_{\Orb_X}(\cplx{Q},-)$ preserves small coproducts. Hence, $\cplx{Q}$ is compact in $\DQCOH(X)$. But $X$ is affine, so $\cplx{Q}$ is perfect. 
\end{proof}
\subsection{Compact complexes}
We now move on to a description of the compact objects of $\DQCOH(X)$ and their relationship to perfect complexes.
\begin{lemma}\label{L:compact_conc}
  Let $X$ be a quasi-separated algebraic stack. 
  \begin{enumerate}
  \item \label{L:compact_conc:c_perf} If $\cplx{Q} \in \DQCOH(X)^c$, then $\cplx{Q}$ is 
    perfect. 
  \item\label{L:compact_conc:tens_c_p} If $\cplx{P}$ is perfect on $X$ and $\cplx{Q} \in \DQCOH(X)^c$, then $\cplx{Q}\tensor^{\LDERF}_{\Orb_X} \cplx{P} \in \DQCOH(X)^c$.
  \item \label{L:compact_conc:perf_conc_c} If $X$ is concentrated, then $\Orb_X \in \DQCOH(X)^c$; in particular, if $\cplx{P}$ is 
    perfect on $X$, then $\cplx{P} \in \DQCOH(X)^c$.
  \end{enumerate}
\end{lemma}
\begin{proof}
  We first prove \itemref{L:compact_conc:c_perf}. Consider a 
  smooth morphism $p\colon \spec A \to X$. It follows from the quasi-separatedness of $X$ that 
  $p$ is quasi-compact, quasi-separated, and representable. In particular, $p$ is concentrated 
  (Lemma \ref{lem:conc}\itemref{lem:conc:item:rep}) and so
  $\LDERF \QCPBK{p}\cplx{Q} \in \DQCOH(\spec A)^c$ by
  Example \ref{ex:pres_cpt_conc}.
  The claim now follows from Example \ref{ex:compact_ring} and Lemma 
  \ref{L:perf_wd}.

  To prove \itemref{L:compact_conc:tens_c_p}, we note that if $\cplx{M} \in \DCAT(X)$, 
  then \eqref{eq:ghomadj} implies
  \[
  \Hom_{\Orb_X}(\cplx{Q}\tensor^{\LDERF}_{\Orb_X} \cplx{P},\cplx{M}) \simeq \Hom_{\Orb_X}(\cplx{Q},\SRHom_{\Orb_X}(\cplx{P},\cplx{M})).
  \]
  By Lemma \ref{L:basic_props_perfect}\itemref{L:basic_props_perfect:dual},  the 
  restriction of $\SRHom(\cplx{P},-)$ to $\DQCOH(X)$ preserves small coproducts. Since 
  $\cplx{Q} \in \DQCOH(X)^c$, it follows that the restriction of 
  $\SRHom_{\Orb_X}(\cplx{Q}\tensor^{\LDERF}_{\Orb_X}\cplx{P},-)$ to $\DQCOH(X)$ 
  preserves small coproducts. Hence, $\cplx{Q} \tensor^{\LDERF}_{\Orb_X} \cplx{P} \in 
  \DQCOH(X)^c$.
  
  For \itemref{L:compact_conc:perf_conc_c}, we note that if $\cplx{M} \in \DCAT(X)$, 
  then by definition $\RHom_{\Orb_X}(\Orb_X,\cplx{M}) = \RDERF 
  \Gamma(X_{\lisset},\cplx{M})$. Since $X$ is concentrated, Theorem 
  \ref{thm:fcd_coprod}\itemref{thm:fcd_coprod:item:sm_cp} implies that the restriction of 
  $\RDERF \Gamma(X_{\lisset},-)$ to $\DQCOH(X)$ preserves small coproducts. Thus, 
  $\Orb_X \in \DQCOH(X)^c$. The latter claim follows from the former and 
  \itemref{L:compact_conc:tens_c_p}.
\end{proof}

In the following Lemma we provide criteria for a perfect complex on an algebraic stack to be compact. We
have not seen this characterization in the literature before. 
\begin{lemma}\label{lem:char_compact_stack}
  Let $X$ be a quasi-compact and quasi-separated algebraic stack and
  let $\cplx{P} \in \DQCOH(X)$ be a perfect complex. The following conditions are
  equivalent:
  \begin{enumerate}
  \item\label{item:char_compact_stack:base} $\cplx{P}$ is a compact
    object of $\DQCOH(X)$;
  \item\label{item:char_compact_stack:bddext} there exists an integer
    $r\geq 0$ such that $\Hom_{\Orb_X}(\cplx{P},\shv{N}[i])=0$ for all
    $\shv{N}\in\QCOH(X)$ and $i>r$; and
  \item\label{item:char_compact_stack:stabletrunc} there exists an
    integer $r\geq 0$ such that the natural map
    \[
    \tau^{\geq j}\RHom_{\Orb_X}(\cplx{P},\cplx{M}) \to \tau^{\geq
      j}\RHom_{\Orb_X}(\cplx{P},\tau^{\geq j-r}\cplx{M}) 
    \]
    is a quasi-isomorphism for all $\cplx{M} \in \DQCOH(X)$ and integers $j$.
  \end{enumerate}
\end{lemma}
\begin{proof}
  Assume that \itemref{item:char_compact_stack:bddext} does not
  hold. Then there is an infinite sequence of quasi-coherent
  $\Orb_X$-modules $\shv{M}_1$, $\shv{M}_2$, $\dots$ and strictly increasing
  sequence of integers $d_1 < d_2 < \dots$ such that
  $\Hom_{\Orb_X}(\cplx{P},\cplx{M}_i[d_i]) \neq 0$ for every $i$. Since $\DQCOH(X)$ is left-complete
  \cite[Thm.~B.1]{hallj_neeman_dary_no_compacts}, there is a
  quasi-isomorphism in $\DQCOH(X)$:
  \[
  \bigoplus_{i=1}^\infty \shv{M}_i[d_i] \homotopic \prod_{i=1}^\infty \shv{M}_i[d_i].
  \]
  This implies that the natural morphism
  \[
  \bigoplus_{i=1}^\infty \Hom_{\Orb_X}(\cplx{P},\shv{M}_i[d_i]) \to
  \Hom_{\Orb_X}(\cplx{P},\bigoplus_{i=1}^{\infty}\shv{M}_i[d_i])
  \homotopic \prod_{i=1}^{\infty}
  \Hom_{\Orb_X}(\cplx{P},\shv{M}_i[d_i]) 
  \]
  is not an isomorphism. In particular, $\cplx{P}$ is not
  compact. Thus, by the contrapositive, we have proved the implication
  \itemref{item:char_compact_stack:base}$\Rightarrow$\itemref{item:char_compact_stack:bddext}.

  For
  \itemref{item:char_compact_stack:bddext}$\Rightarrow$\itemref{item:char_compact_stack:stabletrunc}, first choose a diagram as in
  \eqref{eq:simplicial_construction_of_derived_functors} with $V=Y=\spec \Z$.
  Let $\cplx{P}^+_{\bullet,\et}=\RDERF (\REST[{U_\bullet}])_*\LDERF (p_{\bullet,\lisset}^+)^*\cplx{P}$
    and $\cplx{M}^+_{\bullet,\et}=\RDERF (\REST[{U_\bullet}])_*\LDERF (p_{\bullet,\lisset}^+)^*\cplx{M}$.
  Note that
  \[
  \RDERF (\REST[{U_\bullet}])_* \LDERF
  (p_{\bullet,\lisset}^+)^*\SRHom_{\Orb_X}(\cplx{P},\cplx{M})\homotopic \SRHom_{\Orb_{U^+_{\bullet,\et}}}(\cplx{P}^+_{\bullet,\et},\cplx{M}^+_{\bullet,\et})
  \]
  has quasi-coherent cohomology (Lemma~\ref{L:basic_props_perfect}) and that
  $\RHom_{\Orb_{U^+_{\bullet,\et}}}(\cplx{P}^+_{\bullet,\et},\shv{N}[i])=0$
  for all $i>r$ and $\shv{N}\in\QCOH(U^+_{\bullet,\et})$. It is enough to prove that
  \[
  \tau^{\geq j}\RHom_{\Orb_{U^+}}(\cplx{P}^+_{\bullet,\et},
    \cplx{M}) \to \tau^{\geq
    j}\RHom_{\Orb_{U^+}}(\cplx{P}^+_{\bullet,\et},\tau^{\geq j-r}\cplx{M}) 
  \]
  for every integer $j$ and $\cplx{M}\in\DQCOH(U^+_{\bullet,\et})$.
  This follows as in the proof of
  \cite[Lem.~2.1.10]{MR2434692} with $\epsilon_*$
  replaced by $\Hom_{\Orb_{U^+_{\bullet,\et}}}(\cplx{P}^+_{\bullet,\et},-)$. 

  Finally, for
  \itemref{item:char_compact_stack:stabletrunc}$\Rightarrow$\itemref{item:char_compact_stack:base},
  we note that because $\cplx{P}$ is perfect, it is dualizable. Thus, for
  all $\cplx{M} \in \DQCOH(X)$ and integers $j$, there are natural
  quasi-isomorphisms: 
  \begin{align*}
  \trunc{\geq j}\RHom_{\Orb_X}(\cplx{P},\cplx{M}) &\homotopic
  \trunc{\geq j}\RHom_{\Orb_X}(\cplx{P},\trunc{\geq j-r}\cplx{M}) \\
   &\homotopic \trunc{\geq j}\RDERF\Gamma(X_\lisset,\cplx{P}^\vee
  \tensor_{\Orb_X}^{\LDERF} \trunc{\geq j-r}\cplx{M}). 
  \end{align*}
  Also since $\cplx{P}$ is perfect, the restriction of
  the functor
  $\cplx{P}^\vee \tensor_{\Orb_X}^{\LDERF}(-) $ to $\DQCOH^{\geq j-r}(X)$
  factors through $\DQCOH^{\geq j-r'}(X)$ for some fixed integer
  $r'$. The result now follows from Lemma
  \ref{lem:qcqs_conditions}\itemref{item:qcqs_conditions:bdd_coprod}.
\end{proof}
For future reference, we summarize the situation in the following remark.
\begin{remark}\label{rem:fin_coh_dim}
Let $X$ be a quasi-compact and quasi-separated algebraic stack. Then the following are equivalent:
\begin{enumerate}
\item every perfect object of $\DQCOH(X)$ is compact;
\item the structure sheaf $\Orb_X$ is compact;
\item $X$ has finite cohomological dimension; and
\item the derived global section functor $\RDERF\Gamma\colon \DQCOH(X)\to
  \DCAT(\AB)$ commutes with small coproducts.
\end{enumerate}
The equivalence of the first two conditions follows from Lemma \ref{L:compact_conc}. The structure sheaf is compact, if and only if $X$
has finite cohomological dimension (Lemma~\ref{lem:char_compact_stack}). The
last condition is equivalent to the definition of $\Orb_X$ being compact.
\end{remark}
\subsection{Supports and generation}
Let $X$ be a scheme and let $M \in \QCOH(X)$. The \emph{support} of $M$ is the subset
\[
\supp(M) = \{x\in |X| \suchthat M_x \neq 0\}.
\]
More generally, if $X$ is an algebraic stack and $M \in \QCOH(X)$, then we define the \emph{support} of $M$ as follows: let $p\colon U \to X$ be a smooth surjection from a scheme $U$; then $\supp(M) = p(\supp(p^*M))$. It is easily verified that this is well-defined (i.e., independent of the cover $p$). 
\begin{definition}
  Let $X$ be an algebraic stack and let $E \in \DQCOH(X)$. Define the \emph{cohomological support} of $E$ to be the subset 
  \[
  \supph(E) = \cup_{n\in \Z} \supp(\COHO{n}(E)) \subseteq |X|.
  \]
\end{definition}
Recall that if $X$ is a quasi-separated algebraic stack, then for
every $x \in |X|$, there is a quasi-affine monomorphism $i_x \colon
\mathcal{G}_x \hookrightarrow X$ with image $x$, where $\mathcal{G}_x$
is a gerbe over a field $\kappa(x)$, the residue field at $x$
\cite[Thm.~B.2]{MR2774654}. We refer to this data as the
\emph{residual gerbe at $x$}. 
We now have the following lemma.
\begin{lemma}\label{L:supports_perfect_complexes_properties}
  Let $X$ be a quasi-compact and quasi-separated algebraic stack. Let
  $E$ be a perfect complex on $X$.
  \begin{enumerate}
  \item\label{L:supports_perfect_complexes_properties:points} Then
    $x\in \supph(E)$ if and only if the complex $\LDERF \QCPBK{(i_x)}E$
    is not acyclic in $\DQCOH(\mathcal{G}_x)$, where $i_x \colon
    \mathcal{G}_x \to X$ is the residual gerbe.
  \item \label{L:supports_perfect_complexes_properties:pb} If $f\colon
    X' \to X$ is a quasi-compact and quasi-separated morphism, then
    $\supph(\LDERF \QCPBK{f}E) = f^{-1}\supph(E)$.
  \item \label{L:supports_perfect_complexes_properties:cons}
    $\supph(E)$ is closed with quasi-compact complement.
  \end{enumerate}
\end{lemma}
\begin{proof}
  For \itemref{L:supports_perfect_complexes_properties:points}, then
  $x\in \supph(E)$, if and only if there exists a smooth and
  surjective morphism $p\colon U \to X$, where $U$ is a scheme, and
  $u\in U$ such that $p(u) = x$ and $(\LDERF \QCPBK{p}E)_u$ is not
  acyclic in $\DQCOH(\Orb_{U,u})$. By \cite[Lem.~3.3(a)]{MR1436741},
  this is equivalent to $(\LDERF \QCPBK{p}E)\tensor \kappa(u)$ not
  being acyclic in $\DQCOH(\kappa(u))$. Since $\spec \kappa(u) \to
  \mathcal{G}_x$ is faithfully flat, the result follows.

  For \itemref{L:supports_perfect_complexes_properties:pb}, we use
  \itemref{L:supports_perfect_complexes_properties:points} and argue
  as in \cite[Lem.~3.3(b)]{MR1436741}, with the additional observation
  that for every $x'\in |X'|$, the induced morphism on residual gerbes
  $\mathcal{G}_{x'} \to \mathcal{G}_{f(x')}$ is faithfully flat.

  For \itemref{L:supports_perfect_complexes_properties:cons}, by
  \itemref{L:supports_perfect_complexes_properties:pb}, the conclusion
  may be verified smooth-locally. In particular, we may assume that
  $X$ is an affine scheme, and the result follows from
  \cite[Lem.~3.3(c)]{MR1436741}.
\end{proof}
If $X$ is an algebraic stack and $Z\subseteq |X|$, then we define
\[
\DQCOH[,|Z|](X) =
\{ \cplx{M} \in \DQCOH(X) \suchthat \supph(\cplx{M})\subseteq Z \}.
\]
If $j\colon U \hookrightarrow X$ is a
flat monomorphism, e.g., an open immersion, then
\[
\DQCOH[,|X\setminus U|](X) = \{ \cplx{M} \in
\DQCOH(X) \suchthat j^*\cplx{M} \homotopic 0\}.
\]
\begin{lemma}\label{L:local-generator}
  Let $X$ be an algebraic stack and let $\cplx{P}\in \DQCOH(X)$ be a perfect
  complex with support $Z=\supph(\cplx{P})$. If $\cplx{M}\in \DQCOH[,|Z|](X)$,
  then $\cplx{M}\homotopic 0$ if and only if
  $\SRHom_{\Orb_X}(\cplx{P},\cplx{M}) \homotopic 0$.
\end{lemma}
\begin{proof}
  The question is local on $X$ for the smooth topology, so we may
  assume that $X$ is an affine scheme.
  By \cite[Prop.~6.1]{MR1214458}, there is a perfect generating complex
  $K\in\DQCOH[,|Z|](X)^c$ (a Koszul complex). As $\cplx{P}$ is
  perfect and the support of $\cplx{P}$ is $|Z|$, it follows from
  \cite[Lem.~A.3]{MR1174255} that $K$ is in the thick closure of $\cplx{P}$, so
  $\cplx{P}$ is also a generator of
  $\DQCOH[,|Z|](X)$. Thus, $\RHom_{\Orb_X}(\cplx{P},\cplx{M}) \homotopic
  0$ if and only if $\cplx{M}\homotopic 0$.
\end{proof}
\begin{lemma}\label{lem:supports_and_supported_complexes}
  Let $X$ be a quasi-compact and quasi-separated algebraic stack and
  let $j\colon U \to X$ be a quasi-compact open immersion with
  complement $|Z|$.
  \begin{enumerate}
  \item\label{lem:supports_and_supported_complexes:item:1}
    If $\DQCOH[,|Z|](X)$ is generated by a set whose
    elements have compact image in $\DQCOH(X)$, then there exists a
    compact object $\cplx{Q}$ of $\DQCOH(X)$ with support $|Z|$.
  \item \label{lem:supports_and_supported_complexes:item:2} If
    $\DQCOH(X)$ is generated by a set of compact objects $\{\cplx{Q}_b\}_{b\in B}$ and
    there exists a perfect complex $\cplx{P}$ on $X$ with support
    $|Z|$, then $\DQCOH[,|Z|](X)$ is generated by the set
    $\{\cplx{Q}_b\otimes^{\LDERF}_{\Orb_X} \cplx{P}\}_{b\in B}$
    (whose elements have compact image in $\DQCOH(X)$).
  \end{enumerate}
\end{lemma}
\begin{proof}
  For \itemref{lem:supports_and_supported_complexes:item:1}:
  let $\{\cplx{Q}_\lambda\}_{\lambda\in \Lambda}$ be a set of
  generators for $\DQCOH[,|Z|](X)$ whose elements have compact image
  in $\DQCOH(X)$. Let $z\in |Z|$ be a point and choose a representative,
  that is, a field $k$ and a
  $1$-morphism of algebraic stacks $\bar{z}\colon \spec k \to
  X$ with image $z$.
  Since the diagonal of $X$ is quasi-compact and quasi-separated,
  it follows that $\bar{z}$ is quasi-compact and
  quasi-separated. By Lemma \ref{lem:qcqs_conditions}\itemref{item:qcqs_conditions:bdd_fl_bc}, it follows that $\LDERF
  \QCPBK{j}\RDERF \QCPSH{\bar{z}}\Orb_{\spec k} \homotopic 0$ and so there
  exists a $\lambda\in \Lambda$, an integer $n$, and a non-zero
  morphism $\cplx{Q}_\lambda[n] \to \RDERF \QCPSH{\bar{z}}\Orb_{\spec k}$. By
  adjunction, $\LDERF \QCPBK{\bar{z}} \cplx{Q}_\lambda \neq 0$ and we deduce
  that $\cup_{\lambda\in \Lambda}\supph(\cplx{Q}_\lambda) = |Z|$. It suffices
  to show that there is a finite subset $\Lambda' \subseteq \Lambda$
  such that $\cup_{\lambda\in \Lambda'} \supph(\cplx{Q}_\lambda) = |Z|$. This
  is obvious if $X$ is noetherian or, more generally, if $Z$ has a finite
  number of irreducible components.

  In complete generality, we note that $|Z|$ is constructible (by hypothesis) and that
  $\supph(\cplx{Q}_\lambda)$ is constructible for every $\lambda$ (by Lemma \ref{L:supports_perfect_complexes_properties}\itemref{L:supports_perfect_complexes_properties:cons}). Indeed, both subsets are closed with quasi-compact complement.
  We conclude that $|Z|=\cup_{\lambda\in \Lambda}\supph(\cplx{Q}_\lambda)$ has a
  finite subcovering since the constructible topology is quasi-compact.

  For \itemref{lem:supports_and_supported_complexes:item:2}: first note that the complex
  $\cplx{Q}_b\otimes^{\LDERF}_{\Orb_X} \cplx{P}$, which belongs to
  $\DQCOH[,|Z|](X)$, is a compact object of $\DQCOH(X)$
  (Lemma~\ref{L:compact_conc}\itemref{L:compact_conc:tens_c_p}).
  Let $\cplx{M} \in
  \DQCOH[,|Z|](X)$ and suppose that
  $\RHom_{\Orb_X}(\cplx{Q}_b\otimes^{\LDERF}_{\Orb_X} \cplx{P},\cplx{M}) \homotopic 0$. By
  adjunction \eqref{eq:ghomadj},
  $\RHom_{\Orb_X}(\cplx{Q}_b,\SRHom_{\Orb_X}(\cplx{P},\cplx{M})) \homotopic 0$. Since the
  set $\{\cplx{Q}_b\}_{b\in B}$ is generating, it follows that
  $\SRHom_{\Orb_X}(\cplx{P},\cplx{M}) \homotopic 0$. Thus $\cplx{M}\homotopic 0$ by
  Lemma~\ref{L:local-generator}.
\end{proof}

\subsection{Projection formula}
A typical application of Corollary \ref{cor:thomason_gens} is given by
the following Proposition. The given argument is a variant of
\cite[Prop.~5.3]{MR1308405}, though we have not seen this Proposition
in the literature before.
\begin{proposition}[Strong projection formula]\label{prop:strong_proj_form}
  Let $A$ be a ring and let $\pi\colon X \to \spec A$ be a morphism of
  algebraic stacks. Let $\cplx{Q}$ be a compact object of
  $\DQCOH(X)$ and let $\cplx{G} \in \DQCOH(X)$. Then for every $I \in
  \DCAT(A)$, there is a natural quasi-isomorphism:
  \[
  \RHom_{\Orb_X}(\cplx{Q},\cplx{G}) \otimes_A^{\LDERF} I
  \homotopic 
  \RHom_{\Orb_X}(\cplx{Q},\cplx{G}\otimes_{\Orb_X}^{\LDERF}
  \LDERF\QCPBK{\pi} I).
  \]
\end{proposition}
\begin{proof}
  First we describe the morphism: by adjunction, there is a natural
  morphism
  \[
  \LDERF \QCPBK{\pi}\RDERF \QCPSH{\pi}\SRHom_{\Orb_X}(\cplx{Q},\cplx{G})
  \to \SRHom_{\Orb_X}(\cplx{Q},\cplx{G})
  \]
  and so by \eqref{eq:ghomadj} there is a natural morphism
  \[
  \bigl(\LDERF \QCPBK{\pi}\RDERF \QCPSH{\pi}\SRHom_{\Orb_X}(\cplx{Q},\cplx{G})\bigr)
  \tensor_{\Orb_X}^{\LDERF} \cplx{Q} \tensor_{\Orb_X}^{\LDERF} \LDERF\QCPBK{\pi}I
  \to \cplx{G} \tensor_{\Orb_X}^{\LDERF} \LDERF \QCPBK{\pi}I.
  \]
  By \eqref{eq:ghomadj} again, there is a natural morphism:
  \[
  \LDERF \QCPBK{\pi}\left[\RDERF
    \QCPSH{\pi}\SRHom_{\Orb_X}(\cplx{Q},\cplx{G})
    \tensor_{A}^{\LDERF} I\right] \to \SRHom_{\Orb_X}(\cplx{Q},\cplx{G}
  \tensor_{\Orb_X}^{\LDERF} \LDERF \QCPBK{\pi}I).
  \]
  By adjunction and \eqref{eq:lghom} and
  \eqref{eq:global_sections_qcpsh} we deduce the existence of the
  required natural morphism
  \[
  \phi_I \colon \RHom_{\Orb_X}(\cplx{Q},\cplx{G}) \otimes_A^{\LDERF} I
  \to \RHom_{\Orb_X}(\cplx{Q},\cplx{G}\otimes_{\Orb_X}^{\LDERF}
  \LDERF\QCPBK{\pi} I).
  \]
  Let $\mathcal{K} \subset \DCAT(A)$ be the full subcategory with
  objects those $I$ such that $\phi_I$ is a quasi-isomorphism. It
  remains to show that $\mathcal{K} = \DCAT(A)$. Clearly,
  $\mathcal{K}$ is a triangulated subcategory that contains $A[k]$ for
  every integer $k$. Moreover, since $\cplx{Q}$ is a compact object of
  $\DQCOH(X)$, $\mathcal{K}$ is closed under small coproducts. The
  result now follows from Corollary \ref{cor:thomason_gens}. 
\end{proof}
A straightforward implication is the usual projection formula for
concentrated morphisms of algebraic stacks.
\begin{corollary}[Projection formula]\label{cor:proj_formula}
Let $f\colon X\to Y$ be a concentrated $1$-morphism of algebraic stacks.
The natural map
\[
(\RDERF \QCPSH{f} \cplx{M})\otimes^{\LDERF}_{\Orb_Y} \cplx{N} \to
\RDERF \QCPSH{f} (\cplx{M}\otimes^{\LDERF}_{\Orb_X} \LDERF \QCPBK{f}
\cplx{N})
\]
is a quasi-isomorphism for every $\cplx{M}\in \DQCOH(X)$ and $\cplx{N}\in \DQCOH(Y)$.
\end{corollary}
\begin{proof}
  By adjunction, there is a natural morphism
  \[
  \LDERF\QCPBK{f}\RDERF\QCPSH{f}\cplx{M} \tensor_{\Orb_X}^{\LDERF}
  \LDERF \QCPBK{f}\cplx{N} \to \cplx{M}\tensor_{\Orb_X}^{\LDERF}
  \LDERF \QCPBK{f}\cplx{N}.
  \]
  By adjunction again, we deduce the existence of a natural morphism
  \[
  \psi_{\cplx{N}} \colon (\RDERF \QCPSH{f} \cplx{M})\otimes^{\LDERF}_{\Orb_Y} \cplx{N} \to \RDERF \QCPSH{f} (\cplx{M}\otimes^{\LDERF}_{\Orb_X} \LDERF \QCPBK{f} \cplx{N}).
  \]
  It remains to show that $\psi_{\cplx{N}}$ is a quasi-isomorphism for
  every $\cplx{N} \in \DQCOH(Y)$. Note that the verification of this
  is smooth local on $Y$, so by Theorem
  \ref{thm:fcd_coprod}\itemref{thm:fcd_coprod:item:fl_bc}, we may
  reduce to the situation where $Y$ is an affine scheme. By Lemma
  \ref{L:compact_conc}\itemref{L:compact_conc:perf_conc_c}, $\Orb_X$ is compact and the result now
  follows from Proposition \ref{prop:strong_proj_form}. 
\end{proof}
\subsection{Tor-independent base change}
Let $f\colon X \to Y$ and $g\colon Y' \to Y$ be morphism of algebraic stacks. We say that 
$f$ and $g$ are \fndefn{tor-independent} if for every smooth morphism $\spec A \to Y$ 
and every pair of smooth morphisms $\spec B \to X\times_Y \spec A$ and $\spec A' \to 
Y'\times_Y \spec A$ we have that $\Tor_i^A(B,A') = 0$ for all $i>0$. Equivalently, 
$\STor^{Y,f,g}_i(\Orb_X,\Orb_{Y'}) = 0$ for every integer $i>0$ (see 
\cite[App.~C]{MR3589351} for details). Note that if $g$ is flat, then it is tor-independent 
of every $f$. The projection formula of Corollary \ref{cor:proj_formula} is powerful 
enough to prove a very general tor-independent base change result which extends 
Theorem \ref{thm:fcd_coprod}\itemref{thm:fcd_coprod:item:fl_bc}.
\begin{corollary}\label{cor:tor-independent-bc}
  Fix a $2$-cartesian square of algebraic stacks
    \[
    \xymatrix{X' \ar[r]^-{g'} \ar[d]_{f'} & X \ar[d]^f\\ Y'
      \ar[r]^{g} & Y.}
    \]
    If $f$ and $g$ are tor-independent and $f$ is concentrated, then
    there is a natural quasi-isomorphism for every $\cplx{M}\in
    \DQCOH(X)$:
    \[
    \LDERF \QCPBK{g} \RDERF \QCPSH{f}\cplx{M} \homotopic \RDERF
    \QCPSH{f'} \LDERF \QCPBK{g'}\cplx{M}.
    \]
\end{corollary}
\begin{proof}
  By Theorem \ref{thm:fcd_coprod}\itemref{thm:fcd_coprod:item:fl_bc},
  the result can be verified smooth-locally on $Y$ and $Y'$. Thus, we
  may assume that $Y=\spec A$ and $Y'=\spec A'$. In particular, $g$
  and $g'$ are affine. Since $g$ is affine, $\RDERF \QCPSH{g}$ is
  conservative (Corollary \ref{C:qaff_cons}). Hence, it is sufficient to verify that the morphism in
  question is a quasi-isomorphism after application of the functor
  $\RDERF \QCPSH{g}$. By the projection formula applied to $g$ and
  then $f$, there are natural quasi-isomorphisms
  \begin{align*}
    \RDERF \QCPSH{g}\LDERF \QCPBK{g} \RDERF \QCPSH{f}\cplx{M}
    &\homotopic \RDERF \QCPSH{f}\cplx{M} \tensor_{\Orb_Y} \RDERF \QCPSH{g}\Orb_{Y'} \\
    &\homotopic \RDERF \QCPSH{f}(\cplx{M} \tensor_{\Orb_X}^{\LDERF}
    \LDERF \QCPBK{f}\RDERF \QCPSH{g}\Orb_{Y'}).
  \end{align*}
  Note, however, that because $f$ and $g$ are tor-independent and $g$
  is affine, the natural map
  \[
  \LDERF \QCPBK{f} \RDERF \QCPSH{g}\Orb_{Y'} \to \RDERF \QCPSH{g'}\Orb_{X'}
  \]
  is a quasi-isomorphism. Indeed, this may be verified smooth-locally on $X$, so we may assume that $X=\spec C$. The morphism in question corresponds to the map $C\tensor_A^{\LDERF} A' \to (C\tensor_A A')[0]$ in $\DCAT(C)$, which is a quasi-isomorphism because $f$ and $g$ are
  tor-independent. With the projection formula and functoriality, we
  now obtain the following natural sequence of quasi-isomorphisms:
  \begin{align*}
    \RDERF \QCPSH{f}(\cplx{M} \tensor_{\Orb_X}^{\LDERF} \LDERF
    \QCPBK{f}\RDERF \QCPSH{g}\Orb_{Y'}) &\homotopic \RDERF
    \QCPSH{f}(\cplx{M} \tensor_{\Orb_X}^{\LDERF} \RDERF
    \QCPSH{g'}\Orb_{X'})\\
    &\homotopic \RDERF
    \QCPSH{f}\RDERF
    \QCPSH{g'}\LDERF \QCPBK{g'}\cplx{M}\\
    &\homotopic \RDERF \QCPSH{g}\RDERF \QCPSH{f'}\LDERF \QCPBK{g'}\cplx{M}.
  \end{align*}
  The result follows.
 \end{proof}
Note that in the setting of derived algebraic geometry, tor-independence is not necessary to obtain a base change result \cite[Prop.~3.10]{MR2669705}.
\subsection{Finite duality}
Using the projection formula, we can also establish finite duality.
\begin{theorem}\label{T:fin_duality}
  Let $f \colon X \to Y$ be a concentrated morphism of algebraic stacks.
  \begin{enumerate}
  \item\label{T:fin_duality:existence}
    $\RDERF \QCPSH{f} \colon \DQCOH(X) \to \DQCOH(Y)$ admits
    a right adjoint $f^\times$.
  \item\label{T:fin_duality:perfect}
    For every $\cplx{M} \in 
    \DQCOH(Y)$ there is a natural 
    quasi-isomorphism:
    \[
    \RDERF \QCPSH{f} f^\times(\cplx{M}) \simeq \SRHom_{\Orb_Y}^{\qcsubscript}(\RDERF
    \QCPSH{f}\Orb_X,\cplx{M}).
    \]
  \item \label{T:fin_duality:affine} If $f$ is affine, then for every $\cplx{M} \in \DQCOH(Y)$ there is a natural quasi-isomorphism:
    \[
    f^\times(\cplx{M}) \simeq \bar{f}^*\SRHom^{\qcsubscript}_{\Orb_Y}(f_*\Orb_X,\cplx{M}),
    \]
    where $\bar{f}^*$ is the functor from Corollary \ref{C:aff_equiv_der}. 
  \item \label{T:fin_duality:finite} If $f$ is affine and 
    $f_*\Orb_X$ is perfect, then for every $\cplx{M} \in \DQCOH(Y)$ there is a natural quasi-isomorphism:
    \[
    f^\times(\Orb_Y)\tensor_{\Orb_X}^{\LDERF} f^*(\cplx{M})
    \simeq f^\times(\cplx{M}).
    \]
    In particular, $f^\times$ preserves small coproducts. Moreover, $f^\times$ is compatible with tor-independent 
    base change on $Y$. If in addition $f$ is surjective, 
    then $f^\times$ is conservative.
  \end{enumerate}
\end{theorem}
\begin{proof}
  By Theorem \ref{thm:fcd_coprod}\itemref{thm:fcd_coprod:item:sm_cp}, the functor 
  $\RDERF \QCPSH{f}$ preserves small coproducts. Since $\DQCOH(X)$ is well generated 
  \cite[Thm.~B.1]{hallj_neeman_dary_no_compacts}, the existence of $f^\times$ follows 
  from \cite[Prop.~1.20]{MR1812507}. Now fix $\cplx{N} \in \DQCOH(Y)$; then there are natural 
  isomorphisms:
  \begin{align*}
    \Hom_{\Orb_Y}(\cplx{N},    \RDERF \QCPSH{f} f^\times(\cplx{M})) 
    &\cong \Hom_{\Orb_X}(\LDERF \QCPBK{f}\cplx{N},f^\times(\cplx{M})) \\
    &\cong \Hom_{\Orb_Y}(\RDERF \QCPSH{f}\LDERF \QCPBK{f}\cplx{N},\cplx{M})\\
    &\cong \Hom_{\Orb_Y}((\RDERF \QCPSH{f}\Orb_X)\otimes^{\LDERF}_{\Orb_Y} \cplx{N}, \cplx{M})\\
    &\cong \Hom_{\Orb_Y}(\cplx{N},\SRHom_{\Orb_Y}^{\qcsubscript}(\RDERF
    \QCPSH{f}\Orb_X,\cplx{M})).
  \end{align*}
  The penultimate isomorphism follows from the projection formula (Corollary 
  \ref{cor:proj_formula}). By the Yoneda 
  Lemma, this proves \itemref{T:fin_duality:perfect}.

  We now address \itemref{T:fin_duality:affine}. Let $\tilde{f}^\times \colon \DQCOH(Y) \to \DQCOH(X)$ be the functor 
  \[
  \tilde{f}^\times(\cplx{M}) = \bar{f}^*\SRHom_{\Orb_Y}^{\qcsubscript}(f_*\Orb_X,\cplx{M}),
  \]
  where $\bar{f}^*$ comes from the equivalence of Corollary \ref{C:aff_equiv_der}. We 
  claim there is a natural transformation of functors $\tilde{f}^\times \Rightarrow f^\times$. 
  To see this, let $\cplx{N} \in \DQCOH(X)$ and $\cplx{M} \in \DQCOH(Y)$; then there are 
  natural morphisms:
  \begin{align*}
    \Hom_{\Orb_X}(\cplx{N},\tilde{f}^\times(\cplx{M})) &\to \Hom_{\Orb_Y}(\RDERF \QCPSH{f}\cplx{N},\RDERF \QCPSH{f}\tilde{f}^\times(\cplx{M}))\\
    &= \Hom_{\Orb_Y}(\RDERF \QCPSH{f}\cplx{N}, \SRHom_{\Orb_Y}^{\qcsubscript}(f_*\Orb_X,\cplx{M}))\\
      &\to \Hom_{\Orb_Y}(\RDERF \QCPSH{f}\cplx{N}, \SRHom_{\Orb_Y}^{\qcsubscript}(\Orb_Y,\cplx{M}))\\
          &\cong \Hom_{\Orb_Y}(\RDERF \QCPSH{f}\cplx{N},\cplx{M})\\
    &\cong \Hom_{\Orb_X}(\cplx{N},f^\times(\cplx{N})).
  \end{align*}
  By the Yoneda Lemma, we have the claim. Since $f$ is affine, to prove that the natural 
  transformation $\tilde{f}^\times \Rightarrow f^\times$ is an isomorphism, it is sufficient to 
  prove it is after application of $\RDERF \QCPSH{f}$ (Lemma \ref{C:qaff_cons}). This follows from the definition of 
  $\tilde{f}^\times$ and \itemref{T:fin_duality:perfect}. 

  We now treat \itemref{T:fin_duality:finite}. In general, it is not difficult to produce a natural morphism for every $\cplx{M} \in \DQCOH(Y)$:
  \[
  f^\times(\Orb_Y) \tensor^{\LDERF}_{\Orb_X} \LDERF \QCPBK{f}(\cplx{M}) \to f^\times(\cplx{M}).
  \]
  Since $f$ is affine, it is sufficient to prove this morphism is an isomorphism after 
  application of $\RDERF \QCPSH{f}$ (Lemma \ref{C:qaff_cons}). By 
  \itemref{T:fin_duality:perfect} and the projection formula 
  (Corollary~\ref{cor:proj_formula}), we see that it is sufficient to prove that the induced 
  morphism:
  \[
  \SRHom_{\Orb_Y}^{\qcsubscript}(f_*\Orb_X,\Orb_Y) \tensor^{\LDERF}_{\Orb_Y} \cplx{M} \to \SRHom_{\Orb_Y}^{\qcsubscript}(f_*\Orb_X,\cplx{M})
  \]
  is a quasi-isomorphism for every $\cplx{M} \in \DQCOH(Y)$. But $f_*\Orb_X$ is 
  perfect, so Lemma \ref{L:basic_props_perfect}\itemref{L:basic_props_perfect:dual} now 
  gives the claim. 

  For the compatibility of $f^\times$ with base change, we consider a tor-independent $2$-cartesian diagram 
  of algebraic stacks:
  \[
  \xymatrix{X' \ar[r]^{g'} \ar[d]_{f'} & X \ar[d]^{f}\\ Y' \ar[r]^g & Y.}
  \]
  Adjointness and tor-independent base change (Corollary~\ref{cor:tor-independent-bc}) 
  provides a natural transformation $\LDERF \QCPBK{g'} f^\times \to (f')^\times 
  \LDERF \QCPBK{g}$ of functors that we must show is an isomorphism. Tor-independent 
  base change also implies that there is a quasi-isomorphism: $\LDERF \QCPBK{g}\RDERF 
  \QCPSH{f}\Orb_X \simeq \RDERF \QCPSH{f'} \LDERF \QCPBK{g'}\Orb_X$.
  Since $\RDERF \QCPSH{f}\Orb_X$ is perfect, it follows that $\RDERF 
  \QCPSH{f'}\Orb_{X'}$ is perfect. By the formula just determined for $f^\times$, we thus 
  see that it is sufficient to prove that $\LDERF \QCPBK{g'} f^\times(\Orb_Y) \to (f')^\times 
  \LDERF \QCPBK{g}(\Orb_Y)$ is a quasi-isomorphism. Since $f$ (and so 
  $f'$) is affine, it is sufficient prove this isomorphism after application of $\RDERF 
  \QCPSH{f'}$. This observation, together with \itemref{T:fin_duality:perfect} and 
  tor-independent base change shows that it is sufficient to prove that the morphism:
  \[
  \LDERF  \QCPBK{g}\SRHom_{\Orb_Y}^{\qcsubscript}(f_*\Orb_X,\Orb_Y) \to \SRHom_{\Orb_{Y'}}(\LDERF \QCPBK{g}(f_*\Orb_X),\Orb_{Y'})
  \]
  is a quasi-isomorphism. But $f_*\Orb_X$ and $\LDERF \QCPBK{g}(f_*\Orb_X)$ are both 
  perfect and so dualizable (Lemma \ref{L:basic_props_perfect}). In particular, the derived pullback of the dual of $f_*\Orb_X$ coincides with the dual of $\LDERF \QCPBK{g}(f_*\Orb_X)$. It follows that the asserted map is a quasi-isomorphism 
  and the claim follows.

  Finally, we address the conservativity. For this, it is 
  sufficient to observe that if $f$ is surjective, then $\supph(f_*\Orb_X)=|Y|$. 
 But 
  $\SRHom_{\Orb_Y}(f_*\Orb_X,\cplx{M})\homotopic 0$ if and only if
  $\cplx{M}\homotopic 0$ (Lemma~\ref{L:local-generator}).
\end{proof}
\begin{corollary} \label{C:fin_duality}  
  If $f\colon X \to Y$ is a finite and faithfully flat morphism of finite presentation between 
  algebraic stacks, then the functor 
  \[
  f^\times(\cplx{M}) = \bar{f}^*\SRHom_{\Orb_Y}(f_*\Orb_X,\cplx{M}), \quad \mbox{where 
    $\cplx{M} \in \DQCOH(Y)$,}
  \]
  is right adjoint to $\RDERF \QCPSH{f}\colon \DQCOH(X) \to \DQCOH(Y)$. Moreover, 
  $f^\times$ is compatible with arbitrary base change on $Y$, 
  $f^\times(\Orb_Y)\tensor_{\Orb_X}^{\LDERF} f^*(-)
    \simeq f^\times(-)$, preserves small coproducts, and is conservative.  
\end{corollary}
\subsection{Coherent functors}
Combining the strong projection formula of Proposition
\ref{prop:strong_proj_form} with the characterization of compact
objects in Lemma \ref{lem:char_compact_stack}, we can prove most of
Theorem \ref{mainthm:coherent_functor}. 
\begin{corollary}\label{cor:coherent_functor}
  Let $A$ be a noetherian ring and let $\pi \colon X \to \spec A$ be a
  morphism of finite type between noetherian algebraic stacks.
  Suppose that
  \begin{enumerate}
  \item for every $i\geq 0$ and $\shv{M}\in \COH(X)$, the cohomology
    $H^i(X_{\lisset},\shv{M})$ is a coherent $A$-module
    (e.g., $\pi$ proper); and
  \item $\DQCOH(X)$ is compactly generated.
  \end{enumerate}
  Then, for every $\cplx{F} \in \DQCOH(X)$ and $\cplx{G} \in \DCAT_{\COH}^b(X)$, the functor
  \[
  \Hom_{\Orb_X}(\cplx{F},\cplx{G} \tensor_{\Orb_X}^{\LDERF} \LDERF
  \QCPBK{\pi}(-)) \colon \MOD(A) \to \MOD(A)
  \]
  is coherent.
\end{corollary}
\begin{proof}
  We begin by observing that the coherent functors $\MOD(A) \to
  \MOD(A)$ constitute a full abelian subcategory of the category of
  $A$-linear functors, which is closed under products (where everything
  is computed ``pointwise'') \cite[Ex.~3.9]{hallj_coho_bc}. Let
  $\mathscript{T} \subseteq \DQCOH(X)$ denote the full subcategory
  with objects those $\cplx{F} \in \DQCOH(X)$ where the functor
  $\Hom_{\Orb_X}(\mathcal{F},\mathcal{G}\tensor^{\LDERF}_{\Orb_X}
  \LDERF\QCPBK{\pi}(-))$ is coherent for every $\mathcal{G} \in
  \DCAT^b_{\COH}(X)$. In particular, $\mathscript{T}$ is closed under
  small coproducts, shifts, and triangles. By Corollary
  \ref{cor:thomason_gens}, it is enough to prove that $\mathscript{T}$
  contains the compact objects of $\DQCOH(X)$. If $\mathcal{Q} \in
  \DQCOH(X)$ is compact, then the strong projection formula
  (Proposition \ref{prop:strong_proj_form}) implies that there is a
  natural quasi-isomorphism:
  \[
  \RHom_{\Orb_X}(\cplx{Q},\cplx{G}) \otimes_A^{\LDERF} I
  \homotopic 
  \RHom_{\Orb_X}(\cplx{Q},\cplx{G}\otimes_{\Orb_X}^{\LDERF}
  \LDERF\QCPBK{\pi} I).
  \]
  Since $\mathcal{Q}$ is compact, it is perfect (Lemma
  \ref{L:compact_conc}\itemref{L:compact_conc:c_perf}) and so $\SRHom_{\Orb_X}(\mathcal{Q},\mathcal{G}) \in
  \DCAT_{\COH}^b(X)$. The assumption on preservation of coherence implies
  that $\RDERF \QCPSH{\pi}$ sends $\DCAT_{\COH}^+(X)$ to $\DCAT_{\COH}^+(A)$.
  In particular, $\RHom_{\Orb_X}(\mathcal{Q},\mathcal{G}) \simeq \RDERF \QCPSH{\pi} \SRHom_{\Orb_X}(\mathcal{Q},\mathcal{G})\in \DCAT_{\COH}^+(A)$. By Lemma
  \ref{lem:char_compact_stack} we also have
  $\RHom_{\Orb_X}(\mathcal{Q},\mathcal{G}) \in \DCAT^b(A)$. Thus the
  functor
  $\Hom_{\Orb_X}(\mathcal{Q},\mathcal{G}\otimes_{\Orb_X}^{\LDERF}\LDERF
  \QCPBK{\pi}(-))$ is coherent \cite[Ex.~3.13]{hallj_coho_bc} and we
  deduce the result.
\end{proof}
\section{Presheaves of triangulated categories}\label{sec:presheaves-tri-cats}
Throughout this section we fix a small category $\mathscript{D}$ that admits all finite limits. Let $\TRICAT$ denote the $2$-category of
triangulated categories. A
\fndefn{$\mathscript{D}$-presheaf of triangulated categories} is a 
$2$-functor $\mathcal{T} \colon \mathscript{D}^\opp \to \TRICAT$.

Given a
morphism $f \colon U \to V$ in $\mathscript{D}$, there is an induced
pullback functor $f^*_{\mathcal{T}} \colon \mathcal{T}(V) \to
\mathcal{T}(U)$. When there is no cause for confusion, we will
suppress the subscript $\mathcal{T}$ from $f^*_{\mathcal{T}}$. For any
such $f$ (not necessarily a monomorphism), we let
\[
\mathcal{T}_{V\setminus U}(V) = \ker (f^* \colon \mathcal{T}(V) \to
\mathcal{T}(U)).
\]
We say that $\mathcal{T}$ \emph{has adjoints} if for every morphism
$f \colon U \to V$
  in $\mathscript{D}$, the pullback functor $f^* \colon \mathcal{T}(V)
  \to \mathcal{T}(U)$ 
  admits a right adjoint $f_* \colon \mathcal{T}(U) \to
  \mathcal{T}(V)$.

  \begin{definition}\label{defn:tflat}
    Suppose that $\mathcal{T}$ is a $\mathscript{D}$-presheaf of
    triangulated categories with adjoints. Let $f\colon U \to V$ be a
    morphism in $\mathscript{D}$ and let $N \in \mathcal{T}(V)$. We
    denote by $\eta^f_N\colon N \to f_*f^*N$ the unit of the
    adjunction.  A morphism $g \colon W \to V$ in $\mathscript{D}$ is
    \fndefn{$\mathcal{T}$-preflat} if for every cartesian square in
    $\mathscript{D}$:
    \[
    \xymatrix{U_W \ar[d]_{f_W} \ar[r]^{g_U} & U \ar[d]^f\\ W \ar[r]^g
      & V,}
    \]
    the natural transformation $g^*f_* \to (f_W)_*(g_U)^*$ is an
    isomorphism. A morphism $g \colon W \to V$ in $\mathscript{D}$ is
    \fndefn{$\mathcal{T}$-flat} if for every morphism $V' \to V$, the
    pullback $g'\colon W' \to V'$ of $g$ is $\mathcal{T}$-preflat.
  \end{definition}
Note that because monomorphisms are stable under base change, 
$\mathcal{T}$-flat monomorphisms are stable under base change. 
\begin{example}\label{ex:dmodqc_defn}
  Let $Y$ be an algebraic stack that is quasi-compact and quasi-separated. 
  Let $\REP[\mathrm{fp}]{Y}$ denote the category of $1$-morphisms $X
  \to Y$ that are representable and of finite presentation.
  The category $\REP[\mathrm{fp}]{Y}$ is small. 
  We have a $\REP[\mathrm{fp}]{Y}$-presheaf of triangulated 
  categories $\DQCOH \colon (\REP[\mathrm{fp}]{Y})^\opp \to \TRICAT$, that sends $X
  \to Y$ to $\DQCOH(X)$ and a $1$-morphism $f \colon X' \to X$ in 
  $\REP[\mathrm{fp}]{Y}$ to $\LDERF \QCPBK{f} \colon \DQCOH(X) \to
  \DQCOH(X')$.
  The functor $\LDERF \QCPBK{f}$ admits a right adjoint $\RDERF
  \QCPSH{f}$ so
  $\DQCOH$ is a presheaf with adjoints.

  By Theorem \ref{thm:fcd_coprod}\itemref{thm:fcd_coprod:item:fl_bc},
  if $f$ is a flat morphism, then it is $\DQCOH$-flat. Conversely, if
  $f\colon X' \to X$ is $\DQCOH$-flat, then $f$ is flat. Indeed, this
  is local on the source and target of $f$, so it is
  sufficient to show that if $f\colon \spec B \to \spec A$ is
  $\DQCOH$-preflat, then $B$ is a flat $A$-algebra. For this, we note
  that if $I$ is an ideal of $A$, then corresponding to $i\colon \spec
  (A/I) \to \spec A$ we see that there is a quasi-isomorphism $(A/I)
  \tensor_{A}^\LDERF B \homotopic (B/IB)[0]$. That is, for all $n>0$ and
  ideals $I$ of $A$ we have that $\Tor^n_A(B,A/I)
  = 0$---hence $B$ is flat over $A$. It follows that the $\DQCOH$-flat
  monomorphisms are the quasi-compact open immersions~\cite[IV.17.9.1]{EGA}.
\end{example}
\begin{example}
  Our notion of $\mathcal{T}$-flatness is not always optimal. In particular, it
  is weaker than expected in the derived setting. If $\mathcal{T}$ is a
  presheaf of triangulated categories \emph{with $t$-structures}, then a better
  definition is that $f$ is $\mathcal{T}$-flat if $f^*$ is $t$-exact.

  To illustrate this, suppose $\mathscript{D}=\mathsf{SCR}^\opp$ is the
  $\infty$-category of affine derived schemes, that is, the opposite category to the
  $\infty$-category of simplicial commutative rings. Further, let $\mathcal{T}=\MOD(-)$
  be the functor that takes a simplicial commutative ring $A$ to the stable 
  $\infty$-category $\MOD(A)$ of (not necessarily connective) $A$-modules. Then every
  morphism in $\mathscript{D}$ is $\mathcal{T}$-flat whereas $\spec B\to
  \spec A$ is flat exactly when the pullback $B\ltensor_A -\colon \MOD(A)\to
  \MOD(B)$ is $t$-exact. Nevertheless, just as in the non-derived case, the
  finitely presented $\mathcal{T}$-flat monomorphisms are exactly the
  quasi-compact open immersions since every monomorphism of derived schemes is
  formally \'etale~\cite[2.2.2.5 (2)]{MR2394633}.
\end{example}
\begin{lemma}\label{lem:adm_basic}
  Let $\mathcal{T}$ be a
  $\mathscript{D}$-presheaf of triangulated categories with adjoints. Fix a commutative
  diagram in $\mathscript{D}$:
    \[
    \xymatrix@-1ex{W_V \ar[d]_{f_V} \ar[r]^{j_W} & W \ar[d]^f\\ V \ar[r]^j &
      X.}
    \] 
  \begin{enumerate}
  \item\label{lem:adm_basic:item:ff_bc} If $j$ is a
    $\mathcal{T}$-preflat monomorphism in $\mathscript{D}$, then the adjunction $j^*j_* \to
    \ID{\mathcal{T}(V)}$ is an isomorphism.
    \item \label{lem:adm_basic:item:pb_supp} $f^*\colon
      \mathcal{T}_{X\setminus V}(X) \to \mathcal{T}(W)$ factors
      through $\mathcal{T}_{W\setminus W_V}(W)$. 
    \item \label{lem:adm_basic:item:im_supp} If the diagram is
      cartesian and $j$ is $\mathcal{T}$-preflat, then the functor $f_*\colon
      \mathcal{T}_{W\setminus W_V}(W) \to \mathcal{T}(X)$ factors
      through $\mathcal{T}_{X\setminus V}(X)$ and is right adjoint
      to $f^* \colon \mathcal{T}_{X\setminus V}(X) \to
      \mathcal{T}_{W\setminus W_V}(W)$. 
  \end{enumerate}
\end{lemma}
\begin{proof}
If $j \colon V \to X$ is a monomorphism, then the
    commutative diagram:
    \[
    \xymatrix@-2ex{V \ar[r]^{\ID{V}} \ar[d]_{\ID{V}}& V
      \ar[d]^{j} \\ V \ar[r]^{j} & X }
    \]
    is cartesian, whence  $j^*j_* \homotopic
    (\ID{V})_*(\ID{V})^* \homotopic \ID{\mathcal{T}(V)}$. 

By functoriality \itemref{lem:adm_basic:item:pb_supp} is
        trivial. For \itemref{lem:adm_basic:item:im_supp}, given
        $M \in 
    \mathcal{T}_{W\setminus W_V}(W)$, then $j^*f_*M \homotopic
    (f_V)_*(j_W)^*M \homotopic 0$, hence $f_*M \in
    \mathcal{T}_{X\setminus V}(X)$. 
\end{proof}
\begin{definition}\label{defn:mv}
Fix a $\mathscript{D}$-presheaf of
triangulated categories $\mathcal{T}$ with adjoints.
A \fndefn{Mayer--Vietoris $\mathcal{T}$-square}
is a cartesian diagram in $\mathscript{D}$:
\[
\xymatrix{U' \ar[d]_{f_U} \ar[r]^{j'} & X' \ar[d]^f \\ U
  \ar[r]^j & X,}
\] 
satisfying the following three conditions:
\begin{enumerate}
\item\label{defn:mv:item:1} $j$ is a $\mathcal{T}$-flat monomorphism,
\item\label{defn:mv:item:2} the natural transformation $f^*j_* \to j'_*f_U^*$ is an
  isomorphism, and
\item\label{defn:mv:item:3} the induced functor
 $f^* \colon \mathcal{T}_{X\setminus U}(X) \to \mathcal{T}_{X'\setminus U'}(X')$
  is an equivalence of categories.
\end{enumerate}
\end{definition}
Condition~\itemref{defn:mv:item:2} for a Mayer--Vietoris $\mathcal{T}$-square is satisfied
if $f$ is a $\mathcal{T}$-(pre)flat morphism. By tor-independent base change (Corollary \ref{cor:tor-independent-bc}), if $\mathcal{T} = \DQCOH$, then condition~\itemref{defn:mv:item:2} is satisfied for every $f$. In~\cite{mayer-vietoris}, we
will consider applications of these Mayer--Vietoris triangles to a result
of Moret-Bailly \cite{MR1432058}. For this intended application, it is essential
that we permit $f$ to be non-flat. 
\begin{example}\label{ex:dmodqc_mv}
  We continue with Example \ref{ex:dmodqc_defn}. Let $f \colon X' \to X$ be a representable, quasi-compact, and quasi-separated \'etale neighborhood of a closed subset $|Z| \subseteq |X|$ with quasi-compact complement $|U|$. 
  Let $j \colon U \hookrightarrow X$ be the resulting quasi-compact open immersion. 
  Then the cartesian square:
  \[
  \xymatrix{U' \ar[d]_{f_U} \ar[r]^{j'} & X' \ar[d]^f \\ U
  \ar[r]^j & X,}
  \]
  is a Mayer--Vietoris $\DQCOH$-square.
  To see this, it remains to prove that the functor $\LDERF \QCPBK{f}$ induces the desired equivalence. 
  Now the exact functor $f^* \colon \MOD(X) \to \MOD(X')$ admits an exact left adjoint $f_! \colon \MOD(X') \to \MOD(X)$~\cite[\spref{03DI}]{stacks-project}.
  Explicitly, for $M\in \MOD(X')$ we have that $f_!M$ is the
  sheafification of the presheaf
  \[
  (V\to X) \longmapsto \bigoplus_{\mathclap{\phi \in \Hom_X(V,X')}} M(V \xrightarrow{\phi} X').
  \]
  Note that the natural map $M \to f^*f_!M$ is an isomorphism for all $M\in \MOD(X')$ such that $j'^*M = 0$.
  Also, if $N \in \MOD(X)$ and $j^*N = 0$, then the natural map $f_!f^*N \to N$ is an isomorphism.
  The exactness of the adjoint pair $(f_!,f^*)$ now gives an adjoint pair on the level of derived categories $(f_!,f^*) \colon \DCAT(X) \leftrightarrows \DCAT(X')$ and that the relations just given also hold on the derived category.
  Next, we observe that the restriction of $f^*$ to $\DQCOH(X)$ coincides with $\LDERF \QCPBK{f}$. 
  Thus it remains to prove that if $M \in \DQCOH(X')$ and $\LDERF \QCPBK{(j')}M  \homotopic 0$, then $f_!M \in \DQCOH(X)$. 
  The exactness of $f_!$ and $\LDERF \QCPBK{(j')}$ show that is sufficient to prove this result when $M$ is a quasi-coherent sheaf such that $(j')^*M = 0$. 
  Note that $(U \xrightarrow{j} X, X' \xrightarrow{f} X)$ is an \'etale cover of $X$, $j^*f_!M = (f_U)_!(j')^*M = 0$ and $f^*f_!M \cong M$.
  We deduce that \'etale locally $f_!M$ is quasi-coherent. 
  By descent $f_!M$ is quasi-coherent and the result is proved. For a different
  proof in a more general context, see~\cite[Prop.~4.2]{mayer-vietoris}.
\end{example}
\begin{example}[\'Etale cohomology]\label{ex:etale-coh}
  Let $Y$ be an algebraic stack that is quasi-compact and quasi-separated.  Let
  $\Lambda$ be a noetherian ring such that $\Lambda$ is torsion and $|\Lambda|$ is invertible on $Y$. We have the derived category
  $\DCART(Y,\Lambda)$ of lisse-\'etale $\Lambda$-modules on $Y$ with cartesian
  cohomology~\cite[12.10]{MR1771927}, \cite[\S\S2.1--2.2,
    Ex.~2.1.8]{MR2434692}. More generally, we have a $\REP[\mathrm{fp}]{Y}$-presheaf
  of triangulated categories $\DCART(-,\Lambda) \colon
  (\REP[\mathrm{fp}]{Y})^\opp \to \TRICAT$, that sends $X \to Y$ to
  $\DCART(X,\Lambda)$ and a $1$-morphism $f \colon X' \to X$ in
  $\REP[\mathrm{fp}]{Y}$ to $f^* \colon \DCART(X,\Lambda) \to
  \DCART(X',\Lambda)$. By smooth base change, $f_*$ takes cartesian sheaves to cartesian
  sheaves so the functor $f^*$ admits a right adjoint $\RDERF f_*$ on bounded below objects, i.e., the subpresheaf of bounded below objects
  $\DCART^+(-,\Lambda)$ is a presheaf with adjoints. Moreover, \'etale and smooth morphisms are
  $\DCART^+(-,\Lambda)$-flat. 

  If $f\colon X'\to X$ is an \'etale neighborhood of $|Z|\subseteq |X|$ as in
  the previous example, then the resulting square is a Mayer--Vietoris
  $\DCART^+(-,\Lambda)$-square. Indeed, if $i\colon Z\to X$ is a closed immersion
  for some scheme structure on $Z$, then $i^*$ and $i_*$ induces equivalences
  of categories $\DCAT^+_{\mathrm{cart},|X\setminus U|}(X,\Lambda)\cong \DCART^+(Z,\Lambda)$.

  We also have a subpresheaf $\DCONS(-,\Lambda)$ where
  $\DCONS(X,\Lambda)\subseteq \DCART(X,\Lambda)$ consists of the objects with
  constructible cohomology sheaves~\cite[18.6]{MR1771927},~\cite[\S\S2.1--2.2,
    Ex.~2.2.6]{MR2434692}. If $Y$ is quasi-excellent of finite dimension,
  $\Lambda=\Z/N\Z$ and $N$
  is invertible on $Y$, then $\DCONS(-,\Lambda)$ has adjoints (Deligne--Gabber's
  finiteness theorem) and an \'etale neighborhood is a Mayer--Vietoris
  $\DCONS(-,\Lambda)$-square.
\end{example}
\begin{example}\label{E:simplecat}
  Consider a $2$-commutative diagram of triangulated functors:
  \[
  \xymatrix{  \mathscript{R}' & \mathscript{S}' \ar[l]_{j'^*}\\
  \ar[u]^{r^*} \mathscript{R} & \mathscript{S} \ar[l]_{j^*}\ar[u]_{f^*}, }
  \]
  and assume that they all have right adjoints, which we will denote as $f_*$, $j_*$, $r_*$, 
  and $j'_*$, respectively. Let $\mathscript{K} = \ker(j^*)$ and $\mathscript{K}' = 
  \ker(j'^*)$.   Let $\mathscript{D}$ be the category consisting of the following objects and
  arrows:
  \[
  \xymatrix{R' \ar[r]^{j'} \ar[d]_{r} & S' \ar[d]^{f}\\ R \ar[r]^j & S.}
  \]
  There is a $\mathscript{D}$-presheaf of triangulated categories $\mathscript{T}$ with adjoints 
  such that $\mathscript{T}(S) = \mathscript{S}$ etc. The square above is a Mayer--Vietoris 
  $\mathscript{T}$-square if and only if the following conditions are satisfied:
  \begin{enumerate}
  \item\label{E:simplecat:bousfield} the natural transformations
    $j^*j_* \to \ID{}$ and $j'^*j'_* \to \ID{}$ are isomorphisms,
  \item\label{E:simplecat:fbc} the natural transformation
    $j^*f_* \to r_*j'^*$ is an isomorphism,
  \item\label{E:simplecat:tibc} the natural transformation
    $f^*j_* \to j'_*r^*$ is an isomorphism, and
  \item\label{E:simplecat:nhd} the induced functor
    $\mathscript{K} \to \mathscript{K}'$ is an equivalence of categories.
  \end{enumerate}
  Condition \itemref{E:simplecat:bousfield} implies that
  $\mathscript{R} = \mathscript{S}/\mathscript{K}$,
  $\mathscript{R}' = \mathscript{S}'/\mathscript{K}'$ and
  $\mathscript{K}$, $\mathscript{K}'$ are Bousfield subcategories of
  $\mathscript{S}$, $\mathscript{S}'$ respectively (Lemma
  \ref{lem:identify_quots} and \cite[Ch.~9]{MR1812507}).
\end{example}

Mayer--Vietoris $\mathcal{T}$-squares give rise to many nice properties.
In particular, we obtain a familiar distinguished triangle.
\begin{lemma}\label{lem:mv_triangle}
  Let $\mathcal{T}$ be a
  $\mathscript{D}$-presheaf of triangulated categories with adjoints.
  Consider a Mayer--Vietoris
  $\mathcal{T}$-square in $\mathscript{D}$:
  \[  
  \xymatrix{U' \ar[d]_{f_U} \ar[r]^{j'} & X' \ar[d]^f \\ U \ar[r]^j & X.}  
  \] 
  \begin{enumerate}
  \item\label{item:mv_triangle:MV} If $N\in \mathcal{T}(X)$, then
    there is a unique map $d$ that makes the triangle:
    \[ 
    \xymatrix@C2pc{%
      N \ar[r]^-{\begin{psmallmatrix}\eta^j_N\\\eta^f_N\end{psmallmatrix}}
       & j_*j^*N \oplus f_*f^*N \ar[rrr]^-{\begin{psmallmatrix}\eta^f_{j_*j^*N}
       & -f_*f^*\eta^j_N\end{psmallmatrix}}
       &&& f_*f^*j_*j^*N\ar[r]^-{d}
       & N[1] }
    \]
    distinguished. Moreover, this $d$ is functorial in $N$.
  \item\label{item:mv_triangle:adj} Let $M \in \mathcal{T}_{X'\setminus
      U'}(X')$ and let $N \in \mathcal{T}(X)$. Then there is a natural
    bijection:
    \[
    \Hom_{\mathcal{T}(X)}(f_*M,N) \cong
    \Hom_{\mathcal{T}(X')}(M,f^*N). 
    \]
  \item\label{item:mv_triangle:desc} Given $N_U \in \mathcal{T}(U)$, $N' \in
    \mathcal{T}(X')$, and an isomorphism $\delta \colon j'^*N' \to
    f_U^*N_U$, define $N$ by a distinguished triangle in
    $\mathcal{T}(X)$: 
    \[
    \xymatrix{%
      N \ar[r]
       & j_*N_U \oplus f_*N' \ar[rr]^{\begin{psmallmatrix}\eta^f_{j_*N_U} & -\alpha\end{psmallmatrix}}
       && f_*f^*j_*N_U\ar[r] & N[1],}
    \]
    where $\alpha \colon f_*N' \to f_*f^*j_*N_U$ is the composition:
    \[
    f_*N' \xrightarrow{f_*\eta^{j'}_{N'}} f_*j'_*j'^*N'
    \xrightarrow{f_*j'_*\delta} f_*j'_*f_U^*N_U \cong f_*f^*j_*N_U.
    \]
    Then the induced maps $j^*N
    \to N_U$ and $f^*N \to N'$ are isomorphisms. 
  \item\label{item:mv_triangle:comp} If $N \in \mathcal{T}(X)$ satisfies 
    $j^*N \in \mathcal{T}(U)^c$, $f^*N \in
    \mathcal{T}(X')^c$, and $f_U^*j^*N \in
    \mathcal{T}(U')^c$, then $N \in
    \mathcal{T}(X)^c$. 
  \end{enumerate}
\end{lemma}
\begin{proof} An equivalent formulation of \itemref{item:mv_triangle:MV} is that
\[
\xymatrix@C3pc{%
  N \ar[r]^{\eta^j_N}\ar[d]^{\eta^f_N}
    & j_*j^*N\ar[d]^{\eta^f_{j_*j^*N}} \\
  f_*f^*N \ar[r]^-{f_*f^*\eta^j_N}
    & f_*f^*j_*j^*N}
\]
is a homotopy cartesian square~\cite[Def.~1.4.1]{MR1812507} whose differential
$d$ is unique and is functorial in $N$. To see that the
square is cartesian, first choose $C$ such that we have a distinguished
triangle:
\[
\xymatrix{C \ar[r]^l & N \ar[r]^-{\eta^j_N} & j_*j^*N \ar[r]^m & C[1]. }
\]
Since $j$ is a $\mathcal{T}$-flat monomorphism, $j^*\eta^j_N$ is
an isomorphism (Lemma
\ref{lem:adm_basic}\itemref{lem:adm_basic:item:ff_bc}). It follows
that $j^*C \cong 0$, so $C \in \mathcal{T}_{X\setminus U}(X)$, and
$\eta^f_C$ is an isomorphism. We thus obtain a morphism of distinguished
triangles:
\[
\xymatrix@C4.5pc{%
  C \ar[r]^l\ar@{=}[d] & N \ar[r]^-{\eta^j_N}\ar[d]^{\eta^f_N} & j_*j^*N \ar[r]^m\ar[d]^{u} & C[1]\ar@{=}[d] \\
  C \ar[r]^-{\eta^f_N\circ l} & f_*f^*N \ar[r]^-{f_*f^*\eta^j_N} & f_*f^*j_*j^*N \ar[r]^-{(\eta^f_{C[1]})^{-1}\circ f_*f^*m} & C[1] \\
}
\]
for some morphism $u$. We can certainly let $u=\eta^f_{j_*j^*N}$ and we will
soon see that this is actually the only possible $u$.
On the other hand, we can choose $u$ such that the middle square
is a homotopy cartesian square by the Octahedral
Axiom~\cite[Lem.~1.4.3]{MR1812507}. After applying $j_*j^*$ to the middle
square and adjoining to it the natural square relating $u$ and $j_*j^*u$ we
obtain the commutative diagram:
\[
\xymatrix@C4pc{%
  j_*j^*N \ar[r]^-{j_*j^*\eta^j_N}\ar[d]^{j_*j^*\eta^f_N}
   & j_*j^*j_*j^*N \ar[d]^{j_*j^*u}
   & j_*j^*N \ar[l]_-{\eta^j_{j_*j^*N}}\ar[d]^{u}\\
  j_*j^*f_*f^*N \ar[r]^-{j_*j^*f_*f^*\eta^j_N}
   & j_*j^*f_*f^*j_*j^*N
   & f_*f^*j_*j^*N \ar[l]_-{\eta^j_{f_*f^*j_*j^*N}}. \\
}
\]
Since $j$ is a $\mathcal{T}$-flat monomorphism, the horizontal maps
are all isomorphisms and it follows
that $u=\eta^f_{j_*j^*N}$. Moreover, it is readily verified that the induced differential
\[
d:=l[1]\circ (\eta^f_C)^{-1}[1] \circ f_*f^*m\colon f_*f^*j_*j^*N \to N[1]
\]
is independent of the choice of the triangle $C\xrightarrow{l} N\xrightarrow{\eta^j_N} j_*j^*N\xrightarrow{m} C[1]$.  The functoriality of the
Mayer--Vietoris triangle now follows from the construction. Finally,
to show that $d$ is unique, if $d' \colon f_*f^*j_*j^*N \to N[1]$ is
another morphism that makes a distinguished triangle, then there is an
induced morphism of distinguished triangles:
\[
\xymatrix{N \ar[r] \ar@{=}[d] & \ar@{=}[d]j_*j^*N \oplus f_*f^*N \ar[r] & f_*f^*j_*j^*N
  \ar[r]^d \ar@{-->}[d]^{\theta} & N[1] \ar@{=}[d]\\
N \ar[r] & j_*j^*N \oplus f_*f^*N \ar[r] & f_*f^*j_*j^*N
  \ar[r]^{d'} & N[1].}
\]
It remains to show that $\theta$ is the identity morphism. Splitting up the sum in the middle square, we obtain the commutative diagram:
\[
\xymatrixcolsep{4pc}%
\xymatrix{%
 f_*f^*N \ar@{=}[d] \ar[r]^-{f_*f^*\eta^j_N}
  & f_*f^*j_*j^*N \ar[d]^{\theta}
  \\
 f_*f^*N \ar[r]^-{f_*f^*\eta^j_N}
  & f_*f^*j_*j^*N.}
\]
Applying $j_*j^*$ to this diagram, we may append another square on the right to obtain the commutative diagram:
\[
\xymatrix@C4pc{%
 j_*j^*f_*f^*N \ar@{=}[d] \ar[r]^-{j_*j^*f_*f^*\eta^j_N}
  & j_*j^*f_*f^*j_*j^*N \ar[d]^{j_*j^*\theta}
  & f_*f^*j_*j^*N \ar[d]^{\theta} \ar[l]_-{\eta_{f_*f^*j_*j^*N}^j}\\
 j_*j^*f_*f^*N \ar[r]^-{j_*j^*f_*f^*\eta^j_N}
  & j_*j^*f_*f^*j_*j^*N
  & f_*f^*j_*j^*N \ar[l]_-{\eta_{f_*f^*j_*j^*N}^j},}
\]
where the horizontal arrows are isomorphisms. It follows that
$\theta$ is the identity.

To obtain the isomorphism in \itemref{item:mv_triangle:adj}, first note that
by the definition of a Mayer--Vietoris square, the counit
$f^*f_*M \to M$ is an isomorphism since
$M \in \mathcal{T}_{X'\setminus U'}(X')$. This gives us a natural
isomorphism
\[
\Hom_{\mathcal{T}(X')}(M,f^*N)
\cong \Hom_{\mathcal{T}(X')}(f^*f_*M,f^*N)
\cong \Hom_{\mathcal{T}(X)}(f_*M,f_*f^*N).
\]
Now apply
    the homological functor
    $\Hom_{\mathcal{T}(X)}(f_*M,-)$ to the triangle $N \to
    j_*j^*N \oplus f_*f^*N \to f_*f^*j_*j^*N$ from
    \itemref{item:mv_triangle:MV}. Since
    $j^*f_*M\homotopic 0$ we obtain an isomorphism
    $\Hom_{\mathcal{T}(X)}(f_*M,N) \cong \Hom_{\mathcal{T}(X)}(f_*M,f_*f^*N)$
    and the result follows.

For \itemref{item:mv_triangle:desc}, the natural maps $v_j\colon j^*N \to N_U$ and $v_f \colon f^*N
    \to N'$ are obtained by adjunction from the maps $v^\vee_j \colon N \to
    j_*N_U$ and $v^\vee_f \colon N \to f_*N'$ in the defining triangle of
    $N$. The defining triangle exhibits $N$ as a homotopy pullback. We may
    thus find a morphism of distinguished triangles
    \cite[Lem.~1.4.4]{MR1812507} (Octahedral axiom):
    \[
    \xymatrix@C3pc{%
      C \ar[r]\ar@{=}[d] & N \ar[r]^-{v^\vee_j}\ar[d]^{v^\vee_f}
        & j_*N_U \ar[r]\ar[d]^{\eta^f_{j_*N_U}} & C[1]\ar@{=}[d] \\
      C \ar[r] & f_*N' \ar[r]^-{\alpha} & f_*f^*j_*N_U \ar[r] & C[1]. \\
    }
    \]
    Since $j$ is a $\mathcal{T}$-flat monomorphism, we have that $j^*\alpha$ is
    an isomorphism, so $j^*C\cong 0$. It follows that $j^*v^\vee_j\colon
    j^*N \to j^*j_*N_U$ is an isomorphism and hence that $v_j\colon j^*N\to
    N_U$ is an isomorphism
    (Lemma~\ref{lem:adm_basic}\itemref{lem:adm_basic:item:ff_bc}).

    Now, by \itemref{item:mv_triangle:MV}, we have a morphism of distinguished
    triangles:
    \[
    \xymatrix{%
      N \ar[r] \ar@{=}[d]
        & \ar[d]^-{j_*v_j \oplus f_*v_f} j_*j^*N \oplus f_*f^*N \ar[r]
        & \ar@{-->}[d]^{\theta} f_*f^*j_*j^*N\ar[r]^-{d} & N[1]\ar@{=}[d] \\
      N \ar[r]
        & j_*N_U \oplus f_*N' \ar[r]
        & f_*f^*j_*N_U \ar[r] & N[1].}
    \]
    As before, it follows that $\theta=f_*f^*j_*v_j$ by considering the
    application of $j_*j^*$ to the middle square.
    Since $v_j$ is an isomorphism, it follows that $f_*v_f$ is an isomorphism.
    Now, if we let $W$ be a cone of $v_f$, then $f_*W \homotopic
    0$. We will be done if we can show that $j'^*v_f$ is an
    isomorphism. Indeed, it would then follow that $j'^*W \homotopic
    0$ and so $W \homotopic f^*f_*W \homotopic f^*0 \homotopic
    0$. To this end, since the following diagram commutes:
    \[
    \xymatrix{j'^*f^*N \ar[r]^-{j'^*v_f} \ar[d]_{\mathrm{can}} & j'^*N'
      \ar[d]^{\delta}  \\ f_U^*j^*N \ar[r]^{f_U^*v_j} & f_U^*N_U, }
    \]
    and all appearing morphisms except $j'^*v_f$ are known to be
    isomorphisms, it follows that $j'^*v_f$ is an isomorphism.

For \itemref{item:mv_triangle:comp}, let $h \in \{j,f,j\circ f_U\}$. Since $h^*$ admits a right adjoint, it
    commutes with small coproducts. Thus if
    $\{Q_\lambda\}$ is a set of objects of $\mathcal{T}(X)$, then
    \begin{align*}
      \oplus_\lambda \Hom(N,h_*h^*Q_\lambda) &\cong
      \oplus_\lambda\Hom(h^*N,h^*Q_\lambda) \\
      &\cong \Hom(h^*N,\oplus_\lambda h^*Q_\lambda) \\
      &\cong \Hom(h^*N, h^*(\oplus_\lambda Q_\lambda)) \\
      &\cong \Hom(N,h_*h^*(\oplus_\lambda Q_\lambda)). 
    \end{align*}
    The result now follows by consideration of the Mayer--Vietoris
    triangles associated to $Q_\lambda$ and $\oplus_\lambda
    Q_\lambda$, together with the long exact sequence given by the
    homological functor $\Hom(N,-)$.
\end{proof}

In the following definition, we axiomatize the required properties of
open immersions of algebraic stacks.
\begin{definition}\label{defn:supports}
  Let $\mathcal{T}$ be a $\mathscript{D}$-presheaf of triangulated
  categories with adjoints. Let $\mathscript{L}$
  be a collection of morphisms in $\mathscript{D}$. We say that
  $\mathscript{L}$ \emph{supports $\mathcal{T}$} if it satisfies the
  following five conditions:
  \begin{enumerate}
  \item \label{defn:supports:mono} if $j \colon U \to X$ belongs to $\mathscript{L}$, then $j$ is
    a $\mathcal{T}$-flat monomorphism;
  \item \label{defn:supports:iso} if $j\colon U \to V$ is an isomorphism, then $j$ belongs to
    $\mathscript{L}$;
  \item \label{defn:supports:comp} if $j \colon U \to V$ and $k\colon V \to X$ belong to
    $\mathscript{L}$, then $k\circ j$ belongs to $\mathscript{L}$;
  \item \label{defn:supports:bc} if $j\colon U \to X$ belongs to $\mathscript{L}$ and $f\colon
    X' \to X$ is a morphism in $\mathscript{D}$, then the induced
    morphism $j'\colon U \times_X X' \to X'$ belongs to
    $\mathscript{L}$; and
  \item \label{defn:supports:union} if $i \colon U
    \to X$ and $j \colon V \to X$ belong to
    $\mathscript{L}$, then there exists a commutative diagram:
    \[
    \xymatrix{U\cap V \ar[d]^{\bar{\jmath}_U} \ar[r]_{\bar{\imath}_V} & V
      \ar[d]_{\bar{\jmath}} \ar@/^/[ddr]^j & \\ U \ar[r]^{\bar{\imath}}
      \ar@/_/[drr]_i & U \cup V \ar[dr]|k & \\ & & X,}
    \]
    where $k$ belongs to $\mathscript{L}$, and the square is
    a Mayer--Vietoris $\mathcal{T}$-square.
  \end{enumerate}
\end{definition}
Note that $U\cap V=U\times_X V$. 
\begin{example}\label{ex:dmodqc_supports}
  We continue Example \ref{ex:dmodqc_mv}. Let $\mathscript{I}$ be
  the collection of morphisms in $\REP[\mathrm{fp}]{Y}$ that are open
  immersions. By Example \ref{ex:dmodqc_mv} and standard arguments,
  $\mathscript{I}$ supports $\DQCOH$.
\end{example}
We now have a straightforward lemma.
\begin{lemma}\label{lem:glue_mono}
  Let $\mathcal{T}$ be a
  $\mathscript{D}$-presheaf of 
  triangulated categories with adjoints. Let
  $\mathscript{L}$ be a collection of
  morphisms in $\mathscript{D}$ that supports $\mathcal{T}$. If $i \colon U \to X$ and $j \colon V \to
  X$ belong to $\mathscript{L}$, then
  \begin{enumerate}
    \item\label{lem:glue_mono:item:int} $\mathcal{T}_{X\setminus U \cup V}(X) = \mathcal{T}_{X\setminus U}(X)\cap\mathcal{T}_{X\setminus V}(X)$, and
    \item\label{lem:glue_mono:item:exact} there is an exact sequence of triangulated categories: 
  \[
  \mathcal{T}_{X\setminus U \cup V}(X) \to \mathcal{T}_{X\setminus
    V}(X) \xrightarrow{i^*} \mathcal{T}_{U\setminus U\cap V}(U).
  \]
  \end{enumerate}
\end{lemma}
\begin{proof}
  \itemref{lem:glue_mono:item:int}
  Certainly, we have $\mathcal{T}_{X\setminus U\cup V}(X)
  \subset (\ker i^*) \cap (\ker j^*)$. For the other inclusion let $M
  \in (\ker i^*) \cap (\ker j^*)$. If $k \colon U \cup V \to X$ denotes the morphism induced by Definition \ref{defn:supports}\itemref{defn:supports:union}, then $k^*M \in \mathcal{T}(U\cup
  V)$. Now let $\bar{k} \colon U \cap V \to U \cup V$ denote the induced
  morphism. There is a triangle in $\mathcal{T}(U\cup V)$:
  \[
  k^*M \to \bar{\imath}_*\bar{\imath}^*k^*M \oplus
  \bar{\jmath}_*\bar{\jmath}^*k^*M \to \bar{k}_*\bar{k}^*k^*M.
  \]
  Functoriality induces isomorphisms $\bar{\imath}^*k^* \homotopic
  i^*$, $\bar{\jmath}^*k^* \homotopic j^*$, and $\bar{k}^*k^*
  \homotopic \bar{\jmath}_U^*i^*$. By hypothesis $i^*M$ and $j^*M$
  vanish, so the triangle gives $k^*M \homotopic 0$. Hence $M \in
  \mathcal{T}_{X\setminus U\cup V}(X)$.

  \itemref{lem:glue_mono:item:exact} Lemma
  \ref{lem:adm_basic}\itemref{lem:adm_basic:item:pb_supp} shows that the
  functor $i^* \colon \mathcal{T}_{X\setminus V}(X) \to \mathcal{T}(U)$ factors
  through $\mathcal{T}_{U\setminus U\cap V}(U)$ and the kernel is $(\ker i^*)
  \cap (\ker j^*)=\mathcal{T}_{X\setminus U\cup V}(X)$ by
  \itemref{lem:glue_mono:item:int}. Also, Lemma
  \ref{lem:adm_basic}\itemref{lem:adm_basic:item:im_supp} shows that
  $i_* \colon \mathcal{T}_{U \setminus U \cap V}(U) \to
  \mathcal{T}_{X\setminus V}(X)$ is a right adjoint to $i^*$ and Lemma
  \ref{lem:adm_basic}\itemref{lem:adm_basic:item:ff_bc} shows that the
  natural transformation $i^*i_* \to \ID{}$ is an isomorphism. It follows
  that the sequence is exact by Lemma \ref{lem:identify_quots}.
\end{proof}
\section{Descent of compact generation}\label{sec:descent-compact-gen}
For this section we fix a small category $\mathscript{D}$ that admits all 
finite limits. We also fix a
collection $\mathscript{L}$ of morphisms in $\mathscript{D}$.
\begin{definition}
  An \emph{admissible
    $(\mathscript{L},\mathscript{D})$-presheaf} of triangulated
  categories is a $\mathscript{D}$-presheaf $\mathcal{T}$
  of triangulated categories
  with adjoints (\S\ref{sec:presheaves-tri-cats}) satisfying
  \begin{enumerate}
  \item for all $X\in \mathscript{D}$, the triangulated category
    $\mathcal{T}(X)$ is closed under small coproducts;
  \item for all $(f\colon X\to Y)\in \mathscript{D}$, the push-forward
    $f_*\colon \mathcal{T}(X)\to \mathcal{T}(Y)$ preserves small
    coproducts; and
  \item $\mathscript{L}$ supports $\mathcal{T}$ (Definition~\ref{defn:supports}).
  \end{enumerate}
\end{definition}

\begin{example}\label{ex:dmodqc_adm}
  We continue Example \ref{ex:dmodqc_supports}: $\DQCOH$ is an
  admissible $(\mathscript{I},\REP[\mathrm{fp}]{Y})$-presheaf of
  triangulated categories. Indeed,
  the only non-trivial condition is that $f_*$ preserves small
  coproducts, which follows from Theorem~\ref{thm:fcd_coprod}\itemref{thm:fcd_coprod:item:sm_cp} since
  representable morphisms are concentrated.
\end{example}
\begin{definition}
Let $\beta$ be a cardinal and let $X \in \mathscript{D}$. We say that
an admissible $(\mathscript{L},\mathscript{D})$-presheaf of triangulated categories
$\mathcal{T}$ is \fndefn{compactly generated with $\mathscript{L}$-supports by $\beta$ 
  objects at $X$} if for every $j \colon V \to X$ in $\mathscript{L}$ the
triangulated category $\mathcal{T}_{X\setminus V}(X)$ is generated by
a set of cardinality $\leq \beta$ whose elements have compact image in
$\mathcal{T}(X)$. 
\end{definition}

In practice, $\mathscript{D}$ will often contain an initial object $\emptyset$ and for every
$X \in \mathscript{D}$ it will be the case that $(\emptyset \to X) \in \mathscript{L}$ and 
$\mathcal{T}(\emptyset) \simeq 0$. Hence, in this situation, 
$\mathcal{T}(X)=\mathcal{T}_{X\setminus \emptyset}(X)$ is also compactly
generated by a set of cardinality
$\leq \beta$. Also observe that if $\beta$ is a finite cardinal, then
$\mathcal{T}$ can always be compactly generated with supports by
\emph{one} object at $X$. In this section we will give conditions on
$\mathcal{T}$ that guarantee that the condition of compact generation
with $\mathscript{L}$-supports by $\beta$ objects descends along certain morphisms and
diagrams in $\mathscript{D}$.

Our first result is of an
elementary nature and is similar to the arguments of To\"en \cite[Lem.~4.11]{MR2957304}. First we require a
definition:
\begin{definition}
Let $\mathcal{T}$ be an admissible
$(\mathscript{L},\mathscript{D})$-presheaf of triangulated
categories. A morphism $f \colon X' \to X$ in
$\mathscript{D}$ is \fndefn{$\mathcal{T}$-quasiperfect with respect to
  $\mathscript{L}$} if the following three
conditions are satisfied:
\begin{enumerate}
\item $f$ is $\mathcal{T}$-flat (Definition \ref{defn:tflat});
\item\label{defn:quasiperfect:item:2} if $P \in \mathcal{T}(X')^c$, then $f_*P \in \mathcal{T}(X)^c$;
\item\label{defn:quasiperfect:item:3} $f_*$ admits a right adjoint $f^\times$ such that for every
 $j \colon V \to X$ in $\mathscript{L}$, the restriction of $f^\times$ to
  $\mathcal{T}_{X\setminus V}(X)$ factors through 
  $\mathcal{T}_{X'\setminus V\times_X X'}(X')$. 
\end{enumerate}
\end{definition}
By Example \ref{ex:pres_cpt_adj}, a potentially easy way to verify
condition~\itemref{defn:quasiperfect:item:2} above is for $f^\times$ to preserve small coproducts. To
verify condition~\itemref{defn:quasiperfect:item:3} above, it is sufficient to prove the following: for
every $j\colon V \to X$ in $\mathscript{L}$, if $j'\colon V'
\to X'$ is the pullback of $j$ along $f$ and $f_V \colon V'\to V$ is
the projection to $V$, then $f_*$ and $(f_V)_*$ both admit right adjoints and the
natural transformation $j'^*f^\times \to f_V^\times j^*$ is an
isomorphism of functors.
\begin{example}\label{ex:finiteflat_qp}
  We continue with Example \ref{ex:dmodqc_adm}. If $q\colon W' \to W$
  is a finite and faithfully flat morphism of finite presentation,
  then $q$ is $\DQCOH$-quasiperfect with respect to $\mathscript{I}$
  (Corollary \ref{C:fin_duality}). In
  \cite[App.~A]{hallj_dary_alg_groups_classifying}, we prove
  that if $q \colon W' \to W$ is a proper, smooth and locally schematic
  morphism of noetherian algebraic stacks, then $q$ is
  $\DQCOH$-quasiperfect with respect to $\mathscript{I}$. 
\end{example}
We now have the first important result of this section.
\begin{proposition}\label{prop:adm_finite_dev}
  Let $\beta$ be a cardinal. Let $\mathcal{T}$ be an admissible
  $(\mathscript{L},\mathscript{D})$-presheaf of triangulated categories. Let $f \colon X' \to
  X$ be a morphism in $\mathscript{D}$ that is
  $\mathcal{T}$-quasiperfect with respect to $\mathscript{L}$. If the functor $f^\times$ (which exists
  because $f$ is $\mathcal{T}$-quasiperfect) is conservative and
  $\mathcal{T}$ is compactly generated with $\mathscript{L}$-supports by $\beta$
  objects at $X'$, then $\mathcal{T}$ is compactly generated with
  $\mathscript{L}$-supports by $\beta$ objects at $X$. In fact, if  $j \colon V \to X$ belongs to $\mathscript{L}$, let $V'=X'\times_X V$ and let $\mathcal{B}' \subset \mathcal{T}(X')^c \cap
  \mathcal{T}_{X'\setminus V'}(X')$ be a subset of cardinality $\leq \beta$
  generating $\mathcal{T}_{X'\setminus V'}(X')$, then $f_*\mathcal{B}' \subseteq \mathcal{T}(X)^c \cap \mathcal{T}_{X\setminus V}(X)$  and $f_*\mathcal{B}'$ generates $\mathcal{T}_{X\setminus V}(X)$. 
\end{proposition}
\begin{proof}
  It suffices to prove the latter assertion. Set $\mathcal{B} = f_*\mathcal{B}'= \{
  f_*P \suchthat P \in \mathcal{B}'\}$. Then $\mathcal{B}$ has
  cardinality $\leq \beta$, and $\mathcal{B}
  \subset \mathcal{T}_{X\setminus V}(X)$ by Lemma
  \ref{lem:adm_basic}\itemref{lem:adm_basic:item:im_supp}. Moreover,
  $\mathcal{B} \subset \mathcal{T}(X)^c$, since $f$ is
  $\mathcal{T}$-quasiperfect with respect to $\mathscript{L}$.
  It remains to show that $\mathcal{B}$ generates
  $\mathcal{T}_{X\setminus V}(X)$. Let $N \in \mathcal{T}_{X\setminus
    V}(X)$ satisfy $\Hom_{\mathcal{T}(X)}(f_*P[n],N) = 0$ for all $P
  \in \mathcal{B}'$ and all $n\in \Z$. Since $f$ is
  $\mathcal{T}$-quasiperfect with respect to $\mathscript{L}$, it follows that $f^\times N \in
  \mathcal{T}_{X'\setminus V'}(X')$.  As $\mathcal{B}'$ is generating
  for $\mathcal{T}_{X'\setminus V'}(X')$, we may conclude that
  $f^\times N \homotopic 0$. By assumption, $f^\times$ is
  conservative. Thus $N \homotopic 0$ and $\mathcal{B}$ generates
  $\mathcal{T}_{X\setminus V}(X)$.
\end{proof}
Our next descent result is deeper, relying on Thomason's
Localization Theorem \ref{thm:thomason}. First, however, we require a lemma. 
\begin{lemma}\label{lem:cpct_basic}
  Let $\beta$ be a cardinal. Let $\mathcal{T}$ be an admissible
  $(\mathscript{L},\mathscript{D})$-presheaf of triangulated categories.  Suppose that
  $\mathcal{T}$ is compactly generated with $\mathscript{L}$-supports by $\beta$
  objects at $X \in \mathscript{D}$. Let $W \to V$ and $V \to X$
  belong to $\mathscript{L}$. Then the following holds.
  \begin{enumerate}
  \item\label{lem:cpct_basic:item:sub} $\mathcal{T}_{X\setminus W}(X)$
    is closed under small coproducts and the subcategory
    $\mathcal{T}_{X\setminus V}(X) \subset \mathcal{T}_{X\setminus
      W}(X)$ is localizing;
  \item\label{lem:cpct_basic:item:loc} $\mathcal{T}_{X\setminus V}(X)$
    is the localizing envelope of a set of compact objects of
    $\mathcal{T}_{X\setminus W}(X)$; and
  \item\label{lem:cpct_basic:item:comp} $\mathcal{T}_{X\setminus
      V}(X)^c = \mathcal{T}_{X\setminus W}(X)^c \cap
    \mathcal{T}_{X\setminus V}(X)$.
  \end{enumerate}
\end{lemma}
\begin{proof}
  First, observe that $\mathcal{T}(X)$ is closed under small
  coproducts. Also, if $f \colon X' \to X$ is a morphism in $\mathscript{D}$,
  then $f^*$ admits a right adjoint, so $f^*$ preserves small
  coproducts. Hence, we see that $\mathcal{T}_{X\setminus X'}(X)$ is a
  localizing subcategory of $\mathcal{T}(X)$. In particular,
  $\mathcal{T}_{X\setminus X'}(X)$ is closed under small
  coproducts. The claim \itemref{lem:cpct_basic:item:sub} is now
  immediate. 

  By hypothesis, $\mathcal{T}_{X\setminus V}(X)$ is generated by a
  subset $R$ that has compact image in $\mathcal{T}(X)$, hence also in
  $\mathcal{T}_{X\setminus W}(X)$. Let $\mathscript{R} \subset
  \mathcal{T}_{X\setminus W}(X)$ denote the localizing envelope of
  $R$, then $\mathscript{R} \subset \mathcal{T}_{X\setminus
    V}(X)$. Applying Corollary \ref{cor:thomason_gens} to $R \subset
  \mathcal{T}_{X\setminus V}(X)$, we find $\mathscript{R} =
  \mathcal{T}_{X\setminus V}(X)$, proving \itemref{lem:cpct_basic:item:loc}.
  The claim \itemref{lem:cpct_basic:item:comp} is now an immediate
  consequence of \itemref{lem:cpct_basic:item:loc} and Thomason's Theorem
  \ref{thm:thomason}. 
\end{proof}
\begin{proposition}\label{prop:nis_square}
  Let $\beta$ be a cardinal. Let $\mathcal{T}$ be an admissible
  $(\mathscript{L},\mathscript{D})$-presheaf of triangulated categories. Consider a Mayer--Vietoris
  $\mathcal{T}$-square (Definition \ref{defn:mv}):
  \[
  \xymatrix{U' \ar[d]_{f_U} \ar[r]^{j'} & X' \ar[d]^f\\ U
    \ar[r]^j & X}
  \] 
  with $j\in \mathscript{L}$.
  If $\mathcal{T}$ is compactly generated with $\mathscript{L}$-supports by $\beta$
  objects at $U$ and $X'$, then $\mathcal{T}$ is compactly generated
  with $\mathscript{L}$-supports by $\beta$ objects at $X$.
\end{proposition}
\begin{proof}
  Let $V \to X$ belong to $\mathscript{L}$. Form the cartesian cube:
  \[
  \xymatrix@R-3ex@C-2ex{& U'\cap V' \ar[dd]|\hole \ar[rr] \ar[dl] && V'
    \ar[dl] \ar[dd]^{f_V} \\U' \ar[rr]^-(0.7){j'} \ar[dd]_{f_U} & & X'
    \ar[dd]^(0.35){f} & \\ & U\cap V \ar[rr]|-\hole \ar[dl] & & V
    \ar[dl]\\ U
    \ar[rr]^{j} & & X. & }
  \]
  By Lemma \ref{lem:glue_mono}, we have an exact sequence
  \begin{equation}\label{eq:localization-X'}
  \mathcal{T}_{X'\setminus U'\cup V'}(X') \to
  \mathcal{T}_{X'\setminus V'}(X') \to \mathcal{T}_{U'\setminus U'\cap
    V'}(U').
  \end{equation}
  The category $\mathcal{T}_{X'\setminus V'}(X')$ is
  compactly generated and by Lemma
  \ref{lem:cpct_basic}\itemref{lem:cpct_basic:item:sub} it is also
  closed under small coproducts. By Lemma
  \ref{lem:cpct_basic}\itemref{lem:cpct_basic:item:loc} the
  subcategory $\mathcal{T}_{X'\setminus U'\cup V'}(X') \subset
  \mathcal{T}_{X'\setminus V'}(X')$ is the localizing envelope of a
  set of compact objects of $\mathcal{T}_{X'\setminus V'}(X')$.
  
  Now let $P \in \mathcal{T}(U)^c \cap \mathcal{T}_{U\setminus U\cap
    V}(U)$. Then $f_U^*P \in \mathcal{T}(U')^c$ since
  $(f_U)_*$ preserves coproducts (Example \ref{ex:pres_cpt_adj}) and
  $f_U^*P \in 
  \mathcal{T}_{U'\setminus U'\cap V'}(U')$ (Lemma \ref{lem:adm_basic}\itemref{lem:adm_basic:item:pb_supp}).
  Thus $f_U^*P\in \mathcal{T}_{U'\setminus U'\cap V'}(U')^c$ by Lemma
  \ref{lem:cpct_basic}\itemref{lem:cpct_basic:item:comp}. We now
  apply Thomason's localization Theorem, in the form of Corollary
  \ref{cor:thomason_lift}, to the
  exact sequence \eqref{eq:localization-X'}. This gives us
  $P' \in \mathcal{T}_{X'\setminus V'}(X')^c=\mathcal{T}(X')^c \cap
  \mathcal{T}_{X'\setminus V'}(X')$ and an isomorphism
  $j'^*P' \homotopic f_U^*(P\oplus P[1])$.  As in Lemma
  \ref{lem:mv_triangle}\itemref{item:mv_triangle:desc}, form the following triangle in $\mathcal{T}(X)$:
  \[
  \tilde{P} \to j_*(P\oplus P[1]) \oplus f_*P' \to
  f_*f^*j_*(P\oplus P[1]).
  \]
  By Lemma
  \ref{lem:mv_triangle}(\ref{item:mv_triangle:desc},\ref{item:mv_triangle:comp})
  we have that $j^*\tilde{P}\homotopic P\oplus P[1]$ and $f^*\tilde{P}\homotopic P'$ and that
  $\tilde{P} \in \mathcal{T}(X)^c$. Since $j_*j^*\tilde{P}$, $f_*f^*\tilde{P}$
  and $f_*f^*j_*j^*\tilde{P}\in \mathcal{T}_{X\setminus V}(X)$, it follows that
  $\tilde{P} \in \mathcal{T}_{X\setminus V}(X)$.

  Now let $Q\in \mathcal{T}_{X'\setminus U'\cup V'}(X')$, and note that
  $f_*Q\in\mathcal{T}_{X\setminus U\cup V}(X)$ (Lemma~\ref{lem:glue_mono}\itemref{lem:glue_mono:item:int}). Moreover, $f^*f_*Q\to Q$ is an
  isomorphism, because $f^*:\mathcal{T}_{X\setminus U}(X)\to
  \mathcal{T}_{X'\setminus U'}(X')$ is an equivalence of categories.  We also
  have that $j^*f_*Q \homotopic 0$ and $j'^*f^*f_*Q \homotopic 0$. Thus, if in
  addition $Q\in \mathcal{T}(X')^c$, then $f_*Q\in \mathcal{T}(X)^c$ by Lemma
  \ref{lem:mv_triangle}\itemref{item:mv_triangle:comp}.

  By hypothesis, there is a subset $\mathcal{B}_{0} \subset
  \mathcal{T}(U)^c \cap \mathcal{T}_{U\setminus U\cap V}(U)$ (resp.~$\mathcal{B}_1 \subset
  \mathcal{T}(X')^c \cap\mathcal{T}_{X'\setminus U'\cup V'}(X')$) of
  cardinality $\leq \beta$ generating $\mathcal{T}_{U\setminus U\cap V}(U)$
  (resp.~$\mathcal{T}_{X'\setminus U'\cup V'}(X')$). Define:  
  \[
  \mathcal{B} = \{\tilde{P}
  \suchthat P \in \mathcal{B}_0\} \cup \{f_*Q \suchthat Q \in \mathcal{B}_1\}.
  \]
  If $\beta$ is infinite, then the cardinality of $\mathcal{B}$ is
  $\leq \beta$, and if $\beta$ is finite then the same is true of
  $\mathcal{B}$. By the above considerations,
  $\mathcal{B} \subset \mathcal{T}(X)^c \cap \mathcal{T}_{X\setminus
    V}(X)$ and it remains to show that $\mathcal{B}$ generates
  $\mathcal{T}_{X\setminus V}(X)$. 
 
  Let $M \in \mathcal{T}_{X\setminus V}(X)$ so that $f^*M \in
  \mathcal{T}_{X'\setminus V'}(X')$ and
  $j^*M \in \mathcal{T}_{U\setminus U\cap V}(U)$.
  Suppose that
  $\Hom_{\mathcal{T}(X)}(f_*Q[n],M) =0$ for all $Q\in \mathcal{B}_1$ and all
  $n\in \Z$. By Lemma \ref{lem:mv_triangle}\itemref{item:mv_triangle:adj}, we see that
  $\Hom_{\mathcal{T}(X')}(Q[n],f^*M) = 0$ for all $Q\in \mathcal{B}_1$ and all
  $n\in \Z$. Let $K$ be a cone of $f^*M\to j'_*j'^*f^*M$. Note that
  \[
  \Hom_{\mathcal{T}(X')}(Q[n],j'_*j'^*f^*M) =
  \Hom_{\mathcal{T}(X')}(j'^*Q[n],j'^*f^*M)=0
  \]
  so
  $\Hom_{\mathcal{T}(X')}(Q[n],K) = 0$. Since $K\in \mathcal{T}_{X'\setminus
    U'\cup V'}(X')$ and $\mathcal{B}_1$ is generating, we
  see that $K \homotopic 0$, so $f_*f^*M\to
  f_*j'_*j'^*f^*M\homotopic f_*f^*j_*j^*M$ is
  an isomorphism. From the Mayer--Vietoris triangle $M
  \to j_*j^*M \oplus f_*f^*M \to f_*f^*j_*j^*M$, we deduce that the natural
  map $M \to j_*j^*M$ is an isomorphism for all such $M$. 

  Now suppose that $M$ also satisfies
  $\Hom_{\mathcal{T}(X)}(\tilde{P}[n],M) = 0$ for all $P\in \mathcal{B}_0$ and
  $n\in \Z$. Since the natural map $M \to j_*j^*M$ is an isomorphism,
  it follows that $\Hom_{\mathcal{T}(U)}(j^*\tilde{P}[n],j^*M) = 0$
  for all $P\in \mathcal{B}_0$ and $n\in \Z$. By Lemma
  \ref{lem:mv_triangle}\itemref{item:mv_triangle:desc}, $j^*\tilde{P} \homotopic
  P\oplus P[1]$ and so $\Hom_{\mathcal{T}(U)}(P[n],j^*M) = 0$ for all
  $P\in \mathcal{B}_0$ and all $n\in\Z$. By assumption, $\mathcal{B}_0$
  generates $\mathcal{T}_{U\setminus U\cap V}(U)$ and thus $j^*M
  \homotopic 0$. Since $M\homotopic j_*j^*M\homotopic 0$, we deduce
  that $\mathcal{B}$ generates $\mathcal{T}_{X\setminus V}(X)$. 
\end{proof}
We are now in a position to prove the main technical result of the
article. 
\begin{theorem}\label{thm:general_qffdesc_cpt_gen}
Let $X$ be a quasi-compact and quasi-separated algebraic stack. Let
$\mathscript{D}$ be $\REP[\mathrm{fp}]{X}$ or one of the full subcategories
$\REP[\mathrm{fp,qff,sep}]{X}$ or $\REP[\mathrm{fp},\et,\mathrm{sep}]{X}$. Let
$\mathscript{I}$ denote the set of open immersions in
$\mathscript{D}$. Let $\mathcal{T}$
be a presheaf of triangulated categories on $\mathscript{D}$. Assume that
\begin{enumerate}
\item\label{thm:general_qffdesc_cpt_gen:item:1} $\mathcal{T}(W)$ is closed under small coproducts for all $W \in \mathscript{D}$,
\item\label{thm:general_qffdesc_cpt_gen:item:2} for every morphism $f\colon W_1\to W_2$ in $\mathscript{D}$, the pullback
  $f^*\colon \mathcal{T}(W_2)\to \mathcal{T}(W_1)$ admits a right adjoint $f_*$
  that preserves small coproducts,
\item\label{thm:general_qffdesc_cpt_gen:item:3} for every cartesian square in $\mathscript{D}$
\[
\xymatrix{U_W \ar[d]_{f_W} \ar[r]^{g_U} & U \ar[d]^f\\ W \ar[r]^g & V,}
\]
such that $g$ is flat, the natural transformation $g^*f_* \to (f_W)_*(g_U)^*$
is an isomorphism,
\item\label{thm:general_qffdesc_cpt_gen:item:4}
  for every open immersion $U\to W$ and \'etale neighborhood $f\colon W'\to
  W$ of $W\setminus U$, the pullback $f^*$ induces an equivalence
  $\mathcal{T}_{W\setminus U}(W)\to \mathcal{T}_{W'\setminus U'}(W')$,
\item\label{thm:general_qffdesc_cpt_gen:item:5}
  for every finite faithfully flat morphism $W'\to W$ of finite
  presentation, the functor $f_*\colon \mathcal{T}(W')\to \mathcal{T}(W)$
  admits a right adjoint $f^\times$ that preserves small coproducts,
  is conservative, and commutes with pullback along open immersions.
\end{enumerate}
Let $\mathcal{C}\subseteq \mathcal{D}$ be the collection of all objects $W$
such that for every separated \'etale morphism $q\colon W' \to W$ in
$\mathcal{D}$ and every open immersion $V'\to W'$ in $\mathscript{I}$, the
triangulated category $\mathcal{T}_{W'\setminus V'}(W')$ is generated by a set
of cardinality $\leq \beta$ whose elements have compact image in
$\mathcal{T}(W')$.

If $p \colon W\to X$ is a separated, quasi-finite and faithfully flat morphism
in $\mathscript{D}$ such that $W\in \mathcal{C}$, then $X\in \mathcal{C}$.
\end{theorem}
\begin{proof}
  Condition~\itemref{thm:general_qffdesc_cpt_gen:item:2} says that $\mathcal{T}$ has adjoints. 
  Condition~\itemref{thm:general_qffdesc_cpt_gen:item:3} says that flat morphisms are
  $\mathcal{T}$-flat. Conditions~\itemref{thm:general_qffdesc_cpt_gen:item:3}
  and~\itemref{thm:general_qffdesc_cpt_gen:item:4} imply that \'etale
  neighborhoods
  are Mayer--Vietoris squares. In particular, $\mathscript{I}$
  supports $\mathcal{T}$. Combining this with conditions~\itemref{thm:general_qffdesc_cpt_gen:item:1} and~\itemref{thm:general_qffdesc_cpt_gen:item:2}, we
  conclude that $\mathcal{T}$ is an admissible
  $(\mathscript{I},\mathscript{D})$-presheaf of triangulated
  categories. 

  By assumption, there exists an object $W\in \mathcal{C}$ with
  $W\to X$ separated, quasi-finite and faithfully flat.
  We will apply~\cite[Thm.~6.1]{MR2774654} to deduce that
  $X \in \mathcal{C}$. To do this, we need to verify the following 
  three conditions for a flat morphism $q\colon W' \to W$ in
  $\mathcal{D}$. 
\begin{itemize}
\item[(D1)] If $W \in \mathcal{C}$  and $q$ is \'etale and
  separated, then $W' \in \mathcal{C}$; 
\item[(D2)] if $W' \in \mathcal{C}$ and $q$ is finite and surjective,
  then $W \in \mathcal{C}$; and
\item[(D3)] if $q$ is an \'etale neighborhood of $W\setminus U$,
  where $U \to W$ is an open immersion in $\mathcal{D}$, 
  and $U$ and $W'$ belong to $\mathcal{C}$, then $W \in \mathcal{C}$.
\end{itemize}
Now (D1) tautologically follows from the definition of
$\mathcal{C}$. For (D2), Condition~\itemref{thm:general_qffdesc_cpt_gen:item:5} implies that $p$ is
$\mathcal{T}$-quasiperfect with respect to $\mathscript{I}$. By
Proposition~\ref{prop:adm_finite_dev}, we deduce that (D2) is
satisfied.  As noted previously, Conditions~\itemref{thm:general_qffdesc_cpt_gen:item:3} and~\itemref{thm:general_qffdesc_cpt_gen:item:4} imply that
\'etale neighborhoods are Mayer--Vietoris $\mathcal{T}$-squares. By
Proposition~\ref{prop:nis_square}, (D3) is
satisfied. The result follows.
\end{proof}

\section{Algebraic stacks with the \texorpdfstring{$\beta$-}{}resolution property}
Let $X$ be an algebraic stack. 
Recall that $X$ is said to have the \fndefn{resolution property} if every quasi-coherent $\Orb_X$-module $M$ is a quotient of a direct sum of locally free $\Orb_X$-modules of finite type. 
The resolution property is a subtle and difficult property, although it is always satisfied for quasi-projective schemes. It has been studied systematically by several authors, with notable contributions due to Thomason \cite{MR893468}, Totaro \cite{MR2108211}, and Gross \cite{gross_phdthesis,2013arXiv1306.5418G}.

The following simple refinements of the resolution property will be useful for us. Let $\VB{X}\subseteq \QCOH(X)$ denote the subcategory of locally free
$\Orb_X$-modules of finite type.
Let $\beta$ be a cardinal. 
We say that $X$ has the \fndefn{$\beta$-resolution property} if there exists a subset 
$\mathcal{B}\subseteq \VB{X}$ of cardinality $\leq \beta$, with the property that every 
quasi-coherent $\Orb_X$-module $M$ is a quotient of a direct sum of elements of 
$\mathcal{B}$. If, in addition, it can be arranged that $\mathcal{B}$ consists of vector 
bundles that are compact objects of $\DQCOH(X)$, then we say that $X$ has the 
\fndefn{compact} $\beta$-resolution property. These conditions are equivalent to: 
every quasi-coherent
$\Orb_X$-module $M$ is the quotient of a (possibly infinite) direct sum of
objects in $\mathcal{B}$. 

If $X$ is concentrated, then every locally free $\Orb_X$-module of finite type is a compact 
object of $\DQCOH(X)$ (Lemma 
\ref{L:compact_conc}\itemref{L:compact_conc:perf_conc_c}). In particular, if $X$ also 
has the $\beta$-resolution property, then $X$ has the compact $\beta$-resolution 
property. Since quasi-compact, quasi-separated schemes and algebraic spaces are 
concentrated, the compact $\beta$-resolution property and the $\beta$-resolution property 
coincide for schemes and algebraic spaces.

The following simple Lemma will be important.
\begin{lemma}[{\cite[Prop.~1.8 (v)]{2013arXiv1306.5418G}}]\label{lem:quaff_beta_res}
  Let $f \colon X \to Y$ be a quasi-affine morphism of algebraic stacks.
  Let $\beta$ be a cardinal.
  If $Y$ has the $\beta$-resolution property or the compact $\beta$-resolution property, then so does $X$.
\end{lemma}
\begin{proof}
  If $\mathcal{B}\subseteq \VB{Y}$ is a resolving set of cardinality $\beta$, then
  $f^*\mathcal{B}=\{f^*E\;:\;E\in \mathcal{B}\}$ is a resolving set of cardinality $\beta$. Indeed, $f^*$ is right-exact and $f^*f_*M\to
  M$ is surjective for every $M\in\QCOH(X)$. In addition, if $\mathcal{B}$ consists of compact objects 
  of $\DQCOH(Y)$, then $f^*\mathcal{B}$ consists of compact objects in $\DQCOH(X)$ 
  (Example \ref{ex:pres_cpt_conc})
\end{proof}
\begin{remark}\label{rem:finflat-descent-of-res-prop}
Similarly, there is the following partial converse: if $f\colon
X \to Y$ is finite and faithfully flat of finite presentation and $X$ has the
$\beta$-resolution property or compact $\beta$-resolution property, then so does $Y$. In this case, one takes $f_*\mathcal{B}$
as the resolving set and uses that $f_*$ is right-exact and that $f_*f^!M\to M$
is surjective~\cite[Prop.~1.13]{2013arXiv1306.5418G}. But $f^!=f^\times$
preserves coproducts (Corollary~\ref{C:fin_duality}), so $f_*$ preserves
compact objects (Example~\ref{ex:pres_cpt_adj}).
\end{remark}

\begin{proposition}\label{prop:beta-res}
  Let $X$ be an algebraic stack. If $X$ is quasi-affine, then $X$ has the
  $1$-resolution property. If $X$ is quasi-compact and quasi-separated with
  affine stabilizer groups, then the following are equivalent:
\begin{enumerate}
\item\label{prop:beta-res:aleph0}
  $X$ has the $\aleph_0$-resolution property;
\item\label{prop:beta-res:arb}
  $X$ has the resolution property; and
\item\label{prop:beta-res:qaff/GLn}
  $X=[V/\GL_{n,\Z}]$ where $V$ is a quasi-affine scheme.
\end{enumerate}
When these conditions hold, $X$ has affine diagonal.
\end{proposition}
\begin{proof}
The first statement follows from the fact that $\Orb_X\in \QCOH(X)$ is a
generator if $X$ is quasi-affine. Trivially,
\itemref{prop:beta-res:aleph0}$\implies$\itemref{prop:beta-res:arb}. That
\itemref{prop:beta-res:arb}$\implies$\itemref{prop:beta-res:qaff/GLn} is Totaro's
theorem~\cite{MR2108211,2013arXiv1306.5418G}.
Finally, to see that
\itemref{prop:beta-res:qaff/GLn}$\implies$\itemref{prop:beta-res:aleph0} it is
enough to prove that $B\GL_{n,\Z}$ has the $\aleph_0$-resolution property since
$X\to B\GL_{n,\Z}$ is quasi-affine. That $B\GL_{n,\Z}$ has the resolution
property is a special case of~\cite[Lem.~2.4]{MR893468}: every coherent sheaf
on $B\GL_{n,\Z}$ is a quotient of a finite-dimensional subrepresentation of a
finite number of copies of the regular representation.
%
%
%
Since there is a
countable number of vector bundles on $B\GL_{n,\Z}$, the $\aleph_0$-resolution
property holds.

The last statement follows since $[V/\GL_{n,\Z}]$ has affine diagonal.
\end{proof}

\begin{question}
Every algebraic stack that admits a finite flat cover $V\to X$ with $V$
quasi-affine has the compact $1$-resolution property by
Remark~\ref{rem:finflat-descent-of-res-prop}. Are all stacks with
the (compact) $1$-resolution property of this form? Is every algebraic space with the
$1$-resolution property quasi-affine?
\end{question}

Many quotient stacks have the resolution property:
\begin{example}\label{ex:res-prop}
Let $S=\spec R$ be a regular scheme of dimension at most $1$ (e.g.,
$R=\Z$ or $R$ is a field). Let $G\to S$ be a flat affine group
scheme of finite type and let $V$ be an algebraic space with an action of
$G$. Then $X=[V/G]$ has the resolution property in the following cases~\cite[Thm.~2.1]{MR2108211}:
\begin{enumerate}
\item\label{ex:item:qaff} $V$ is quasi-affine.
\item\label{ex:item:normal-ample} $V$ is normal, noetherian and has an ample family of line bundles
  (e.g.\ $V$ quasi-projective) and $G$ is an extension of a finite flat group
  scheme by a smooth group scheme with connected fibers (this is automatic if
  $R$ is a field).
\item\label{ex:item:linear}
  $V$ has an ample family of $G$-equivariant line bundles
  (e.g.\ $V$ is quasi-projective and $G$ is acting linearly).
\end{enumerate}
For~\itemref{ex:item:qaff}, use that $BG_S$ has the resolution property~\cite[Lem.~2.4]{MR893468}
and Lemma~\ref{lem:quaff_beta_res}. For~\itemref{ex:item:normal-ample} use~\cite[Lem.~2.10 and
  2.14]{MR893468} and for~\itemref{ex:item:linear} use~\cite[Lem.~2.6]{MR893468}.
Note that in~\itemref{ex:item:normal-ample} it is crucial that $V$ is normal to apply
Sumihiro's theorem~\cite[Thm.~1.6]{MR0387294} and deduce that sufficiently high
powers of the line bundles carry a $G$-action. In fact, this is false otherwise
(cf.~\cite[\S9]{MR2108211}) and whether $X$ has the resolution property in this
case is not known in general. Alternatively, we could assume that $V$ is
a quasi-projective scheme with a linear action of $G$, as
in \cite[Cor.~3.22]{MR2669705}.
\end{example}
We conclude this section with a simple lemma.
\begin{lemma}\label{L:uniform_compact_beta_res}
  Let $X$ be a quasi-compact and quasi-separated algebraic stack. The
  following conditions are equivalent:
  \begin{enumerate}
  \item \label{LI:uniform_compact_beta_res:allcompact} $X$ has the
    compact $\beta$-resolution property.
  \item \label{LI:uniform_compact_beta_res:1compact} $X$ has the
    $\beta$-resolution property and there exists a vector bundle $F$
    such that $\supp(F)=X$ and $F$ is a compact object of $\DQCOH(X)$.
  \item \label{LI:uniform_compact_beta_res:1compactbis} There is a
    vector bundle $F$ on $X$ that is a compact object of $\DQCOH(X)$
    and a subset $\mathcal{B} \subseteq \VB{X}$ of cardinality
    $\leq \beta$ such that the set $\{F\otimes E\}_{E \in \mathcal{B}}$
    generates $\QCOH(X)$.
  \item \label{LI:uniform_compact_beta_res:uniform} There is a subset
    $\mathcal{B} \subseteq \VB{X}$ of cardinality $\leq \beta$ that
    generates $\QCOH(X)$ and an integer $r\geq 0$ such that
    $\Ext^i_{\Orb_X}(E,N)=0$ for all $i>r$, all $E\in \mathcal{B}$ and
    all $N\in \QCOH(X)$.
  \item \label{LI:uniform_compact_beta_res:uniformbis} There is a subset
    $\mathcal{B} \subseteq \VB{X}$ of cardinality $\leq \beta$ that
    generates $\QCOH(X)$ and an integer $r\geq 0$ such that the
    natural map
    \[
      \trunc{\geq j}\RHom_{\Orb_X}(E,\cplx{M}) \to \trunc{\geq
        j}\RHom_{\Orb_X}(E,\trunc{\geq j-r}\cplx{M})
    \]
    is a quasi-isomorphism for all integers $j$, all $E\in \mathcal{B}$ and
    all $\cplx{M}\in \DQCOH(X)$.
  \end{enumerate}
  \end{lemma}
  \begin{proof}
    \itemref{LI:uniform_compact_beta_res:allcompact}$\Rightarrow$\itemref{LI:uniform_compact_beta_res:1compact}:
    there is a surjection $F \twoheadrightarrow O_X$ with $F$ a vector
    bundle on $X$ that is compact in $\DQCOH(X)$. Note that
    $\supp(F)=X$.

    \itemref{LI:uniform_compact_beta_res:1compact}$\Rightarrow$\itemref{LI:uniform_compact_beta_res:1compactbis}:
    If $\mathcal{C}\subseteq\VB{X}$ is a resolving set of vector
    bundles, then so is the set $\{F\otimes F^*\otimes E\}_{E\in
      \mathcal{C}}$. Indeed, the
evaluation map $F\otimes F^* \twoheadrightarrow \Orb_X$ is surjective. We take
    $\mathcal{B} = \{F^*\otimes E \}_{E\in \mathcal{C}}$.

    \itemref{LI:uniform_compact_beta_res:1compactbis}$\Rightarrow$\itemref{LI:uniform_compact_beta_res:uniform}:
    for all $i$, all $E\in \mathcal{B}$, and $N\in \QCOH(X)$ we have
    $\Ext^i_{\Orb_X}(F\otimes E,N)=\Ext^i_{\Orb_X}(F,E^*\otimes
    N)$. Now choose $r\geq 0$ as in Lemma \ref{lem:char_compact_stack}
    for $F$. An identical argument gives
    \itemref{LI:uniform_compact_beta_res:1compactbis}$\Rightarrow$\itemref{LI:uniform_compact_beta_res:uniformbis}. Also,
    \itemref{LI:uniform_compact_beta_res:uniformbis}$\Rightarrow$\itemref{LI:uniform_compact_beta_res:uniform}
    is trivial.

    \itemref{LI:uniform_compact_beta_res:uniform}$\Rightarrow$\itemref{LI:uniform_compact_beta_res:allcompact}:
    immediate from Lemma \ref{lem:char_compact_stack}.
  \end{proof}

\section{\Crisp{} stacks}
In this section we define $\beta$-\crisp{} stacks and show that the
compact $\beta$-resolution property implies $\beta$-\crispness.

\begin{definition}\label{def:beta-perfect}
  Let $\beta$ be a cardinal. Let $X$ be an algebraic stack. We say that $X$ satisfies the 
  \fndefn{$\beta$-Thomason condition} if:
  \begin{enumerate}
  \item $\DQCOH(X)$ is compactly generated by a set of cardinality $\leq \beta$; and
  \item for every quasi-compact open immersion $j\colon U \hookrightarrow X$ there
    exists a perfect object $P$ of $\DQCOH(X)$ with support $X\setminus U$.
  \end{enumerate}
   We say that $X$ is \fndefn{$\beta$-\crisp} if for every representable, \'etale, separated, 
   and quasi-compact morphism $X' \to X$ the stack $X'$ satisfies the $\beta$-Thomason 
   condition.
\end{definition}
By Lemmas~\ref{lem:supports_and_supported_complexes}
and~\ref{lem:compactgen-stable-quasiaffine}, an equivalent definition for $\beta$-crispness is
that for every representable \'etale morphism $W \to X$ that is quasi-compact and separated,
and every quasi-compact open immersion $j\colon U \hookrightarrow W$, the
triangulated category
\[
\DQCOH[,|W|\setminus |U|](W) = \{ M \in
\DQCOH(W) \suchthat j^*M \homotopic 0\}
\]
is generated by a set of
cardinality $\leq \beta$ consisting of compact objects of
$\DQCOH(W)$.
\begin{lemma}\label{lem:compactgen-stable-quasiaffine}
  Let $f \colon X \to Y$ be a quasi-affine morphism of algebraic stacks.
  Let $\beta$ be a cardinal.
  If $\DQCOH(Y)$ is compactly generated by $\beta$ objects, then so is $\DQCOH(X)$. In fact, if $\mathcal{B} \subseteq \DQCOH(Y)^c$ is a subset that generates $\DQCOH(Y)$, then the set $\{ \LDERF \QCPBK{f}B \suchthat B \in \mathcal{B}\}$ compactly generates $\DQCOH(X)$. 
\end{lemma}
\begin{proof}
  Since $f$ is concentrated, $\RDERF\QCPSH{f}$ preserves small coproducts. Since $\RDERF \QCPSH{f}$ is conservative (Corollary \ref{C:qaff_cons}), the result
  follows from Example \ref{ex:compact_cons_gen}.
\end{proof}
\begin{remark}
We do not know if the analogue of Lemma~\ref{lem:compactgen-stable-quasiaffine}
holds for the $\beta$-Thomason condition or even $\beta$-\crispness{}.
\end{remark}
The main result of this section is the following 
\begin{proposition}\label{prop:perf_gen_res_property}
  Let $X$ be a quasi-compact algebraic stack with affine diagonal.
  Let $\beta$ be a cardinal. 
  If $X$ has the compact $\beta$-resolution property, then it is $\beta$-\crisp. In 
  particular, concentrated stacks with the $\beta$-resolution property are 
  $\beta$-crisp. In fact, $\DQCOH(X)$ is compactly generated by any resolving set of 
  compact vector bundles. 
\end{proposition}
\begin{proof}
  By Lemma \ref{lem:quaff_beta_res} the compact $\beta$-resolution property is preserved under quasi-affine morphisms.
  By Zariski's Main Theorem \cite[Thm.~16.5]{MR1771927} \'etale
  morphisms that are quasi-compact, separated, and representable are
  quasi-affine. Thus it is enough to prove the following statement: if $j \colon V \hookrightarrow X$ is a quasi-compact open immersion with complement $|Z|$, then there exists a generating subset $\mathcal{B}_{|Z|} \subset \DQCOH[,|Z|](X)$, of cardinality $\leq \beta$, with compact image in $\DQCOH(X)$. 

  Choose a resolving set $\mathcal{B}\subseteq \VB{X}$ of
  cardinality $\leq \beta$ and an integer $r\geq 0$ as in Lemma
  \ref{L:uniform_compact_beta_res}\itemref{LI:uniform_compact_beta_res:uniformbis}. Let $M \in \DQCOH(X)$. We claim that if $n\in \Z$ is such that $\COHO{n}(M) \neq 0$, then there exists an $E\in \mathcal{B}$ and a non-zero morphism $E[-n] \to M$ in $\DQCOH(X)$. 
  We prove this claim by a small modification (which is likely
  well-known---e.g. \cite[Rem.~1.2.10]{MR3037900}) to the argument of A.~Neeman \cite[Ex.~1.10]{MR1308405}.
  
  Thus, for all $E \in
  \mathcal{B}$, all $n\in \Z$ and all $M \in \DQCOH(X)$ we have
  \[
  \Hom_{\Orb_X}(E,M[n]) = \Hom_{\Orb_X}(E,\trunc{\geq -r}M[n]).
  \]
  We may consequently assume that $M \in \DQCOH^+(X)$. 
  By \cite[Thm.~3.8]{lurie_tannaka}, the natural functor $\DCAT^+(\QCOH(X)) \to \DQCOH^+(X)$ is an equivalence of triangulated categories. 
  Hence, we are free to assume that $M$ is a complex $(\cdots\rightarrow M^k \xrightarrow{d^{k}} M^{k+1}\rightarrow\cdots )$  of quasi-coherent $\Orb_X$-modules. 
  By assumption $\COHO{n}(M) \neq 0$, so there exists $E\in \mathcal{B}$ and a
  morphism $E\to \ker (d^{n})$ such that $E\to \ker (d^{n})\to \COHO{n}(M)$
  is non-zero.
  The composition $E \to \ker(d^{n}) \to M^{n}$ thus induces a non-zero morphism $E \to M[n]$ in $\DQCOH(X)$ and we deduce the claim.

  We now return to the proof of the Proposition. The above considerations shows
  that
  the set $\mathcal{B}$ compactly generates $\DQCOH(X)$. Now let $i \colon Z \hookrightarrow X$ be a closed immersion with support $|Z|$. 
  Since $X$ has the resolution property and $j \colon V \to X$ is quasi-compact, we may choose $i$ such that the quasi-coherent ideal sheaf $I$ defining $Z$ in $X$ is of finite type. 
  It follows that there is a surjection $F \to I$, where $F$ is a finite direct sum of objects of $\mathcal{B}$.
  Corresponding to the morphism $s \colon F \to \Orb_X$, we obtain a section $s^\vee \in \Gamma(X,F^\vee)$ with vanishing locus $|Z|$. 
  If $K(s^\vee)$ is the resulting Koszul complex
  \cite[IV.2]{MR801033}, 
  then $K(s^\vee)$ is a perfect complex on $X$ with support $|Z|$. By
  Lemma
  \ref{lem:supports_and_supported_complexes}\itemref{lem:supports_and_supported_complexes:item:2},
  we deduce the claim.
\end{proof}
We conclude this section with examples of algebraic stacks that are \crisp.
\begin{example}\label{ex:res_ring}
  Let $A$ be a ring. 
  Then $\spec A$ is $1$-\crisp. 
  Indeed, $\spec A$ has the $1$-resolution property and is concentrated,
  thus the result follows from Proposition \ref{prop:perf_gen_res_property}. 
\end{example}
\begin{example}\label{ex:res_bt}
  Let $X$ be a concentrated stack with affine stabilizers and the resolution
  property. Then $X$ has the $\aleph_0$-resolution property and affine diagonal
  (Proposition \ref{prop:beta-res}), hence is $\aleph_0$-\crisp{} (Proposition~\ref{prop:perf_gen_res_property}).

  Examples are stacks
  of the form $X=[V/G]$ where $V$ and $G$ are as in Example~\ref{ex:res-prop}
  and either $S=\spec(\Q)$ or $G$ is linearly reductive (e.g., a torus).
  Indeed, under these
  assumptions on $G$, the classifying stack $BG$ is concentrated, so $X$ is
  concentrated since $X\to BG$ is representable. More generally, we can take
  any stack $X=[V/G]$ as in Example~\ref{ex:res-prop} with linearly reductive
  stabilizers. Such stacks are concentrated by
  \cite[Thm.~C]{hallj_dary_alg_groups_classifying}. 
\end{example}
\begin{remark}
  Let $X$ be a quasi-compact algebraic stack with affine diagonal and the
  resolution property. When $X$ is concentrated, then $\DQCOH(X)$ is compactly
  generated (Proposition \ref{prop:perf_gen_res_property}) and $\DQCOH(X)$ is
  an example of a unital algebraic stable homotopy
  category~\cite[Def.~1.1.4]{MR1388895}. Note that the localizing envelope of a
  set of compact generators is the whole category (Corollary~\ref{cor:thomason_gens}).

  The proof of Proposition
  \ref{prop:perf_gen_res_property} actually shows that even if 
  $X$ is not concentrated, then $\DCAT(\QCOH(X))$ is perfectly generated. Note that since $\DCAT(\QCOH(X))$ may not be compactly generated, Corollary \ref{cor:thomason_gens} does not apply. Nonetheless $\DCAT(\QCOH(X))$ is well-generated \cite[Thm.~0.2]{MR1874232} and there is a version of Corollary \ref{cor:thomason_gens} for well-generated triangulated categories \cite[Thm.~1.14]{MR1812507}. This result and others are also discussed in \cite{2016arXiv160406018A}. 
  Thus, $\DCAT(\QCOH(X))$ is a non-algebraic stable homotopy category in the sense
  of~\cite[Def.~1.1.4]{MR1388895}.
  Note that 
  this says nothing about perfect or compact generation of $\DQCOH(X)$, because the 
  functor $\DCAT(\QCOH(X)) \to \DQCOH(X)$ can fail to be fully faithful or essentially 
  surjective (e.g., if $X=B\Ga$ in positive characteristic 
  \cite{hallj_neeman_dary_no_compacts}). Compact generation of $\DQCOH(X)$, however, is sufficient to prove that 
  $\DCAT(\QCOH(X)) \to \DQCOH(X)$ is an equivalence   \cite{hallj_neeman_dary_no_compacts}.
\end{remark}
\section{Quasi-finite flat locality of
  \texorpdfstring{$\beta$-}{}\crispness{} and applications}\label{sec:applications}
We are now in a position to prove Theorems \ref{thm:main_qf},
\ref{thm:main_globaltype}, and \ref{thm:main_list} as well as
addressing the applications mentioned in the Introduction.
\begin{proof}[Proof of Theorem \ref{thm:main_list}]
  Take $\mathscript{D} = \REP[\mathrm{fp}]{X}$. By
  Examples~\ref{ex:dmodqc_defn}, \ref{ex:dmodqc_mv}, \ref{ex:dmodqc_adm}
  and~\ref{ex:finiteflat_qp}, the $\mathscript{D}$-presheaf of triangulated
  categories $\DQCOH$ satisfies Conditions
  \itemref{thm:general_qffdesc_cpt_gen:item:1}--\itemref{thm:general_qffdesc_cpt_gen:item:5} of Theorem \ref{thm:general_qffdesc_cpt_gen}. The result now
  follows from Lemma \ref{lem:supports_and_supported_complexes} and
  Theorem \ref{thm:general_qffdesc_cpt_gen}.
\end{proof}
Proving Theorems \ref{thm:main_qf} and \ref{thm:main_globaltype} is now very simple.
\begin{proof}[Proof of Theorem \ref{thm:main_qf}]
  By \cite[Thm.~7.1]{MR2774654} there exists a locally quasi-finite flat morphism $p \colon X' \to X$, where $X'$ is a scheme.
  Since $X$ is quasi-compact, we may further assume that $X'$ is an affine scheme, and consequently the morphism $p$ is also quasi-compact and separated. 
  The result now follows by combining Example \ref{ex:res_ring} with Theorem \ref{thm:main_list}. 
\end{proof}
\begin{proof}[Proof of Theorem \ref{thm:main_globaltype}]
  Since $X$ is of s-global type, there exists an integer $N>0$, a quasi-affine $\Q$-scheme $V$ with an action of $\GL_{N}$, together with an \'etale, representable, separated and finitely presented morphism $p \colon [V/\GL_N] \to X$.
  Now the result follows from Theorem \ref{thm:main_list} and Example \ref{ex:res_bt}.   
\end{proof}

We now recall some results of Sumihiro and Brion.
\begin{proposition}[Sumihiro and Brion]\label{P:sglobal_sumi}
 Let $X$ be a variety over a field $k$. Let $G$ be an
 affine algebraic $k$-group scheme acting on $X$. Assume that either $X$ is 
 \begin{enumerate}
 \item geometrically normal, or 
 \item quasi-projective and either
   \begin{enumerate}
   \item geometrically semi-normal, or
   \item $\kar k=p>0$, or
   \item the action is linearizable, or
   \item $G^0$ is a torus.
   \end{enumerate}
 \end{enumerate}
 Then there exists a finite field extension $k'/k$, a quasi-projective variety
 $W'$ over $k'$ with a linear action of $G'=\red{(G\tensor_k k')}^0$ and an
 \'etale $G'$-equivariant morphism $f\colon W'\to X_{k'}$.
 %
\end{proposition}
\begin{proof}
  Choose a finite field extension $k'/k$ such that $X'=X\times_k k'$ is normal
  (resp.\ semi-normal, resp.\ $G'$ is smooth, resp.\ $G'$ is a split torus).
  By construction $G'$ is then a smooth connected group scheme.

  If $X'$ is normal, then by Sumihiro's theorem, we
  can choose $f$ as a Zariski-open covering \cite[Lem.~8]{MR0337963}
  (see~\cite[Thm.~3.8]{MR0387294} when $k$ is not algebraically closed, or
  replace $k$ with a finite field extension).
  If $X'$ is semi-normal and quasi-projective, then an \'etale $f$ exists by
  Brion~\cite[Thm.~4.7]{MR3329192}. If $\kar k>0$, then an \'etale $f$ exists
  by~\cite[\S 4.3]{MR3329192}. If the action is already linearizable, then let
  $f$ be the identity. If $G'$ is a split torus, then an \'etale $f$, with $W'$
  affine, exists by~\cite[Thm.~4.8]{MR3329192}.
\end{proof}
In \cite[Thm.~2.6]{AHR_lunafield} it is proved that if $G$ acts with linearly
reductive stabilizers at closed points (e.g., if $G$ is linearly reductive),
then the result of Proposition~\ref{P:sglobal_sumi} holds for any algebraic
space of finite type (not necessarily normal, quasi-projective or even
separated). In particular, if $G^0$ is a torus, we may drop the requirement
that $X$ is quasi-projective.

The Corollary of Theorem \ref{thm:main_globaltype} is a special case of the following.
\begin{corollary}\label{C:compact-generated-stacks}
Let $(X,G,k)$, and $(W',G',k')$ be as in Proposition~\ref{P:sglobal_sumi}. Then
\begin{enumerate}
\item\label{C:cg:res-prop} $[W'/G']$ has the resolution property;
\item\label{C:cg:qff} the map $[W'/G']\to [X'/G']\to [X/G]$ is quasi-finite and faithfully flat; and
\item\label{C:cg:s-global} $[X/G]$ is of s-global type.
\end{enumerate}
If in addition, $\kar k=0$ or $\red{(G\tensor_k \bar{k})}^0$ is a torus, then
\begin{enumerate}\setcounter{enumi}{3}
\item\label{C:cg:conc-crisp} $[W'/G']$ is concentrated and $\aleph_0$-crisp;
\item $\DCAT(\QCOH_G(X)) = \DQCOH([X/G])$ is compactly generated; and
\item for every $G$-invariant open subset $U \subseteq X$, there exists a
  compact perfect $G$-equivariant complex with support exactly $X\setminus U$.
\end{enumerate}
\end{corollary}
\begin{proof}
\itemref{C:cg:res-prop} is Example~\ref{ex:res-prop}\itemref{ex:item:linear},
\itemref{C:cg:qff} is by construction, and \itemref{C:cg:s-global} follows from
\itemref{C:cg:qff} and~\cite[Prop.~2.8(iii)]{rydh-2009}. Under the additional
assumption on $k$ and $G$, we have that $BG'$ is
concentrated~\cite[Thm.~B]{hallj_dary_alg_groups_classifying}, hence so is
$[W'/G']$ and \itemref{C:cg:conc-crisp} follows from
Proposition~\ref{prop:perf_gen_res_property}.
It follows that $[X/G]$ is $\aleph_0$-crisp by Theorem~\ref{thm:main_list}.
We always have that $\QCOH_G(X)=\QCOH([X/G])$ and we have that
$\DCAT(\QCOH([X/G]))=\DQCOH([X/G])$ since
$\DQCOH([X/G])$ is compactly generated~\cite{hallj_neeman_dary_no_compacts}.
\end{proof}
\begin{example}[Brauer groups]\label{ex:Br=Br'}
Let $X$ be a quasi-compact algebraic stack with quasi-finite and separated
diagonal.  Let $\alpha\in H^2(X,\mathbb{G}_m)$ be an element of the (bigger)
cohomological Brauer group and let
$\mathscript{X}$ denote the $\mathbb{G}_m$-gerbe corresponding to $\alpha$.
Since $X$ has quasi-finite diagonal, there exists a quasi-finite flat
presentation $p\colon X'\to X$ such that $X'$ is affine and
$p^*\alpha=0$. Indeed, recall that $X$ admits a quasi-finite flat presentation
$q\colon U\to X$ by an affine scheme \cite[Thm.~7.1]{rydh-2009} and then we may
trivialize $q^*\alpha$ by a further surjective \'etale morphism of schemes. In
particular, $\mathscript{X}\times_X X'=X'\times B\mathbb{G}_m$ has the
resolution property and is cohomological affine. It follows that
$\DQCOH(\mathscript{X}\times_X X')$ and $\DQCOH(\mathscript{X})$ have
countable sets of compact generators by
Proposition~\ref{prop:perf_gen_res_property} and Theorem~\ref{thm:main_list}.

Let $\mathscript{D}=\REP[\mathrm{fp}]{X}$ and let
$\mathcal{T}=\DCAT(\QCOH^\alpha(-))$ be the presheaf of derived categories of
$\alpha$-twisted sheaves. Then the conditions of
Theorem~\ref{thm:general_qffdesc_cpt_gen} are satisfied. Indeed, there is a
canonical decomposition $\DQCOH(\mathscript{X}\times_X
T)=\DCAT(\QCOH(\mathscript{X}\times_X T))=\bigoplus_{m\in
  \Z}\DCAT(\QCOH^{m\alpha}(T))$ which is respected by pullbacks\footnote{Here
  we have tacitly used~\cite{hallj_neeman_dary_no_compacts} to identify
  $\DQCOH(\mathscript{X}\times_X
T)=\DCAT(\QCOH(\mathscript{X}\times_X T))$, which holds since
$\DQCOH(\mathscript{X}\times_X T)$ is compactly generated.
In general, one could \emph{define}
$\DQCOH^\alpha(X)$ as the degree one part of $\DQCOH(\mathscript{X})$.}. Since
$\DQCOH(\mathscript{X}\times_X -)$ satisfies conditions \itemref{thm:general_qffdesc_cpt_gen:item:1}--\itemref{thm:general_qffdesc_cpt_gen:item:5} of
Theorem~\ref{thm:general_qffdesc_cpt_gen}, so does $\mathcal{T}$.

Thus, $\DCAT(\QCOH^\alpha(X'))=\DCAT(\QCOH(X'))=\DQCOH(X')$ is
compactly generated by $1$ object with supports. It follows that
$\DCAT(\QCOH^\alpha(X))$ is compactly generated by $1$ object with supports
(Theorem~\ref{thm:general_qffdesc_cpt_gen}). The endomorphism algebra of this
object, in a dg-enhancement of $\DCAT(\QCOH^\alpha(X))$,
is a derived Azumaya algebra~\cite{MR2957304}.

Alternatively, one could argue as follows. Take the degree $1$ part of the
compact generators of $\DQCOH(\mathscript{X})$. This gives a countable
generating set $\{P_i\}$ of compact objects in $\DQCOH^\alpha(X)$. For
sufficiently large $n>0$, the direct sum $P:=\bigoplus_{i=1}^n P_i$ gives a
compact object that locally generates $\DQCOH^\alpha(X)$. Indeed, since
$\DQCOH^\alpha(X')=\DQCOH(X')$ and $X'$ is affine, it is sufficient that
$\supph(p^*P)=|X'|$ (Lemma~\ref{L:local-generator}). This compact local
generator $P$ is enough to produce a derived Azumaya
algebra~\cite[Prop.~4.6]{MR2957304}.

This latter argument also works for any quasi-compact and quasi-separated
algebraic stack $X$ such that $\DQCOH(\mathscript{X})$ has a set of compact
generators. In this case, take $p\colon X'\to X$ as a smooth presentation by an
affine scheme such that $p^*\alpha=0$. By~\cite[Thm.~2.26]{AHR_lunafield},
$\DQCOH(\mathscript{X})$ is compactly generated for any algebraic stack $X$ of
finite type over a field, with affine diagonal, and linearly reductive
stabilizers at closed points. We can thus conclude that for such $X$, every
cohomological Brauer class comes from a derived Azumaya algebra.
\end{example}
\begin{example}[Sheaves of linear categories on derived stacks]\label{ex:derived}
Let $(X,\Orb_X)$ be a derived (or spectral) Deligne--Mumford stack. The
$0$-truncation $(X,\pi_0\Orb_X)$ is an ordinary Deligne--Mumford stack with the
same underlying topos $X$. In fact, even for a non-connective
$\mathbb{E}_\infty$-algebra $A$, the category of \'etale $A$-algebras is equivalent to
the category of \'etale $\pi_0
A$-algebras~\cite[Thm.~7.5.0.6]{lurie_highalg}.

Let $F\in \mathrm{QStk}(X)$ be a quasi-coherent stack on $X$ \cite[\S 8]{dag11}, e.g.,
$F=\QCOH(X)$. For every object $U$ in the small \'etale topos of $X$, this
gives an $\Orb_X(U)$-linear $\infty$-category $F(U)$. Let $hF$ be the
presheaf of triangulated categories that assigns to each \'etale $U \to X$ the homotopy category of $F(U)$. Compact generation of $F(U)$ is a statement about
its homotopy category \cite[Rem.~1.4.4.3]{lurie_highalg}. Moreover,
since the conditions \itemref{thm:general_qffdesc_cpt_gen:item:1}--\itemref{thm:general_qffdesc_cpt_gen:item:5} of Theorem \ref{thm:general_qffdesc_cpt_gen} 
can all be verified \'etale-locally (similarly to To\"en's
locally presentable dg-categories \cite{MR2957304}), it follows that
Theorem \ref{thm:general_qffdesc_cpt_gen} can be applied to $hF$ to deduce
compact generation of $F(X)$ from local compact generation of $F$.
\end{example}
\appendix
\section{Generators from above}
Let $w\colon W \to X$ be an \'etale, separated, finitely presented and representable morphism of algebraic stacks. Define 
\[
\mathcal{P}'(w) = \bigl\{ \RDERF \QCPSH{w}\RDERF \QCPSH{u}P \suchthat (u\colon U \to W)  \in \REP[\mathrm{fp},\et,\mathrm{sep}]{W}\, \mbox{and}\, P \in \DQCOH(U)^c\bigr\}
\]
and let $\mathcal{P}(w) \subseteq \DQCOH(X)$ be the smallest thick subcategory of $\DQCOH(X)$ containing $\mathcal{P}'(w)$. The following Lemma---requested by 
Neeman---is a natural generalization to algebraic stacks of a result that has found applications to Grothendieck duality for schemes \cite[Lem.~3.1]{2014arXiv1406.7599N}. 
\begin{lemma}
  Let $w\colon W \to X$ be an \'etale, separated, finitely presented, representable and surjective
  morphism of algebraic stacks. If $X$ is quasi-compact and quasi-separated and 
  $\DQCOH(X)$ is compactly generated, then
  \[
  \DQCOH(X)^c \subseteq \mathcal{P}(w).
  \]
\end{lemma}
\begin{proof}
  Let $\mathcal{C} \subseteq \REP[\mathrm{fp},\et,\mathrm{sep}]{X}$ be the subcategory with 
  objects those $V$ such that $\DQCOH(V)^c \subseteq 
  \mathcal{P}(W\times_X V \to V)$. Note that if $W\times_X V \to V$ admits a section, 
  then it is clear that $V \in \mathcal{C}$. In particular, it follows immediately that $W\in 
  \mathcal{C}$. Now we will prove that $X \in \mathcal{C}$ 
  using~\cite[Thm.~6.1]{MR2774654}. To do this, we need to verify the following 
  three conditions for a morphism $v\colon V' \to V$ in
  $\REP[\mathrm{fp},\et,\mathrm{sep}]{X}$. 
\begin{itemize}
\item[(D1)] If $V \in \mathcal{C}$, then $V' \in \mathcal{C}$; 
\item[(D2)] if $V' \in \mathcal{C}$ and $v$ is finite and surjective,
  then $V \in \mathcal{C}$; and
\item[(D3)] if $v$ is an \'etale neighborhood of $V\setminus U$,
  where $j\colon U \to V$ is an open immersion in $\REP[\mathrm{fp},\et,\mathrm{sep}]{X}$, 
  and $U$ and $V'$ belong to $\mathcal{C}$, then $V \in \mathcal{C}$.
\end{itemize}
For (D1): Lemma \ref{lem:compactgen-stable-quasiaffine} implies that $\DQCOH(V')^c$ is 
the smallest thick subcategory of $\DQCOH(V')$ containing $\LDERF\QCPBK{v}Q$, where 
$Q\in \DQCOH(V)^c$. A simple argument involving flat base change (Theorem 
\ref{thm:fcd_coprod}\itemref{thm:fcd_coprod:item:fl_bc}) and the preservation of compact objects under concentrated morphisms (Example \ref{ex:pres_cpt_conc}) now shows that $V' \in \mathcal{C}$. 
For (D2): Proposition \ref{prop:adm_finite_dev} implies that $\DQCOH(V)^c$ is the smallest thick subcategory of $\DQCOH(V)$ containing the collection of complexes $\RDERF \QCPSH{v}Q$, where $Q \in \DQCOH(V')^c$. The property now follows from the trivial observation that 
\[
\bigl\{\RDERF \QCPSH{v}M \suchthat M \in \mathcal{P}(W\times_X V' \to V')\bigr\} \subseteq \mathcal{P}(W\times_X V \to V).
\]
For (D3): we note that the Mayer--Vietoris triangle of Lemma \ref{lem:mv_triangle}\itemref{item:mv_triangle:MV} implies that if $P \in \DQCOH(V)^c$, then there is a distinguished triangle:
\[
P \to \RDERF \QCPSH{j}\LDERF \QCPBK{j}P \oplus \RDERF \QCPSH{f}\LDERF \QCPBK{f}P \to \RDERF\QCPSH{f}\LDERF \QCPBK{f}\RDERF\QCPSH{j}\LDERF \QCPBK{j}P \to P[1].
\]
Since $\LDERF \QCPBK{f}\RDERF \QCPSH{j} \simeq \RDERF \QCPSH{j'} \LDERF \QCPBK{f'}$, where $j' \colon U'=U\times_V V' \to V'$ and $f'\colon U' \to U$ are the projections,
it follows immediately that $\DQCOH(V)^c$ is contained in the smallest thick subcategory of $\DQCOH(V)$ containing the objects:
\begin{itemize}
\item $\bigl\{ \RDERF \QCPSH{j}M \suchthat M \in \mathcal{P}(W\times_X U \to U)\bigr\}$,
\item $\bigl\{ \RDERF \QCPSH{f}M \suchthat M \in \mathcal{P}(W\times_X V' \to V')\bigr\}$, and
\item $\bigl\{ \RDERF \QCPSH{(f\circ j')}M \suchthat M \in \mathcal{P}(W\times_X U' \to U')\bigr\}$.
\end{itemize}
But all the objects above are contained in $\mathcal{P}(W\times_X V \to V)$ and the claim follows. Since $W \in \mathcal{C}$, we may now conclude that $X \in \mathcal{C}$.
\end{proof}

\bibliography{references}
\bibliographystyle{bibstyle}
\end{document}